\documentclass[leqno,12pt]{amsart}
\usepackage{amssymb}
\usepackage{amsmath}
\usepackage{enumerate}
\usepackage{amsfonts}
\usepackage{hyperref}
\usepackage{mathrsfs}
\usepackage{tikz}

\usepackage[headheight=18pt, top=25mm, bottom=25mm, left=20mm, right=20mm]{geometry}

\hypersetup{
    colorlinks=true,
    linkcolor= blue,
    citecolor =cyan,
    urlcolor = teal,
}

\newtheorem{theorem}{Theorem}[section]
\newtheorem{corollary}[theorem]{Corollary}
\newtheorem{lemma}[theorem]{Lemma}
\newtheorem{proposition}[theorem]{Proposition}
\newtheorem{remark}[theorem]{Remark}

\numberwithin{equation}{section}

\begin{document}
\title[$L^p$ estimates for Dunkl--Schr\"odinger operators]{$L^p$ estimates for Dunkl--Schr\"odinger operators}

\author[Agnieszka Hejna-Łyżwa]{Agnieszka Hejna-Łyżwa}
\author[Suman Mukherjee]{Suman Mukherjee}



\begin{abstract}
We study the Dunkl--Schr\"odinger operator $L=-\Delta_k+V$
on $\mathbb R^N$, $N>2$, where $V$ is a nonnegative potential
in the reverse H\"older class ${\rm RH}^q(dw)$,
$q>\mathbf N/2$, and $\mathbf N$ is the homogeneous dimension.
We establish decay estimates for the resolvent kernel
in terms of the critical radius and the orbit distance.
For reflection-invariant potentials, we obtain $L^p(dw)$
bounds for $VL^{-1}$, second-order Dunkl derivatives of
$L^{-1}$, and the Riesz transform $\nabla_kL^{-1/2}$,
together with fractional potential and mixed estimates.
The proofs combine local estimates for weak solutions,
the Dunkl Fefferman--Phong inequality, and kernel estimates
adapted to the Euclidean and orbit distances.
\end{abstract}

\address{A. Hejna-Łyżwa, Uniwersytet Wroc\l awski,
Instytut Matematyczny,
Pl. Grunwaldzki 2,
50-384 Wroc\l aw,
Poland}
\email{agnieszka.hejna@math.uni.wroc.pl}

\address{S. Mukherjee, Department of Mathematics, Indian Institute of Technology Bombay, Powai, Mumbai, 400076, India}
\email{sumanmukherjee822@gmail.com}

\enlargethispage{0.37cm}
\maketitle
\tableofcontents
\thispagestyle{empty}

\section{Introduction}

We study $L^p$ estimates for Schr\"odinger operators
\[
 L=-\Delta_k+V
\]
in the rational Dunkl setting on $\mathbb R^N$.
Here $\Delta_k$ is the Dunkl Laplacian associated with
a normalized root system $R$, its finite reflection
group $G$, and a nonnegative $G$-invariant multiplicity
function $k$. The underlying measure is
\begin{equation*}
 dw(\mathbf x)
 =\prod_{\alpha\in R}
 |\langle\mathbf x,\alpha\rangle|^{k(\alpha)}
 \,d\mathbf x,
\end{equation*}
and
$\mathbf N=N+\sum_{\alpha\in R}k(\alpha)$
denotes the homogeneous dimension.
We assume that $N>2$ and that the potential $V$ is
nonnegative, not identically zero, and $V\in L^2_{\text{loc}}(dw)$.
Moreover, $V$ belongs to the reverse H\"older class
${\rm RH}^q(dw)$ for some $q>\mathbf N/2$; that is,
\begin{equation}\label{eq:RevHol}
 \left(
 \frac1{w(B)}\int_B V(\mathbf y)^q\,dw(\mathbf y)
 \right)^{1/q}
 \leq
 \frac{C_{\rm RH}}{w(B)}
 \int_B V(\mathbf y)\,dw(\mathbf y)
\end{equation}
for every Euclidean ball $B\subset\mathbb R^N$.
Our aim is to establish $L^p(dw)$ bounds for the
Riesz transforms $\nabla_kL^{-1/2}$ and for operators
involving the Green potential $L^{-1}$, Dunkl
derivatives, and multiplication by powers of $V$.
These estimates provide a counterpart of the
classical Schr\"odinger theory developed by
Shen~\cite{Shen}, in a setting where the differential
operators also contain reflection terms.

For classical Schr\"odinger operators
$\mathscr L=-\Delta_{\mathrm{eucl}}+V$ with nonnegative
reverse H\"older potentials, the critical function
introduced in the work of Fefferman~\cite{Fefferman}
and the Fefferman--Phong inequality provide basic
tools for studying spectral properties and regularity.
Shen~\cite{Shen2,Shen} used this approach to obtain
fundamental-solution estimates and $L^p$ bounds for
Riesz transforms and related operators.

In the rational Dunkl setting, in~\cite{Hejna_Schrodinger}, the following Fefferman--Phong inequality
was established:
\[
 \int_{\mathbb R^N}m(\mathbf x)^2|h(\mathbf x)|^2\,dw(\mathbf x)
 \leq C\left(
 \sum_{j=1}^N\|T_jh\|_{L^2(dw)}^2
 +\|V^{1/2}h\|_{L^2(dw)}^2
 \right)
\]
for $h$ in the quadratic-form domain, and
an atomic characterization of the associated Hardy
space was proved. In~\cite{Hejna_box}, the author obtained a counterpart
of Fefferman's box formula, estimating the eigenvalue
counting function in terms of cubes meeting sublevel
sets of the critical function $m$.
Building on these results, the present paper establishes
resolvent-gradient estimates and $L^p(dw)$ bounds for
Dunkl--Schr\"odinger Riesz transforms and related
operators involving the potential.

Throughout the paper, we assume that $N>2$ and that
$V\not\equiv0$, $V\in L^2_{\text{loc}}(dw)$, is a nonnegative potential
belonging to ${\rm RH}^q(dw)$ for some finite
$q>\mathbf N/2$. We write
\[
 L=-\Delta_k+V,
 \qquad
 \nabla_k=(T_1,\ldots,T_N),
 \qquad
 \nabla_k^2=(T_iT_j)_{1\leq i,j\leq N}.
\]
For vector-valued functions, the $L^p(dw)$ norm is
defined using the Euclidean norm of the components.
For second-order derivatives, we use the convention
\[
 \|\nabla_k^2u\|_{L^p(dw)}
 =\sum_{i,j=1}^N\|T_iT_ju\|_{L^p(dw)}.
\]
For a vector field $\mathbf f=(f_1,\ldots,f_N)$, let
\[
 \operatorname{div}_k\mathbf f=\sum_{j=1}^N T_jf_j.
\]
Negative powers of $L$ are initially interpreted
on their natural domains. At zero spectral parameter,
$L^{-1}$ denotes the Green potential.
The estimates below concern the bounded extensions
of the indicated compositions.

The operator estimates below assume that $V$ is $G$-invariant.
This natural symmetry condition includes radial potentials
and implies $m(g\mathbf x)=m(\mathbf x)$.
It enables arbitrary polynomial decay of the resolvent kernel
on the critical scale, which is used to control the far-field
terms in our proofs.

Our main results are as follows.
\begin{enumerate}
 \item
 \textbf{Potential and second-order derivative estimates.}
 Theorems~\ref{thm:V_boundedness}
 and~\ref{thm:Riesz_boundedness} establish
 \[
  \|VL^{-1}f\|_{L^p(dw)}
  \leq C_p\|f\|_{L^p(dw)},
  \qquad 1\leq p\leq q,
 \]
 and
 \[
  \|\nabla_k^2L^{-1}f\|_{L^p(dw)}
  \leq C_p\|f\|_{L^p(dw)},
  \qquad 1<p\leq q.
 \]

 \item
 \textbf{First-order Riesz transforms.}
 Put
 \[
  p_0=
  \begin{cases}
   \displaystyle\frac{\mathbf Nq}{\mathbf N-q},
       & \mathbf N/2<q<\mathbf N,\\[2mm]
   \infty, & q\geq\mathbf N.
  \end{cases}
 \]
 Proposition~\ref{prop:sharp_Riesz_reduction}
 gives
 \[
  \|\nabla_kL^{-1/2}f\|_{L^p(dw)}
  \leq C_p\|f\|_{L^p(dw)},
  \qquad 1<p<p_0,
 \]
 when $q<\mathbf N$. For $q\geq\mathbf N$, the
 same conclusion holds for every $1<p<\infty$
 by applying that proposition with reverse
 H\"older exponents below and arbitrarily close
 to $\mathbf N$, as explained in the proof of
 Theorem~\ref{thm:conditional_sharp_mixed_estimates}.

 \item
 \textbf{Gradient estimates with a potential factor.}
 Proposition~\ref{prop:weighted_gradient_reduction}
 proves
 \[
  \|V^{1/2}\nabla_kL^{-1}f\|_{L^p(dw)}
  \leq C_p\|f\|_{L^p(dw)},
  \qquad 1<p<p_1,
 \]
 where
 \[
  p_1=
  \begin{cases}
   \displaystyle
   \left(\frac{3}{2q}-\frac1{\mathbf N}\right)^{-1},
       & \mathbf N/2<q<\mathbf N,\\[2mm]
   2q, & q\geq\mathbf N.
  \end{cases}
 \]

 \item
 \textbf{Fractional potential estimates and mixed operators.}
 Theorem~\ref{thm:fractional_potential_boundedness}
 gives, for $0<\theta<1$,
 \[
  \|V^\theta L^{-\theta}f\|_{L^p(dw)}
  \leq C_{p,\theta}\|f\|_{L^p(dw)},
  \qquad 1<p<q/\theta.
 \]
 Combining this result with the first-order
 Riesz transform estimates and duality,
 Theorem~\ref{thm:conditional_sharp_mixed_estimates}
 yields
 \[
  \|L^{-1/2}\operatorname{div}_k\mathbf f\|_{L^p(dw)}
  \leq C_p\|\mathbf f\|_{L^p(dw)},
  \qquad p_0'<p<\infty,
 \]
 \[
  \|\nabla_kL^{-1}\operatorname{div}_k\mathbf f\|_{L^p(dw)}
  \leq C_p\|\mathbf f\|_{L^p(dw)},
  \qquad p_0'<p<p_0,
 \]
 and
 \[
  \|V^{1/2}L^{-1}\operatorname{div}_k\mathbf f\|_{L^p(dw)}
  \leq C_p\|\mathbf f\|_{L^p(dw)},
  \qquad p_0'<p<2q.
 \]
 Here $p_0'$ is the conjugate exponent of $p_0$,
 with $p_0'=1$ when $p_0=\infty$.
\end{enumerate}

The main difficulties arise from the reflection terms
in the Dunkl operators and the Dunkl Laplacian
(see \eqref{eq:T_def} and
\eqref{eq:Dunkl_Laplacian}). Their nonlocality requires
localization on $G$-invariant sets, while mean-value
estimates involve orbit neighbourhoods. Although $dw$
is doubling, ball volumes depend on the position of
their centres, and kernel estimates involve both
Euclidean and orbit distances. This distinction is
particularly relevant to truncations of singular
integrals. Moreover, localization in integration-by-parts
arguments must account for the reflection terms
in the product rule. Since reverse H\"older potentials
need not yield $C^2$ solutions, we formulate the local
estimates for weak solutions.

These estimates provide regularity bounds for the canonical
Green solutions of Dunkl--Schr\"odinger equations with rough
potentials, including equations with right-hand sides in
Dunkl divergence form. In particular, they control the
potential term and the first- and second-order Dunkl
derivatives in the indicated $L^p(dw)$ ranges.
The Riesz transform estimates also provide tools for further
developments in Hardy space and singular integral theory
associated with these operators, while the fractional
potential bounds quantify the relation between multiplication
by $V$ and fractional powers of $L$.
The class of potentials considered here includes
$V(\mathbf x)=\|\mathbf x\|^2$, so the results apply,
in particular, to the Dunkl harmonic oscillator.

Our main tools are estimates for the Dunkl heat kernel
and for translations of radial functions and their
derivatives~\cite{ADzH,DzH1}, together with the
Fefferman--Phong inequality~\cite{Hejna_Schrodinger}.
We combine these with methods of analysis on spaces
of homogeneous type and the Dunkl singular-integral
theory developed in
\cite{ADzH,DzH1,HanLeeLiWick,TanHanHanLeeLi,TanHanLi}.
Local gradient estimates and annular kernel bounds
then lead to the required maximal-function estimates.
Spectral regularization justifies the operator
decompositions without assuming pointwise convergence
of orbit-distance truncations.

\section{Preliminaries}\label{sec:preliminaries}

\subsection{Dunkl geometry and volume estimates}
In this section we present basic facts concerning the theory of the Dunkl operators.  For details we refer the reader to~\cite{Dunkl},~\cite{deJeu},~\cite{Roesler3}, and~\cite{Roesler-Voit}. 

We consider the Euclidean space $\mathbb R^N$ with the scalar product $\langle\mathbf x,\mathbf y\rangle=\sum_{j=1}^N x_jy_j
$, where $\mathbf x=(x_1,...,x_N)$, $\mathbf y=(y_1,...,y_N)$, and the norm $\| \mathbf x\|^2=\langle \mathbf x,\mathbf x\rangle$. For a nonzero vector $\alpha\in\mathbb R^N$,  the reflection $\sigma_\alpha$ with respect to the hyperplane $\alpha^\perp$ orthogonal to $\alpha$ is given by
\begin{align*}
\sigma_\alpha (\mathbf x)=\mathbf x-2\frac{\langle \mathbf x,\alpha\rangle}{\| \alpha\| ^2}\alpha.
\end{align*}
In this paper we fix a normalized root system in $\mathbb R^N$, that is, a finite set  $R\subset \mathbb R^N\setminus\{0\}$ such that $R \cap \alpha \mathbb{R} = \{\pm \alpha\}$,  $\sigma_\alpha (R)=R$, and $\|\alpha\|=\sqrt{2}$ for all $\alpha\in R$. The finite group $G$ generated by the reflections $\sigma_\alpha$, $\alpha \in R$ is called the {\it Weyl group} ({\it reflection group}) of the root system. A~{\textit{multiplicity function}} is a $G$-invariant function $k:R\to\mathbb C$ which will be fixed and $\geq 0$  throughout this paper. 
 Let
\begin{equation}\label{eq:measure_formula}
dw(\mathbf x)=\prod_{\alpha\in R}|\langle \mathbf x,\alpha\rangle|^{k(\alpha)}\, d\mathbf x
\end{equation} 
be  the associated measure in $\mathbb R^N$, where, here and subsequently, $d\mathbf x$ stands for the Lebesgue measure in $\mathbb R^N$.
We denote by 
\begin{equation}\label{eq:homo}
\mathbf{N}=N+\sum_{\alpha \in R} k(\alpha)
\end{equation}
the homogeneous dimension of the system. Clearly, 
\begin{align*} w(B(t\mathbf x, tr))=t^{\mathbf N}w(B(\mathbf x,r)) \ \ \text{\rm for all } \mathbf x\in\mathbb R^N, \ t,r>0,   
\end{align*}
 where $B(\mathbf x, r)=\{\mathbf y\in\mathbb R^N: \|\mathbf y-\mathbf x\|<r\}$. Moreover, 
\begin{align*}
\int_{\mathbb R^N} f(\mathbf x)\, dw(\mathbf x)=\int_{\mathbb R^N} t^{-\mathbf N} f(\mathbf x\slash t)\, dw(\mathbf x)\ \ \text{for} \ f\in L^1(dw)  \   \text{\rm and} \  t>0.
\end{align*}
Observe that there is a constant $C>0$ such that 
\begin{equation}\label{eq:balls_asymp} 
C^{-1}w(B(\mathbf x,r))\leq  r^{N}\prod_{\alpha \in R} (|\langle \mathbf x,\alpha\rangle |+r)^{k(\alpha)}\leq C w(B(\mathbf x,r)),
\end{equation}
so $dw(\mathbf x)$ is doubling, that is, there is a constant $C>0$ such that
\begin{equation}\label{eq:doubling} w(B(\mathbf x,2r))\leq C w(B(\mathbf x,r)) \ \ \text{ for all } \mathbf x\in\mathbb R^N, \ r>0.
\end{equation}
Moreover, there exists a constant $C\ge1$ such that,
for every $\mathbf{x}\in\mathbb{R}^N$ and for every $r_2\ge r_1>0$,
\begin{equation}\label{eq:growth}
C^{-1}\Big(\frac{r_2}{r_1}\Big)^{N}\leq\frac{{w}(B(\mathbf{x},r_2))}{{w}(B(\mathbf{x},r_1))}\leq C \Big(\frac{r_2}{r_1}\Big)^{\mathbf{N}}.
\end{equation}

For a measurable subset $A$ of $\mathbb{R}^N$ we define 
\begin{equation}\label{eq:orbit_of_A}
    \mathcal{O}(A)=\{g(\mathbf{x})\,:\, \mathbf{x} \in A, \, g \in G\}.
\end{equation}
Clearly, by~\eqref{eq:balls_asymp}, for all $\mathbf{x} \in \mathbb{R}^N$ and $r>0$ we get
\begin{equation}\label{eq:ball_orbit_compare}
    w(\mathcal{O}(B(\mathbf{x},r))) \leq |G|w(B(\mathbf{x},r)).
\end{equation}

Let 
\begin{equation}\label{eq:distance_of_orbits}
    d(\mathbf x,\mathbf y)=\min_{\sigma\in G}\| \sigma(\mathbf x)-\mathbf y\|    
\end{equation}\index{distance of orbits}
be the distance of the orbit of $\mathbf{x}$ to the orbit of $\mathbf{y}$.

\subsection{Dunkl operators, transform, and translations}

For $\xi \in \mathbb{R}^N$, the {\it Dunkl operators} $T_\xi$  are the following $k$-deformations of the directional derivatives $\partial_\xi$ by a  difference operator:
\begin{equation}\label{eq:T_def}
     T_\xi f(\mathbf x)= \partial_\xi f(\mathbf x) + \sum_{\alpha\in R} \frac{k(\alpha)}{2}\langle\alpha ,\xi\rangle\frac{f(\mathbf x)-f(\sigma_\alpha(\mathbf{x}))}{\langle \alpha,\mathbf x\rangle}.
\end{equation}
The Dunkl operators $T_{\xi}$, which were introduced in~\cite{Dunkl}, commute and are skew-symmetric with respect to the $G$-invariant measure $dw$. If $\xi \in \mathbb{R}^N$, $f,g \in C^1(\mathbb{R}^N)$ and the function $g$ is $G$-invariant, then we have
\begin{equation}\label{eq:G_inv_Leibniz}
    T_{\xi}(fg)=f(T_{\xi}g)+g(T_{\xi}f).
\end{equation}

For a general smooth multiplier $\varphi$, the product rule
takes the following form:
\begin{equation}\label{eq:general_Leibniz}
\begin{split}
 T_j(\varphi f)(\mathbf{x})
 &=
 \varphi(\mathbf{x})T_jf(\mathbf{x})
 +f(\mathbf{x})\partial_j\varphi(\mathbf{x})
 +
 \sum_{\alpha\in R}\frac{k(\alpha)}2\alpha_j
 f(\sigma_\alpha\mathbf{x})
 \frac{\varphi(\mathbf{x})
       -\varphi(\sigma_\alpha(\mathbf{x}))}
      {\langle\alpha, \mathbf{x}\rangle}.
\end{split}
\end{equation}

If $\varphi\in C_c^\infty(\mathbb R^N)$, the quotient
appearing in~\eqref{eq:general_Leibniz} extends smoothly
across $\alpha^\perp$. Moreover, the mean value theorem
gives
\begin{equation}\label{eq:multiplier_reflection_quotient}
 \left|
 \frac{\varphi(\mathbf{x})
       -\varphi(\sigma_\alpha\mathbf{x})}
      {\langle\alpha,\mathbf{x}\rangle}
 \right|
 \leq
 \frac{2}{\|\alpha\|}
 \|\nabla\varphi\|_{L^\infty}.
\end{equation}

Formula~\eqref{eq:general_Leibniz} also holds for weak
Dunkl derivatives. More precisely, let $\Omega$ be an
open $G$-invariant set, let
$f,T_jf\in L^2_{\mathrm{loc}}(\Omega,dw)$, and let
$\varphi\in C_c^\infty(\Omega)$.
Then~\eqref{eq:general_Leibniz} holds in the
distributional sense in $\Omega$, and its right-hand
side belongs to $L^2(\Omega,dw)$.
This weak version follows by testing against a smooth
compactly supported function, using the skew-symmetry
of $T_j$, and changing variables
$\mathbf{x}\mapsto\sigma_\alpha\mathbf{x}$ in the
reflection terms.
The $L^2$ assertion follows
from~\eqref{eq:multiplier_reflection_quotient} and
the $G$-invariance of $dw$.

For fixed $\mathbf y\in\mathbb R^N$ the {\it Dunkl kernel} $E(\mathbf x,\mathbf y)$ is the unique analytic solution to the system
\begin{equation}\label{eq:Dunkl_kernel_definition}
    T_\xi f=\langle \xi,\mathbf y\rangle f, \ \ f(0)=1.
\end{equation}
The function $E(\mathbf x ,\mathbf y)$, which generalizes the exponential  function $e^{\langle \mathbf x,\mathbf y\rangle}$, has the unique extension to a holomorphic function on $\mathbb C^N\times \mathbb C^N$. Moreover, it satisfies $E(\mathbf{x},\mathbf{y})=E(\mathbf{y},\mathbf{x})$ for all $\mathbf{x},\mathbf{y} \in \mathbb{C}^N$.

Let $\{e_j\}_{1 \leq j \leq N}$ denote the canonical orthonormal basis in $\mathbb R^N$ and let $T_j=T_{e_j}$. 
In our further consideration we shall need the following lemma.
\begin{lemma}
For all $\mathbf{x} \in \mathbb{R}^N$, $\mathbf{z} \in \mathbb{C}^N$ and $\nu \in \mathbb{N}_0^{N}$ we have
$$|\partial^{\nu}_{\mathbf{z}}E(\mathbf{x},\mathbf{z})| \leq \|\mathbf{x}\|^{|\nu|}\exp(\|\mathbf{x}\|\|{\rm Re \;}\mathbf{z}\|).$$
In particular, 
\begin{align*} | E(i\xi, \mathbf x)|\leq 1 \quad \text{ for all } \xi,\mathbf x\in \mathbb R^N.
\end{align*}
\end{lemma}
\begin{proof}
See~\cite[Corollary 5.3]{Roesle99}.
\end{proof}

\begin{corollary}\label{coro:Roesler}
There is a constant $C>0$ such that for all $\mathbf{x},\xi \in \mathbb{R}^N$ we have
\begin{equation}
|E(i\xi,\mathbf{x})-1| \leq C\|\mathbf{x}\|\|\xi\|.
\end{equation}
\end{corollary}

The \textit{Dunkl transform}
  \begin{align*}\mathcal F f(\xi)=c_k^{-1}\int_{\mathbb R^N} E(-i\xi, \mathbf x)f(\mathbf x)\, dw(\mathbf x),
  \end{align*}
  where
  $$c_k=\int_{\mathbb{R}^N}e^{-\frac{\|\mathbf{x}\|^2}{2}}\,dw(\mathbf{x})>0,$$
   originally defined for $f\in L^1(dw)$, is an isometry on $L^2(dw)$, i.e.,
   \begin{equation}\label{eq:Plancherel}
       \|f\|_{L^2(dw)}=\|\mathcal{F}f\|_{L^2(dw)} \text{ for all }f \in L^2(dw).
   \end{equation}
   The Dunkl transform was introduced by C.F. Dunkl in
\cite{Dunkl3}, where the Plancherel theorem was also established.

The Dunkl transform preserves the Schwartz class of functions $\mathcal S(\mathbb R^N)$ (see \cite{deJeu}). Its inverse $\mathcal F^{-1}$ has the form
  \begin{align*} \mathcal F^{-1} g(x)=c_k^{-1}\int_{\mathbb R^N} E(i\xi, \mathbf x)g(\xi)\, dw(\xi).
  \end{align*}
Moreover,
\begin{equation}\label{eq:der_transform}
    \mathcal{F}(T_j f)(\xi)=i\xi_j\mathcal{F}f(\xi).
\end{equation}
The {\it Dunkl translation\/} $\tau_{\mathbf{x}}f$ of a function $f\in\mathcal{S}(\mathbb{R}^N)$ by $\mathbf{x}\in\mathbb{R}^N$ is defined by
\begin{align*}
\tau_{\mathbf{x}} f(\mathbf{y})=c_k^{-1} \int_{\mathbb{R}^N}{E}(i\xi,\mathbf{x})\,{E}(i\xi,\mathbf{y})\,\mathcal{F}f(\xi)\,{dw}(\xi).
\end{align*}
  It is a contraction on $L^2(dw)$, however it is an open  problem  if the Dunkl translations are bounded operators on $L^p(dw)$ for $p\ne 2$.
  
  {The \textit{Dunkl convolution\/} $f*g$ of two reasonable functions (for instance Schwartz functions) is defined by
$$
(f*g)(\mathbf{x})=c_k\,\mathcal{F}^{-1}[(\mathcal{F}f)(\mathcal{F}g)](\mathbf{x})=\int_{\mathbb{R}^N}(\mathcal{F}f)(\xi)\,(\mathcal{F}g)(\xi)\,E(\mathbf{x},i\xi)\,dw(\xi) \text{ for }\mathbf{x}\in\mathbb{R}^N,
$$
or, equivalently, by}
\begin{align*}
  {(f{*}g)(\mathbf{x})=\int_{\mathbb{R}^N}f(\mathbf{y})\,\tau_{\mathbf{x}}g(-\mathbf{y})\,{dw}(\mathbf{y})=\int_{\mathbb R^N} f(\mathbf y)g(\mathbf x,\mathbf y) \,dw(\mathbf{y}) \text{ for all } \mathbf{x}\in\mathbb{R}^N},  
\end{align*}
where, here and subsequently, $g(\mathbf x,\mathbf y)=\tau_{\mathbf x}g(-\mathbf y)$. 

The following formula for the Dunkl translations of radial function was proved by R\"osler in~\cite{Roesler2003}.

\begin{theorem}\label{teo:roesler_formula}\index{R\"osler's formula}
For any $\mathbf{x} \in \mathbb{R}^{N}$  there exists a unique probability Borel measure $\mu_{\mathbf{x}}$ such that for all $f \in C^{\infty}(\mathbb{R}^N)$ such that $f(\mathbf{y})=\widetilde{f}(\|\mathbf{y}\|)$ and for all $\mathbf{y} \in \mathbb{R}^N$,  we have
\begin{equation}\label{eq:translation-radial}
\tau_{\mathbf{x}}f(-\mathbf{y})=\int_{\mathbb{R}^N}{(\tilde{f}\circ A)}(\mathbf{x},\mathbf{y},\eta)\,d\mu_{\mathbf{x}}(\eta) \text{ for all }\mathbf{y} \in \mathbb{R}^N,
\end{equation}
where, here and subsequently,
\begin{equation*}
A(\mathbf{x},\mathbf{y},\eta)=\sqrt{{\|}\mathbf{x}{\|}^2+{\|}\mathbf{y}{\|}^2-2\langle \mathbf{y},\eta\rangle}=\sqrt{{\|}\mathbf{x}{\|}^2-{\|}\eta{\|}^2+{\|}\mathbf{y}-\eta{\|}^2}.
\end{equation*}
Recall that the support of $\mu_{\mathbf{x}}$ is contained in the convex hull of $\mathcal{O}(\mathbf{x})$.
\end{theorem}

\subsection{Dunkl Laplacian and Dunkl heat semigroup}\label{sec:laplacian} The {\it Dunkl Laplacian} associated with $R$ and $k$  is the differential-difference operator $\Delta_k=\sum_{j=1}^N T_{j}^2$, which  acts on $C^2(\mathbb{R}^N)$-functions by

\begin{equation}\label{eq:Dunkl_Laplacian}
    \Delta_k f(\mathbf x)=\Delta f(\mathbf x)+\sum_{\alpha\in R} k(\alpha) \delta_\alpha f(\mathbf x),
\end{equation}
\begin{align*}
    \delta_\alpha f(\mathbf x)=\frac{\partial_\alpha f(\mathbf x)}{\langle \alpha , \mathbf x\rangle} - \frac{\|\alpha\|^2}{2} \frac{f(\mathbf x)-f(\sigma_\alpha \mathbf x)}{\langle \alpha, \mathbf x\rangle^2}.
\end{align*}
Obviously, $\mathcal F(\Delta_k f)(\xi)=-\| \xi\|^2\mathcal Ff(\xi)$. The operator $\Delta_k$ is essentially self-adjoint on $L^2(dw)$ (see for instance \cite[Theorem\;3.1]{AH}) and generates the semigroup $H_t$  of linear self-adjoint contractions on $L^2(dw)$. The semigroup has the form
  \begin{equation}\label{eq:heat}
  H_t f(\mathbf x)=\mathcal F^{-1}(e^{-t\|\xi\|^2}\mathcal Ff(\xi))(\mathbf x)=\int_{\mathbb R^N} h_t(\mathbf x,\mathbf y)f(\mathbf y)\, dw(\mathbf y),
  \end{equation}
  where the heat kernel is given by
 \begin{equation}\label{eq:heat_def}
 \begin{split}
 h_t(\mathbf x,\mathbf y)
 &=\tau_{\mathbf x}h_t(-\mathbf y),\\
 h_t(\mathbf x)
 &=c_k^{-1}\mathcal F^{-1}
       (e^{-t\|\cdot\|^2})(\mathbf x)
 =c_k^{-1}(2t)^{-\mathbf N/2}
       e^{-\|\mathbf x\|^2/(4t)}.
 \end{split}
\end{equation}
$h_t(\mathbf x,\mathbf y)$  is a $C^\infty$-function of all the variables $\mathbf x,\mathbf y \in \mathbb{R}^N$, $t>0$, and satisfies \begin{align*} 0<h_t(\mathbf x,\mathbf y)=h_t(\mathbf y,\mathbf x),
  \end{align*}
 \begin{equation} \label{eq:h_integral_1}
    \int_{\mathbb R^N} h_t(\mathbf x,\mathbf y)\, dw(\mathbf y)=1.
 \end{equation}

Let 
$$d(\mathbf x,\mathbf y)=\min_{\sigma\in G}\| \sigma(\mathbf x)-\mathbf y\|$$
be the distance of the orbit of $\mathbf x$ to the orbit of $\mathbf y$. Let us denote
\begin{equation}\label{eq:mathcal_G}
     \mathcal G_t(\mathbf x,\mathbf y)=\Big(\max (w(B(\mathbf x,\sqrt{t})),w(B(\mathbf y, \sqrt{t})))\Big)^{-1}\exp\Big(-\frac {d(\mathbf x,\mathbf y)^2}{t}\Big).
 \end{equation}
  We shall need the following estimates for $h_t(\mathbf x,\mathbf y)$ - their two step proof, which is based on R\"osler's formula for the Dunkl  translations of  radial functions (see~\cite{Roesler2003}), can be found in~\cite[Theorem 4.1]{ADzH} and~\cite[Theorem 3.1]{DzH1}. 
\begin{theorem}\label{theorem:heat}
There are constants $C,c>0$ such that for all $\mathbf{x},\mathbf{y} \in \mathbb{R}^N$ and $t>0$ we have
\begin{equation}\label{eq:Gauss}
h_t(\mathbf{x},\mathbf{y}) \leq C\Big(1+\frac{\|\mathbf{x}-\mathbf{y}\|}{\sqrt{t}}\Big)^{-2}\mathcal{G}_{t/c}(\mathbf{x},\mathbf{y}).
\end{equation}
\end{theorem}

Theorem~\ref{theorem:heat} implies the following Lemma (see~\cite[Corollary 3.5]{DzH1}).

\begin{lemma}\label{lem:nonradial_estimation} 
Suppose that $\varphi \in C_c^{\infty}(\mathbb{R}^N)$ is radial and supported by the unit ball $B(0,1)$. Set $\varphi_t(\mathbf{x})=t^{-\mathbf{N}}\varphi(t^{-1}\mathbf{x})$. Then there is $C>0$ such that for all $\mathbf{x},\mathbf{y}\in \mathbb{R}^{N}$ and $t>0$ we have
\begin{align*}
    &|\tau_{\mathbf x}\varphi_t(-\mathbf y)| \leq C \Big(1+\frac{\|\mathbf{x}-\mathbf{y}\|}{t}\Big)^{-2}\Big(\max(w(B(\mathbf{x},t)),w(B(\mathbf{y},t)))\Big)^{-1}\chi_{[0,1]}(d(\mathbf{x},\mathbf{y})/t).
\end{align*}
\end{lemma}

\subsection{Dunkl-Schr\"odinger operator and semigroup}\label{sec:Schorodinger}
Let $V \geq 0$ be a measurable function such that $V \in L^2_{\rm{loc}}(dw)$. We consider the following operator on the Hilbert space $L^2(dw)$:
\begin{equation}\label{eq:Schrodinger_operator}
    \mathcal{L}=-\Delta_k+V
\end{equation}
with the domain
\begin{align*}
    \mathcal{D}(\mathcal{L})=\{f \in L^2(dw)\,:\,\|\xi\|^2\mathcal{F}f(\xi) \in L^2(dw(\xi)) \text{ and }V(\mathbf{x})f(\mathbf{x}) \in L^2(dw(\mathbf{x}))\}
\end{align*}
(see~\cite{AH}). We call this operator the \textit{Dunkl-Schr\"odinger operator}. Let us define the quadratic form
\begin{equation}
    \mathbf{Q}(f,g)=\sum_{j=1}^{N}\int_{\mathbb{R}^N}T_jf(\mathbf{x})\overline{T_jg(\mathbf{x})}\,dw(\mathbf{x})+\int_{\mathbb{R}^N}V(\mathbf{x})f(\mathbf{x})\overline{g(\mathbf{x})}\,dw(\mathbf{x})
\end{equation}
with the domain
\begin{align*}
    \mathcal{D}(\mathbf{Q})=\left\{f \in L^2(dw)\;:\; \left(\sum_{j=1}^{N}|T_jf|^2\right)^{1/2},V^{1/2}f \in L^2(dw)\right\} .
\end{align*}
The quadratic form is densely defined and closed (see~\cite[Lemma 4.1]{AH}), so there exists a unique positive self-adjoint operator $L$ such that
\begin{align*}
    \langle Lf,f\rangle=\mathbf{Q}(f,f) \text{ for all }f \in \mathcal{D}(L), 
\end{align*}
moreover,
\begin{align*}
    \mathcal{D}(L^{1/2})=\mathcal{D}(\mathbf{Q}) \text{ and }\mathbf{Q}(f,f)=\|L^{1/2}f\|^2_{L^2(dw)},
\end{align*}
where $L^{1/2}$ is a unique self-adjoint operator such that $(L^{1/2})^2=L$. By \cite[Theorem~4.2]{AH},
$C_c^\infty(\mathbb R^N)$ is dense in
$\mathcal D(\mathbf Q)$ with respect to the norm
\[
 \|f\|_{\mathbf Q}
 =\bigl(\|f\|_2^2+\mathbf Q(f,f)\bigr)^{1/2}.
\]
It was proved in~\cite[Theorem 4.6]{AH}, that $\mathcal{L}$ is essentially self-adjoint on $C^{\infty}_{c}(\mathbb{R}^N)$ and $L$ is its closure. Consequently, $-L$ generates the semigroup of self-adjoint contractions on $L^2(dw)$. The semigroup has the form (see~\cite[Theorem 4.8]{AH})
\begin{equation}\label{eq:K_semigroup}
    K_tf(\mathbf{x})=\int_{\mathbb{R}^N}k_t(\mathbf{x},\mathbf{y})f(\mathbf{y})\,dw(\mathbf{y}),
\end{equation}
where $k_t(\mathbf{x},\mathbf{y})$ is the integral kernel which satisfies
\begin{equation}\label{eq:kernels_compare}
   0 \leq k_t(\mathbf{x},\mathbf{y}) \leq h_t(\mathbf{x},\mathbf{y}). 
\end{equation}

\section{Reverse H\"older potentials and the Fefferman--Phong inequality}
\label{sec:Fefferman--Phong}

In this section we collect the properties of reverse H\"older potentials
and the associated auxiliary function that will be used later. The
doubling property, the scale estimate, the properties of the auxiliary
function, and the Fefferman--Phong inequality are recalled from
\cite[Sections~3--5]{Hejna_Schrodinger}.

\subsection{Reverse Hölder estimates and self-improvement}

Throughout the remainder of the paper, assume that
$V\geq0$, $V\not\equiv0$ modulo $dw$-null sets, and
$V\in{\rm RH}^q(dw)$ for some
$\max(1,\mathbf N/2)<q<\infty$. Thus there exists $C_{\rm RH}>0$ such that
\begin{equation}\label{eq:reverse_Holder}
 \left(\frac{1}{w(B)}\int_B V(\mathbf{x})^q\,dw(\mathbf{x})\right)^{1/q}
 \leq C_{\rm RH}\frac{1}{w(B)}\int_B V(\mathbf{x})\,dw(\mathbf{x})
\end{equation}
for every ball $B\subset\mathbb{R}^N$.

We shall also use the standard nesting and self-improvement properties
of reverse H\"older classes on the doubling measure space
$(\mathbb{R}^N,\|\cdot\|,dw)$.

\begin{lemma}\label{lem: self improve RH}
Let $1<q<\infty$ and $V\in {\rm RH}^{q}(dw)$.
\begin{enumerate}
\item[\rm(i)] $V\in {\rm RH}^{p}(dw)$ for every $1<p<q$.
\item[\rm(ii)] There exists $\varepsilon>0$ such that
$V\in {\rm RH}^{q+\varepsilon}(dw)$. The exponent $\varepsilon$ depends
only on $q$, the reverse H\"older constant, and the structural constants
of $dw$.
\end{enumerate}
\end{lemma}

Define the measure
\begin{equation}\label{eq:mu}
 \mu(A)=\int_A V(\mathbf{x})\,dw(\mathbf{x})
\end{equation}
for measurable sets $A\subseteq\mathbb{R}^N$.

\begin{lemma}[{\cite[Lemma~5]{Hejna_Schrodinger}}]
\label{lem:mu_doubling}
The measure $\mu$ is doubling: there exists $C_\mu>0$ such that
\[
 \mu(B(\mathbf{x},2r))\leq C_\mu\mu(B(\mathbf{x},r))
 \qquad(\mathbf{x}\in\mathbb{R}^N,\ r>0).
\]
\end{lemma}

The following estimate controls the potential on smaller scales.

\begin{lemma}[{\cite[Lemma~7]{Hejna_Schrodinger}}]\label{lem:Rr}
There exists $C\geq1$ such that, for every $\mathbf{x}\in\mathbb{R}^N$
and $0<r_1<r_2$,
\begin{align*}
 &\frac{r_1^2}{w(B(\mathbf{x},r_1))}
   \int_{B(\mathbf{x},r_1)}V(\mathbf{y})\,dw(\mathbf{y})
 \\
 &\qquad\leq
 C\left(\frac{r_1}{r_2}\right)^{2-\mathbf{N}/q}
 \frac{r_2^2}{w(B(\mathbf{x},r_2))}
 \int_{B(\mathbf{x},r_2)}V(\mathbf{y})\,dw(\mathbf{y}).
\end{align*}
\end{lemma}

\subsection{The auxiliary function
\texorpdfstring{$m(\mathbf{x})$}{m(x)}}

Following \cite[Section~4.1]{Hejna_Schrodinger}, define
\begin{equation}\label{eq:m}
 \frac{1}{m(\mathbf{x})}
 =\sup\left\{r>0:\;
 \frac{r^2}{w(B(\mathbf{x},r))}
 \int_{B(\mathbf{x},r)}V(\mathbf{y})\,dw(\mathbf{y})\leq1\right\}.
\end{equation}
Under the above assumptions, $0<m(\mathbf{x})<\infty$ for every
$\mathbf{x}\in\mathbb{R}^N$; see \cite[Section~4.1]{Hejna_Schrodinger}.

\begin{lemma}[{\cite[Lemma~8]{Hejna_Schrodinger}}]\label{lem:m_growth}
There exist constants $C\geq1$ and $\kappa>0$ such that, for every
$\mathbf{x},\mathbf{y}\in\mathbb{R}^N$,
\begin{equation}\label{eq:Shen_A}
 C^{-1}m(\mathbf{y})\leq m(\mathbf{x})\leq Cm(\mathbf{y})
 \quad\text{if }\|\mathbf{x}-\mathbf{y}\|<m(\mathbf{x})^{-1},
\end{equation}
\begin{equation}\label{eq:Shen_B}
 m(\mathbf{y})\leq Cm(\mathbf{x})
 \bigl(1+m(\mathbf{x})\|\mathbf{x}-\mathbf{y}\|\bigr)^\kappa,
\end{equation}
and
\begin{equation}\label{eq:Shen_C}
 m(\mathbf{y})\geq C^{-1}m(\mathbf{x})
 \bigl(1+m(\mathbf{x})\|\mathbf{x}-\mathbf{y}\|\bigr)^{-\kappa/(1+\kappa)}.
\end{equation}
\end{lemma}

If, in addition, $V$ is $G$-invariant, then the definition of $m$ and
the $G$-invariance of $dw$ give
\[
 m(g\mathbf{x})=m(\mathbf{x})
 \qquad(g\in G,\ \mathbf{x}\in\mathbb{R}^N).
\]

\subsection{The Fefferman--Phong inequality}

The following inequality was proved in
\cite[Theorem~6]{Hejna_Schrodinger}. Here $\mathbf{Q}$ is the quadratic
form associated with $L=-\Delta_k+V$, defined in
Section~\ref{sec:Schorodinger}.

\begin{theorem}[Fefferman--Phong inequality]\label{theo:Fefferman-Phong}
There exists $C>0$ such that, for every $f\in\mathcal{D}(\mathbf{Q})$,
\begin{equation}\label{eq:Fefferman-Phong}
 \int_{\mathbb{R}^N}|f(\mathbf{x})|^2m(\mathbf{x})^2\,dw(\mathbf{x})
 \leq C\mathbf{Q}(f,f).
\end{equation}
\end{theorem}
\section{Local estimates for weak solutions}

Throughout this section, $\Omega\subseteq\mathbb R^N$ is an open
$G$-invariant set. We write
\[
 U(\mathbf{x},r)=\mathcal O(B(\mathbf{x},r))
 =\bigcup_{g\in G}B(g(\mathbf{x}),r).
\]

\subsection{Weak solutions and localization of the quadratic form}
We use the local form domain
\[
 \mathcal D_{\mathrm{loc}}(\mathbf Q;\Omega)
 =\{u\in L^2_{\mathrm{loc}}(\Omega,dw):
       \chi u\in\mathcal D(\mathbf Q)
       \text{ for every $G$-invariant }\chi\in C_c^\infty(\Omega)\},
\]
where $\chi u$ is extended by zero outside $\Omega$.
A function $u$ in this domain is a weak solution of
$(-\Delta_k+V+i\tau)u=0$ in $\Omega$ if
\begin{equation}\label{eq:local_weak_equation}
 \sum_{j=1}^N\int_\Omega T_ju\,\overline{T_j\varphi}\,dw
 +\int_\Omega(V+i\tau)u\overline\varphi\,dw=0
 \qquad(\varphi\in C_c^\infty(\Omega)).
\end{equation}
The test functions in this identity need not be $G$-invariant.

We record two facts about localization of the form. We explain why multiplication by a fixed
$\varphi\in C_c^\infty(\Omega)$ preserves the local
form domain and is continuous in the corresponding
form norm.

Choose a compact $G$-invariant set $K\subset\Omega$
containing a neighbourhood of
$\mathcal O(\operatorname{supp}\varphi)$.
By~\eqref{eq:general_Leibniz},
\[
\begin{split}
 T_j(\varphi u)(\mathbf{x})
 &=
 \varphi(\mathbf{x})T_ju(\mathbf{x})
 +u(\mathbf{x})\partial_j\varphi(\mathbf{x})
 +
 \sum_{\alpha\in R}\frac{k(\alpha)}2\alpha_j
 u(\sigma_\alpha\mathbf{x})
 \frac{\varphi(\mathbf{x})
       -\varphi(\sigma_\alpha(\mathbf{x}))}
      {\langle\alpha,\mathbf{x}\rangle}.
\end{split}
\]
The first two terms are supported in
$\operatorname{supp}\varphi$.
For each $\alpha$, the coefficient of
$u(\sigma_\alpha\mathbf{x})$ in the last sum vanishes
outside
\[
 \operatorname{supp}\varphi
 \,\cup\,
 \sigma_\alpha(\operatorname{supp}\varphi)
 \subset K.
\]
Furthermore, this coefficient is bounded
by~\eqref{eq:multiplier_reflection_quotient}.
Since $K$ and $dw$ are $G$-invariant,
\[
 \int_K|u(\sigma_\alpha\mathbf{x})|^2\,dw(\mathbf{x})
 =
 \int_K|u(\mathbf{x})|^2\,dw(\mathbf{x}).
\]
Consequently,
\begin{equation}\label{eq:smooth_multiplier_local_bound}
 \|T_j(\varphi u)\|_{L^2(dw)}
 \leq
 \|\varphi\|_{L^\infty}\|T_ju\|_{L^2(K,dw)}
 +
 C_\varphi\|u\|_{L^2(K,dw)},
\end{equation}
where $\varphi u$ is extended by zero outside $\Omega$.
The constant $C_\varphi$ depends only on the root
system, the multiplicity function, and the first
derivatives of $\varphi$.

The remaining terms in the form norm satisfy
\[
 \|\varphi u\|_{L^2(dw)}
 \leq
 \|\varphi\|_{L^\infty}\|u\|_{L^2(K,dw)}
\]
and
\[
 \|V^{1/2}\varphi u\|_{L^2(dw)}
 \leq
 \|\varphi\|_{L^\infty}
 \|V^{1/2}u\|_{L^2(K,dw)}.
\]
Thus $\varphi u\in\mathcal D(\mathbf Q)$ whenever
$u\in\mathcal D_{\mathrm{loc}}(\mathbf Q;\Omega)$.

Applying the same estimates to $u_n-u$ also shows
continuity: if
\[
 u_n\to u,\qquad
 T_ju_n\to T_ju\quad(1\leq j\leq N),\qquad
 V^{1/2}u_n\to V^{1/2}u
 \quad\text{in }L^2(K,dw),
\]
then
\[
 \|\varphi(u_n-u)\|_2^2
 +
 \mathbf Q\bigl(\varphi(u_n-u),\varphi(u_n-u)\bigr)
 \longrightarrow0.
\]

For $f,g\in C^2(\mathbb R^N)$ we consider \textit{carr\'e du champ operator}: 
\begin{equation}\label{form_Gamma} \Gamma (f,g)=\frac{1}{2}\Big(\Delta_k (f \bar g)-f\Delta_k \bar g-\bar g\Delta_k f\Big).
\end{equation}

By  straightforward calculations (see~\cite[Lemma 3.1]{Velicu1}) we have
\begin{equation}\label{form_Gamma1} \Gamma(f,g)(\mathbf x)=\langle \nabla f(\mathbf x),\nabla \bar g(\mathbf x)\rangle + \sum_{\alpha \in R}\frac{k(\alpha)}{2}\frac{(f(\mathbf x)-f(\sigma_\alpha(\mathbf x)))(\overline{ g(\mathbf x)-g(\sigma_\alpha(\mathbf x)}))}{\langle \alpha,\mathbf x\rangle^2}.
\end{equation}

\subsection{Subharmonicity and mean-value estimates}
\begin{lemma}\label{lem:subharmonic}
Let $\Omega\subseteq\mathbb R^N$ be an open $G$-invariant set,
let $V\geq0$, and let $\tau\in\mathbb R$.
Suppose that
$u\in\mathcal D_{\mathrm{loc}}(\mathbf Q;\Omega)$
is a weak solution of
\begin{equation}\label{eq:main_eq}
    (-\Delta_k+V+i\tau)u=0
    \qquad\text{in }\Omega.
\end{equation}
Then $|u|^2$ is $\Delta_k$-subharmonic in the distributional
sense. More precisely, for every nonnegative
$\varphi\in C_c^\infty(\Omega)$,
\begin{equation}\label{eq:weak_subharmonicity}
    \int_\Omega |u|^2\Delta_k\varphi\,dw
    \geq
    2\int_\Omega V|u|^2\varphi\,dw
    \geq0.
\end{equation}
\end{lemma}

\begin{proof}
We first establish an inequality for smooth functions.
By the definition of the carré du champ,
\[
    \Delta_k|f|^2
    =
    2\operatorname{Re}(\overline f\,\Delta_k f)
    +2\Gamma(f,f).
\]
Moreover, formula~\eqref{form_Gamma1} gives
\[
    \Gamma(f,f)(\mathbf{x})
    =
    |\nabla f(\mathbf{x})|^2
    +
    \sum_{\alpha\in R}\frac{k(\alpha)}2
    \frac{|f(\mathbf{x})-f(\sigma_\alpha\mathbf{x})|^2}
         {\langle\alpha,\mathbf{x}\rangle^2}
    \geq0.
\]
Consequently, if $f\in C_c^\infty(\mathbb R^N)$ and
$0\leq\varphi\in C_c^\infty(\Omega)$, integration by parts
yields
\begin{align}
    \int_{\mathbb R^N}|f|^2\Delta_k\varphi\,dw
    &=
    2\int_{\mathbb R^N}\varphi\Gamma(f,f)\,dw
    -
    2\operatorname{Re}\sum_{j=1}^N
    \int_{\mathbb R^N}
    T_jf\,\overline{T_j(f\varphi)}\,dw
    \nonumber\\
    &\geq
    -2\operatorname{Re}\sum_{j=1}^N
    \int_{\mathbb R^N}
    T_jf\,\overline{T_j(f\varphi)}\,dw.
    \label{eq:smooth_subharmonicity_inequality}
\end{align}

We next pass this inequality to $u$.
Fix $0\leq\varphi\in C_c^\infty(\Omega)$.
Choose a $G$-invariant function
$\chi\in C_c^\infty(\Omega)$ equal to one on a neighbourhood
of $\mathcal O(\operatorname{supp}\varphi)$, and put
$v=\chi u$, extended by zero outside $\Omega$.
By the definition of the local form domain,
$v\in\mathcal D(\mathbf Q)$.

Using the density of $C_c^\infty(\mathbb R^N)$ in the form
domain, choose $v_n\in C_c^\infty(\mathbb R^N)$ such that
\[
    \|v_n-v\|_2^2
    +\mathbf Q(v_n-v,v_n-v)\longrightarrow0.
\]
In particular,
\[
    v_n\longrightarrow v,
    \qquad
    T_jv_n\longrightarrow T_jv
    \quad\text{in }L^2(dw).
\]
The product formula for the Dunkl operators also gives
\[
    T_j(v_n\varphi)\longrightarrow T_j(v\varphi)
    \quad\text{in }L^2(dw).
\]
Indeed, the derivatives of $\varphi$ and the quotients
\[
    \frac{\varphi(\mathbf{x})
          -\varphi(\sigma_\alpha(\mathbf{x}))}
         {\langle\alpha,\mathbf{x}\rangle}
\]
are bounded, so every term in the product formula converges
in $L^2(dw)$.

Furthermore,
\[
    \big\||v_n|^2-|v|^2\big\|_1
    \leq
    \|v_n-v\|_2\bigl(\|v_n\|_2+\|v\|_2\bigr)
    \longrightarrow0.
\]
Since $\Delta_k\varphi$ is bounded and compactly supported,
we may apply~\eqref{eq:smooth_subharmonicity_inequality}
to $v_n$ and pass to the limit. This gives
\[
    \int_{\mathbb R^N}|v|^2\Delta_k\varphi\,dw
    \geq
    -2\operatorname{Re}\sum_{j=1}^N
    \int_{\mathbb R^N}
    T_jv\,\overline{T_j(v\varphi)}\,dw.
\]
All terms in this inequality depend only on the values
of $v$ in a neighbourhood of
$\mathcal O(\operatorname{supp}\varphi)$.
There $\chi=1$, so the inequality becomes
\begin{equation}\label{eq:weak_localized_inequality}
    \int_\Omega |u|^2\Delta_k\varphi\,dw
    \geq
    -2\operatorname{Re}\sum_{j=1}^N
    \int_\Omega
    T_ju\,\overline{T_j(u\varphi)}\,dw.
\end{equation}

Finally, $u\varphi$ is an admissible form-domain test
function in the weak equation, by local approximation
in the form norm. Therefore,
\[
    \sum_{j=1}^N
    \int_\Omega T_ju\,\overline{T_j(u\varphi)}\,dw
    +
    \int_\Omega(V+i\tau)|u|^2\varphi\,dw
    =0.
\]
Taking real parts, we obtain
\[
    \operatorname{Re}\sum_{j=1}^N
    \int_\Omega T_ju\,\overline{T_j(u\varphi)}\,dw
    =
    -\int_\Omega V|u|^2\varphi\,dw.
\]
Substitution into~\eqref{eq:weak_localized_inequality}
proves~\eqref{eq:weak_subharmonicity}.
\end{proof}

Let $\mathbf{1}_A$ denote the characteristic function of the set $A \subseteq \mathbb{R}^N$. We state Theorem 3.1 of~\cite{GR}.

\begin{theorem}\label{teo:GR}
    Assume that $\Omega \subseteq \mathbb{R}^N$ is an open $G$-invariant set and $f \in C^2(\mathbb{R}^N)$. Then, for all $\mathbf{x} \in \Omega$ and $\rho>0$ such that $B(\mathbf{x},\rho) \subseteq \Omega$, and for all $0 < R \leq \frac{\rho}{3}$, we have
    \begin{equation}
        M_B^R(f)(\mathbf{x})=f(\mathbf{x})+\frac{1}{R^{\mathbf{N}}}\int_0^R\int_0^r M_B^t(\Delta_k f)(\mathbf{x}) t\,dt \, r^{\mathbf{N}-1}\,dr,
    \end{equation}
    where, for $t \in [0,R]$,
    \begin{equation}
        M_B^t(f)(\mathbf{x})=\frac{1}{w(B(0,t))}\int_{\mathbb{R}^N}\mathbf{1}_{B(0,t)}(\mathbf{x},\mathbf{y})f(\mathbf{y})\,dw(\mathbf{y}).
    \end{equation}
\end{theorem}

\begin{corollary}\label{coro:u1}
Let $u$ be as in Lemma~\ref{lem:subharmonic}. For each fixed $r>0$,
\begin{equation}\label{eq:weak_mean_value}
 |u(\mathbf{x})|^2\leq
 \frac1{w(B(0,r))}\int_\Omega
 \mathbf1_{B(0,r)}(\mathbf{x},\mathbf{y})|u(\mathbf{y})|^2\,dw(\mathbf{y})
\end{equation}
for almost every $\mathbf{x}$ such that
$\overline{U(\mathbf{x},3r)}\subset\Omega$.
\end{corollary}

\begin{proof}
Fix $r>0$ and write
\[
    A_r=
    \{\mathbf{x}\in\Omega:
      \overline{U(\mathbf{x},3r)}\subset\Omega\}.
\]
We first prove the assertion almost everywhere in a small
neighbourhood of an arbitrary point $\mathbf{x}_0\in A_r$.

Since $\overline{U(\mathbf{x}_0,3r)}$ is a compact subset
of $\Omega$, we may choose $\delta>0$ such that
\[
    \overline{U(\mathbf{x}_0,3r+4\delta)}
    \subset\Omega.
\]
Choose a $G$-invariant function
$\chi\in C_c^\infty(U(\mathbf{x}_0,3r+4\delta))$
such that
\[
    0\leq\chi\leq1,
    \qquad
    \chi=1
    \quad\text{on }U(\mathbf{x}_0,3r+3\delta).
\]
Set $f=\chi|u|^2$ in $\Omega$ and extend $f$ by zero
outside $\Omega$. Then $f\in L^1(dw)$ has compact support.

Choose a nonnegative radial function
$\rho\in C_c^\infty(B(0,1))$ with $\int\rho\,dw=1$,
and define
\[
    \rho_\varepsilon(\mathbf{x})
    =\varepsilon^{-\mathbf N}\rho(\mathbf{x}/\varepsilon),
    \qquad
    f_\varepsilon=f*\rho_\varepsilon.
\]
Since $\rho$ is radial, nonnegative, and normalized by
$\int_{\mathbb R^N}\rho\,dw=1$, Proposition~5.2 and
Remark~5.5 of~\cite{ADzH} apply to the kernels associated
with $\rho_\varepsilon$. Together with the
$L^1(dw)$-contractivity of convolution by $\rho_\varepsilon$,
they yield, by the same approximation argument as in
Corollary~5.4 of~\cite{ADzH},
\begin{equation}\label{eq:mean_value_L1_approximation}
    f_\varepsilon\longrightarrow f
    \quad\text{in }L^1(dw)
    \qquad\text{as }\varepsilon\to0.
\end{equation}
Moreover, $f_\varepsilon$ is smooth.

We now verify that $f_\varepsilon$ is subharmonic
on the smaller neighbourhood
\[
    \Omega_0=U(\mathbf{x}_0,3r+2\delta)
\]
whenever $0<\varepsilon<\delta$.
For $\mathbf{z}\in\Omega_0$, put
\[
    K_{\varepsilon,\mathbf{z}}(\mathbf{y})
    =\tau_{\mathbf{z}}\rho_\varepsilon(-\mathbf{y}).
\]
By positivity of radial Dunkl translation and its
support property,
\[
    K_{\varepsilon,\mathbf{z}}\geq0,
    \qquad
    \operatorname{supp}K_{\varepsilon,\mathbf{z}}
    \subseteq\overline{U(\mathbf{z},\varepsilon)}.
\]
Since $\mathbf{z}\in\Omega_0$ and $\varepsilon<\delta$,
this support is contained in
$U(\mathbf{x}_0,3r+3\delta)$, where $\chi=1$.
Thus $K_{\varepsilon,\mathbf{z}}$ is a nonnegative
smooth compactly supported test function in $\Omega$.

The Dunkl Laplacian acts identically in the two
variables of the convolution kernel:
\[
    \Delta_{k,\mathbf{z}}
    \bigl(\tau_{\mathbf{z}}\rho_\varepsilon(-\mathbf{y})\bigr)
    =
    \Delta_{k,\mathbf{y}}
    \bigl(\tau_{\mathbf{z}}\rho_\varepsilon(-\mathbf{y})\bigr).
\]
Consequently,
\begin{align*}
    \Delta_k f_\varepsilon(\mathbf{z})
    &=
    \int_{\mathbb R^N}
    f(\mathbf{y})\,
    \Delta_{k,\mathbf{y}}
    K_{\varepsilon,\mathbf{z}}(\mathbf{y})\,dw(\mathbf{y})
    \\
    &=
    \int_\Omega
    |u(\mathbf{y})|^2\,
    \Delta_{k,\mathbf{y}}
    K_{\varepsilon,\mathbf{z}}(\mathbf{y})\,dw(\mathbf{y})
    \geq0.
\end{align*}
Here the second equality follows because $\chi=1$
on a neighbourhood of the support of
$K_{\varepsilon,\mathbf{z}}$, which lies in a
$G$-invariant set. The last inequality is precisely
the distributional subharmonicity from
Lemma~\ref{lem:subharmonic}, applied to
$K_{\varepsilon,\mathbf{z}}$.

For every $\mathbf{x}\in B(\mathbf{x}_0,\delta)$,
we have
\[
    B(\mathbf{x},3r)\subset\Omega_0.
\]
We may therefore apply Theorem~\ref{teo:GR} to the
smooth function $f_\varepsilon$ on $\Omega_0$.
Since $\Delta_k f_\varepsilon\geq0$ there and the
averaging kernels are nonnegative, the integral
remainder in that theorem is nonnegative. Hence
\begin{equation}\label{eq:mean_value_regularized}
    f_\varepsilon(\mathbf{x})
    \leq M_B^r(f_\varepsilon)(\mathbf{x}),
    \qquad
    \mathbf{x}\in B(\mathbf{x}_0,\delta).
\end{equation}

It remains to pass to the limit.
The normalized kernel
\[
    \frac{\mathbf1_{B(0,r)}(\mathbf{x},\mathbf{y})}
         {w(B(0,r))}
\]
is nonnegative and has integral one in each variable.
Therefore $M_B^r$ is an $L^1(dw)$ contraction, and
\eqref{eq:mean_value_L1_approximation} implies
\[
    \|M_B^r(f_\varepsilon)-M_B^r(f)\|_{L^1(dw)}
    \leq
    \|f_\varepsilon-f\|_{L^1(dw)}
    \longrightarrow0.
\]
Choose a sequence $\varepsilon_n\to0$ along which
both $f_{\varepsilon_n}\to f$ and
$M_B^r(f_{\varepsilon_n})\to M_B^r(f)$ almost everywhere.
Passing to the limit in~\eqref{eq:mean_value_regularized}
gives
\[
    f(\mathbf{x})\leq M_B^r(f)(\mathbf{x})
\]
for almost every $\mathbf{x}\in B(\mathbf{x}_0,\delta)$.

Finally, for such $\mathbf{x}$ we have
$f(\mathbf{x})=|u(\mathbf{x})|^2$.
The kernel defining $M_B^r$ is supported in
$\overline{U(\mathbf{x},r)}$, which is contained in
the region where $\chi=1$. Thus
\[
    M_B^r(f)(\mathbf{x})
    =
    \frac1{w(B(0,r))}
    \int_\Omega
    \mathbf1_{B(0,r)}(\mathbf{x},\mathbf{y})
    |u(\mathbf{y})|^2\,dw(\mathbf{y}).
\]
This proves~\eqref{eq:weak_mean_value} almost everywhere
in $B(\mathbf{x}_0,\delta)$.

The set $A_r$ is covered by the neighbourhoods obtained
in this way. Taking a countable subcover and discarding
the union of the corresponding null sets proves the
assertion almost everywhere in $A_r$.
\end{proof}

In order to generalize the corollary above, we will use the following estimate.

\begin{proposition}\label{propo:characteristic_estimate}
    There is a constant $C>0$ such that for all $\mathbf{x},\mathbf{y} \in \mathbb{R}^N$ and $R>0$ we have
    \begin{equation*}
        \mathbf{1}_{B(0,R)}(\mathbf{x},\mathbf{y}) \leq C \frac{R^{\mathbf{N}}}{w(B(\mathbf{x},R))}\frac{R^2}{R^2+\|\mathbf{x}-\mathbf{y}\|^2} \mathbf{1}_{\mathcal{O}(B(\mathbf{x},R))}(\mathbf{y}).
    \end{equation*}
\end{proposition}

\begin{proof}
The explicit Gaussian formula gives
$\mathbf1_{B(0,R)}\leq C R^{\mathbf N}h_{R^2}$.
Positivity of radial Dunkl translation therefore implies
\[
 \mathbf1_{B(0,R)}(\mathbf{x},\mathbf{y})
 \leq C R^{\mathbf N}h_{R^2}(\mathbf{x},\mathbf{y}).
\]
By Theorem~\ref{theorem:heat} and doubling, the right-hand side is at most
\[
 C\frac{R^{\mathbf N}}{w(B(\mathbf{x},R))}
 \frac{R^2}{R^2+\|\mathbf{x}-\mathbf{y}\|^2}.
\]
Finally, R\"osler's formula~\eqref{eq:translation-radial} shows that
$\mathbf1_{B(0,R)}(\mathbf{x},\mathbf{y})=0$ when
$d(\mathbf{x},\mathbf{y})\geq R$, which gives the asserted support
restriction.
\end{proof}

\begin{corollary}\label{coro:c2}
Let $\gamma>1$ and let $u$ be a weak solution as in
Lemma~\ref{lem:subharmonic}. Suppose that
$\overline{U(\mathbf{x}_0,\gamma R)}\subset\Omega$. Then
\begin{equation}\label{eq:c2}
\begin{split}
 \operatorname*{ess\,sup}_{\mathbf{x}\in B(\mathbf{x}_0,R)}|u(\mathbf{x})|^2
 &\leq\frac{C_\gamma}{w(B(\mathbf{x}_0,R))}
 \sum_{g\in G}\frac{R^2}{R^2+\|\mathbf{x}_0-g(\mathbf{x}_0)\|^2}
\int_{B(g(\mathbf{x}_0),\gamma R)}|u(\mathbf{y})|^2\,dw(\mathbf{y}).
\end{split}
\end{equation}
In particular, $u\in L^\infty_{\mathrm{loc}}(\Omega,dw)$.
\end{corollary}

\begin{proof}
Put $r=(\gamma-1)R/10$. For every
$\mathbf{x}\in B(\mathbf{x}_0,R)$, the closure of
$U(\mathbf{x},3r)$ lies in $U(\mathbf{x}_0,\gamma R)$.
Apply Corollary~\ref{coro:u1} at radius $r$, followed by
Proposition~\ref{propo:characteristic_estimate}.
If $\mathbf{y}\in B(g(\mathbf{x}),r)$, then
\[
 R^2+\|\mathbf{x}-\mathbf{y}\|^2
 \simeq_\gamma R^2+\|\mathbf{x}_0-g\mathbf{x}_0\|^2,
 \qquad
 B(g\mathbf{x},r)\subset B(g\mathbf{x}_0,\gamma R).
\]
Also, $w(B(\mathbf{x},r))\simeq_\gamma w(B(\mathbf{x}_0,R))$.
Using $w(B(0,r))=r^{\mathbf N}w(B(0,1))$ gives
\eqref{eq:c2} for almost every $\mathbf{x}$, hence the essential supremum
estimate. Local boundedness follows by covering compact subsets of
$\Omega$ with such balls.
\end{proof}

Then, we will need the following lemma.

\begin{lemma}\label{lem:m_instead_of_square}
    Let $\gamma>1$. There is a constant $C=C(\gamma)>0$ such that for all $\mathbf{x}_0 \in \mathbb{R}^N$, $g \in G$, $R>0$, and $f \in L^2(dw)$ we have
    \begin{equation}\label{eq:using_square}
    \begin{split}
        &\frac{R^2}{R^2+\|g(\mathbf{x}_0)-\mathbf{x}_0\|^2}\int_{B(g(\mathbf{x}_0),\gamma R)}|f(\mathbf{y})|^2\,dw(\mathbf{y}) \\
        &\leq \frac{C  \left(1+(m(\mathbf{x}_0)R)^{\frac{2\kappa}{\kappa+1}}\right)}{m(\mathbf{x}_0)^2}\int_{B(g(\mathbf{x}_0),\gamma R)}|f(\mathbf{y})|^2m(\mathbf{y})^2\,dw(\mathbf{y}),
    \end{split}
    \end{equation}
    where $\kappa>0$ is defined in Lemma~\ref{lem:m_growth}.
\end{lemma}

\begin{proof}
    By~\eqref{eq:Shen_C}, for all $\mathbf{y} \in B(g(\mathbf{x}_0),\gamma R)$ we have
    \begin{equation*}
        \frac{1}{m(\mathbf{y})^2} \leq C \frac{1}{m(\mathbf{x}_0)^2}\left(1+m(\mathbf{x}_0)\|\mathbf{x}_0-\mathbf{y}\|\right)^{\frac{2\kappa}{\kappa+1}}.
    \end{equation*}
    Using the triangle inequality, for $\mathbf{y} \in B(g(\mathbf{x}_0),\gamma R)$ we get $\|\mathbf{x}_0-\mathbf{y}\| \leq \|\mathbf{x}_0-g(\mathbf{x}_0)\| + \gamma R$, which implies
    \begin{align*}
        \left(1+m(\mathbf{x}_0)\|\mathbf{x}_0-\mathbf{y}\|\right)^{\frac{2\kappa}{\kappa+1}} 
        &\leq C\left(1+m(\mathbf{x}_0)^2(\|\mathbf{x}_0-g(\mathbf{x}_0)\|^2+R^2)\right)^{\frac{\kappa}{\kappa+1}}.
    \end{align*}
    Thus, setting $z = m(\mathbf{x}_0)^2(\|\mathbf{x}_0-g(\mathbf{x}_0)\|^2+R^2) \geq m(\mathbf{x}_0)^2 R^2$, the factor in front of the integral is bounded by
    \begin{align*}
        \frac{R^2}{R^2+\|\mathbf{x}_0-g(\mathbf{x}_0)\|^2}\frac{1}{m(\mathbf{x}_0)^2} (1+z)^{\frac{\kappa}{\kappa+1}} 
        &= R^2 \frac{(1+z)^{\frac{\kappa}{\kappa+1}}}{z}.
    \end{align*}
    Set $a=\kappa/(\kappa+1)\in(0,1)$. The function
$h(z)=(1+z)^a/z$ is decreasing on $(0,\infty)$, since
\[
 h'(z)=\frac{(1+z)^{a-1}}{z^2}\bigl((a-1)z-1\bigr)<0.
\]
As $z\geq m(\mathbf{x}_0)^2R^2$, it follows that
\[
 R^2\frac{(1+z)^a}{z}
 \leq
 \frac{\bigl(1+m(\mathbf{x}_0)^2R^2\bigr)^a}
      {m(\mathbf{x}_0)^2}
 \leq
 \frac{1+(m(\mathbf{x}_0)R)^{2a}}
      {m(\mathbf{x}_0)^2},
\]
where the last inequality uses $(1+t)^a\leq 1+t^a$
for $t\geq0$. Consequently, for every
$\mathbf y\in B(g(\mathbf{x}_0),\gamma R)$,
\[
 \frac{R^2}{R^2+\|g(\mathbf{x}_0)-\mathbf{x}_0\|^2}
 \leq
 \frac{C\bigl(1+(m(\mathbf{x}_0)R)^{2\kappa/(\kappa+1)}\bigr)}
      {m(\mathbf{x}_0)^2}\,m(\mathbf y)^2.
\]
Multiplying by $|f(\mathbf y)|^2$ and integrating over
$B(g(\mathbf{x}_0),\gamma R)$ proves
\eqref{eq:using_square}.
\end{proof}

\subsection{Caccioppoli and weighted local estimates}

We will also need the following lemma (variant of Caccioppoli's inequality).

\begin{lemma}\label{lem:cacc}
Let $\gamma>1$, $\tau\in\mathbb R$, and let
$u\in\mathcal D_{\mathrm{loc}}(\mathbf Q;\Omega)$ be a weak solution of
$(-\Delta_k+V+i\tau)u=0$ in $\Omega$.
If $\overline{U(\mathbf{x}_0,\gamma R)}\subset\Omega$, then
\begin{equation}\label{eq:local_caccioppoli}
\begin{split}
 &\int_{U(\mathbf{x}_0,R)}
 \left(\sum_{j=1}^N|T_ju|^2+V|u|^2+|\tau||u|^2\right)\,dw
 \leq \frac{C_\gamma}{R^2}
 \int_{U(\mathbf{x}_0,\gamma R)}|u|^2\,dw.
\end{split}
\end{equation}
The same estimate holds for solutions of
$(-\Delta_k+V+a+i\tau)u=0$, $a\geq0$, with a constant independent of $a$.
\end{lemma}

\begin{proof}
Choose a $G$-invariant $\eta\in C_c^\infty(U(\mathbf{x}_0,\gamma R))$
with $0\leq\eta\leq1$, $\eta=1$ on $U(\mathbf{x}_0,R)$, and
$|\nabla\eta|\leq C_\gamma/R$. Such a function can be constructed from
ordinary ball cutoffs by taking $1-\prod_{g\in G}(1-\eta_0\circ g)$.
By the local form approximation described above, $\eta^2u$ is an
admissible test function. Since
$T_j(\eta^2u)=\eta^2T_ju+2\eta(\partial_j\eta)u$, we obtain
\[
 A+B+i\tau C=-2Z,
\]
where
\[
 A=\int\eta^2\sum_j|T_ju|^2\,dw,\quad
 B=\int(V+a)\eta^2|u|^2\,dw,\quad
 C=\int\eta^2|u|^2\,dw,
 \qquad
 Z=\sum_j\int\eta T_ju\,\overline u\,\partial_j\eta\,dw.
\]
Taking real and imaginary parts and applying Cauchy--Schwarz and Young's
inequalities yields
\[
 A+B+|\tau|C\leq4|Z|
 \leq\tfrac12 A+8\int|u|^2|\nabla\eta|^2\,dw.
\]
The properties of $\eta$ imply \eqref{eq:local_caccioppoli}; the
nonnegative term involving $a$ can be discarded.
\end{proof}

\subsection{Decay on the critical-radius scale}
    We conclude this section by the following proposition.

\begin{proposition}\label{propo:large_decay}
Let $\kappa_1\geq0$ and $\tau\in\mathbb R$. Suppose that
$u\in\mathcal D_{\mathrm{loc}}(\mathbf Q;\Omega)$ is a weak solution of
$(-\Delta_k+V+i\tau)u=0$ in $\Omega$.
If $\overline{U(\mathbf{x}_0,2R)}\subset\Omega$ and
$Rm(\mathbf{x}_0)\geq1$, then
\begin{equation}\label{eq:claim_21}
\begin{split}
 &\operatorname*{ess\,sup}_{B(\mathbf{x}_0,R)}|u|^2
 \leq
 C_{\kappa_1}(Rm(\mathbf{x}_0))^{-2/(\kappa+1)}
 \max_{g\in G}(Rm(g\mathbf{x}_0))^{-2\kappa_1}
 \frac1{w(B(\mathbf{x}_0,R))}
 \int_{U(\mathbf{x}_0,2R)}|u|^2\,dw,
\end{split}
\end{equation}
where $\kappa$ is the exponent in Lemma~\ref{lem:m_growth}.
\end{proposition}

\begin{proof}
Put
\[
 \beta=\frac{2}{\kappa+1}.
\]
We first prove an estimate that yields one factor of decay.
We then iterate it over increasing orbit neighbourhoods.

Fix $\gamma>1$. Let $\mathbf a\in\mathbb R^N$ and $s>0$
satisfy
\[
 \overline{U(\mathbf a,\gamma^2s)}\subset\Omega,
 \qquad sm(\mathbf a)\geq1.
\]
Combining Corollary~\ref{coro:c2} with
Lemma~\ref{lem:m_instead_of_square}, we obtain
\begin{equation}\label{eq:large_decay_weighted_step}
 \operatorname*{ess\,sup}_{B(\mathbf a,s)}|u|^2
 \leq
 \frac{C_\gamma
       \bigl(1+(sm(\mathbf a))^{2\kappa/(\kappa+1)}\bigr)}
      {w(B(\mathbf a,s))m(\mathbf a)^2}
 \int_{U(\mathbf a,\gamma s)}m(\mathbf y)^2
                    |u(\mathbf y)|^2\,dw(\mathbf y).
\end{equation}
To apply Lemma~\ref{lem:m_instead_of_square}, which is
formulated for global $L^2(dw)$ functions, we first multiply
$u$ by a compactly supported cutoff equal to one on
$\overline{U(\mathbf a,\gamma s)}$.
The balls occurring in Corollary~\ref{coro:c2} lie in this
orbit neighbourhood, so this localization does not change
their integrals. Summing over $G$ introduces at most a factor
$|G|$, which is absorbed into $C_\gamma$.

We next estimate the weighted integral in
\eqref{eq:large_decay_weighted_step}.
Choose
\[
 t=\frac{\gamma+\gamma^2}{2}
\]
and a $G$-invariant cutoff
$\eta\in C_c^\infty(U(\mathbf a,ts))$ such that
\[
 0\leq\eta\leq1,\qquad
 \eta=1\ \text{on }U(\mathbf a,\gamma s),\qquad
 |\nabla\eta|\leq C_\gamma s^{-1}.
\]
Since $u$ belongs to the local form domain, $\eta u$ belongs
to the global form domain. The Fefferman--Phong inequality
therefore gives
\[
 \int_{U(\mathbf a,\gamma s)}m^2|u|^2\,dw
 \leq \int_{\mathbb R^N}m^2|\eta u|^2\,dw
 \leq C\mathbf Q(\eta u,\eta u).
\]
Moreover, the $G$-invariance of $\eta$ implies
\[
 T_j(\eta u)=\eta T_ju+u\,\partial_j\eta.
\]
Consequently,
\[
 \mathbf Q(\eta u,\eta u)
 \leq
 C\int_{U(\mathbf a,ts)}
       \left(\sum_{j=1}^N|T_ju|^2+V|u|^2\right)\,dw
 +C_\gamma s^{-2}
       \int_{U(\mathbf a,ts)}|u|^2\,dw.
\]
Applying Lemma~\ref{lem:cacc} between
$U(\mathbf a,ts)$ and $U(\mathbf a,\gamma^2s)$ yields
\[
 \int_{U(\mathbf a,\gamma s)}m^2|u|^2\,dw
 \leq C_\gamma s^{-2}
       \int_{U(\mathbf a,\gamma^2s)}|u|^2\,dw.
\]
Substituting this into
\eqref{eq:large_decay_weighted_step}, and using
$sm(\mathbf a)\geq1$, we obtain
\[
 \frac{1+(sm(\mathbf a))^{2\kappa/(\kappa+1)}}
      {(sm(\mathbf a))^2}
 \leq 2(sm(\mathbf a))^{-\beta}.
\]
Thus
\begin{equation}\label{eq:procedure_first_step}
 \operatorname*{ess\,sup}_{B(\mathbf a,s)}|u|^2
 \leq
 \frac{C_\gamma(sm(\mathbf a))^{-\beta}}
      {w(B(\mathbf a,s))}
 \int_{U(\mathbf a,\gamma^2s)}|u|^2\,dw.
\end{equation}

We now apply this estimate to the neighbourhoods of
$\mathbf x_0$. Write
\[
 b_*=\min_{g\in G}Rm(g\mathbf x_0),
 \qquad
 I(s)=\int_{U(\mathbf x_0,s)}|u|^2\,dw.
\]
In particular,
\[
 b_*^{-2\kappa_1}
 =\max_{g\in G}(Rm(g\mathbf x_0))^{-2\kappa_1}.
\]

If $\kappa_1=0$, applying
\eqref{eq:procedure_first_step} with
$\mathbf a=\mathbf x_0$, $s=R$, and $\gamma^2<2$
proves the assertion directly.
The same argument applies when $b_*<1$, since then
$b_*^{-2\kappa_1}\geq1$.

It remains to consider $b_*\geq1$ and $\kappa_1>0$.
Choose an integer $n\geq1$ such that
\[
 n\beta\geq2\kappa_1,
\]
and then choose $\gamma>1$ sufficiently close to one that
\[
 \gamma^{2(n+1)}<2.
\]
Set
\[
 s_i=\gamma^{2i}R,\qquad i=0,\ldots,n+1.
\]
The first application of
\eqref{eq:procedure_first_step} gives
\begin{equation}\label{eq:large_decay_initial_step}
 \operatorname*{ess\,sup}_{B(\mathbf x_0,R)}|u|^2
 \leq
 \frac{C_\gamma(Rm(\mathbf x_0))^{-\beta}}
      {w(B(\mathbf x_0,R))}\,I(s_1).
\end{equation}

For each $i=1,\ldots,n$ and $g\in G$, we have
\[
 s_i m(g\mathbf x_0)\geq b_*\geq1,
 \qquad
 U(g\mathbf x_0,\gamma^2s_i)
   =U(\mathbf x_0,s_{i+1}).
\]
Hence another application of
\eqref{eq:procedure_first_step} gives
\[
 \operatorname*{ess\,sup}_{B(g\mathbf x_0,s_i)}|u|^2
 \leq
 \frac{C_\gamma b_*^{-\beta}}
      {w(B(g\mathbf x_0,s_i))}\,I(s_{i+1}).
\]
Multiplying by $w(B(g\mathbf x_0,s_i))$ and summing over
$g\in G$, we obtain
\[
 \begin{split}
 I(s_i)
 &\leq \sum_{g\in G}
        \int_{B(g\mathbf x_0,s_i)}|u|^2\,dw\leq C_\gamma |G|\,b_*^{-\beta}I(s_{i+1}).
 \end{split}
\]
Iterating this inequality yields
\[
 I(s_1)\leq (C_\gamma|G|)^n b_*^{-n\beta}I(s_{n+1}).
\]
Since $b_*\geq1$ and $n\beta\geq2\kappa_1$, we have
\[
 b_*^{-n\beta}\leq b_*^{-2\kappa_1}.
\]
Also, $s_{n+1}<2R$, so
\[
 I(s_{n+1})\leq I(2R).
\]
Combining these estimates with
\eqref{eq:large_decay_initial_step} proves
\eqref{eq:claim_21}. All applications of the one-step
estimate take place in orbit neighbourhoods whose closures
are contained in $U(\mathbf x_0,2R)\subset\Omega$.
The constants introduced by the iteration depend on
$\kappa_1$ and the fixed structural parameters and are
absorbed into $C_{\kappa_1}$.
\end{proof}

\section{Resolvent kernels}

\subsection{Kernel representation and basic estimates}
The following proposition establishes a formal relationship between the kernel of the semigroup (heat kernel) generated by the operator $-\Delta_k + V$ and the fundamental solution (resolvent kernel) of the operator $-\Delta_k+V+i\tau I$.

\begin{proposition}
Let $\{K_t\}_{t \ge 0} = \{e^{-tL}\}_{t \ge 0}$ be the self-adjoint contraction semigroup generated by $L$ (see~\eqref{eq:K_semigroup}).
Then, for any $\tau \in \mathbb{R} \setminus \{0\}$, the fundamental solution $\Gamma(\mathbf{x},\mathbf{y},\tau)$ of the operator $-\Delta_k + V(\mathbf{x}) + i\tau I$, defined as the integral kernel of the resolvent operator $R(-i\tau, L) = (L + i\tau I)^{-1}$, is given by the formula:
\begin{equation}\label{eq:resolvent_kernel}
\Gamma(\mathbf{x},\mathbf{y},\tau) = \int_0^\infty e^{-i\tau t} k_t(\mathbf{x},\mathbf{y}) \, dt,
\end{equation}
where the integral on the right-hand side converges in the weak sense (in the space of Dunkl distributions or as a limit in the quadratic form formulation).
\end{proposition}

\begin{proof}
From the theory of linear operator semigroups and the Hille--Yosida theorem, it follows that for a nonnegative self-adjoint operator $L$, the resolvent $R(-\lambda, L) = (L + \lambda I)^{-1}$ for $\text{Re}(\lambda) > 0$ can be uniquely represented as the Laplace transform of the semigroup $\{K_t\}_{t \ge 0}$ in the strong operator topology:
\begin{equation}\label{eq:hille_yosida}
(L + \lambda I)^{-1} = \int_0^\infty e^{-\lambda t} K_t \, dt.
\end{equation}
Since the spectrum $\sigma(L)$ is contained in the half-line $[0, \infty)$, the point $-i\tau$ for $\tau \neq 0$ belongs to the resolvent set $\rho(L)$. By virtue of the spectral theorem for self-adjoint operators, this relation extends to the imaginary axis in the following sense:
\begin{equation}
(L + i\tau I)^{-1} f = \lim_{\varepsilon \to 0^+} \int_0^\infty e^{-(\varepsilon + i\tau) t} K_t f \, dt,
\end{equation}
where the limit converges in the $L^2(\mathbb{R}^N, dw)$-norm for every function $f \in L^2(dw)$.

Now, consider the action of this operator on any test function $f \in C_c^\infty(\mathbb{R}^N)$. Utilizing the kernel representation of the semigroup $K_t$ (see~\eqref{eq:K_semigroup}), we obtain:
\begin{equation}
\left( (L + i\tau I)^{-1} f \right)(\mathbf{x}) = \lim_{\varepsilon \to 0^+} \int_0^\infty e^{-(\varepsilon + i\tau) t} \left( \int_{\mathbb{R}^N} k_t(\mathbf{x},\mathbf{y}) f(\mathbf{y}) \, dw(\mathbf{y}) \right) dt.
\end{equation}
Applying Fubini's theorem to interchange the order of integration (which is justified by the absolute convergence introduced by the regularization factor $e^{-\varepsilon t}$, boundedness of $f$, and~\eqref{eq:kernels_compare}), we get:
\begin{equation}
\left( (L + i\tau I)^{-1} f \right)(\mathbf{x}) = \lim_{\varepsilon \to 0^+} \int_{\mathbb{R}^N} \left( \int_0^\infty e^{-(\varepsilon + i\tau) t} k_t(\mathbf{x},\mathbf{y}) \, dt \right) f(\mathbf{y}) \, dw(\mathbf{y}).
\end{equation}
Passing to the limit as $\varepsilon \to 0^+$ (interpreted in the sense of weak operator convergence or within the distributional space dual to the test functions), we arrive at:
\begin{equation}
\left( (L + i\tau I)^{-1} f \right)(\mathbf{x}) = \int_{\mathbb{R}^N} \left( \int_0^\infty e^{-i\tau t} k_t(\mathbf{x},\mathbf{y}) \, dt \right) f(\mathbf{y}) \, dw(\mathbf{y}).
\end{equation}
On the other hand, by the definition of the resolvent kernel (fundamental solution) $\Gamma(\mathbf{x},\mathbf{y},\tau)$, the following identity holds:
\begin{equation}
\left( (L + i\tau I)^{-1} f \right)(\mathbf{x}) = \int_{\mathbb{R}^N} \Gamma(\mathbf{x},\mathbf{y},\tau) f(\mathbf{y}) \, dw(\mathbf{y}).
\end{equation}
By the uniqueness of the integral representation for densely defined operators on $L^2(dw)$, it immediately follows that these kernels must coincide almost everywhere with respect to the measure $dw(\mathbf{y})$, which leads directly to:
\begin{equation}
\Gamma(\mathbf{x},\mathbf{y},\tau) = \int_0^\infty e^{-i\tau t} k_t(\mathbf{x},\mathbf{y}) \, dt.
\end{equation}
This completes the proof.
\end{proof}

In the next proposition, we will prove an estimate for $\Gamma(\mathbf{x},\mathbf{y},\tau)$, which is a consequence of~\eqref{eq:kernels_compare} and the estimates for $h_t(\mathbf{x},\mathbf{y})$ (see Theorem~\ref{theorem:heat}).

\begin{proposition}\label{propo:basic_decay}
Assume $N>2$ and let $\tau\in\mathbb R\setminus\{0\}$.
The resolvent kernel $\Gamma(\mathbf x,\mathbf y,\tau)$
admits the representation
\[
 \Gamma(\mathbf x,\mathbf y,\tau)
 =\int_0^\infty e^{-i\tau t}k_t(\mathbf x,\mathbf y)\,dt,
\]
where the integral converges absolutely for almost every
$(\mathbf x,\mathbf y)$ with $d(\mathbf x,\mathbf y)>0$.
Moreover,
\[
 |\Gamma(\mathbf x,\mathbf y,\tau)|
 \leq C\frac{d(\mathbf x,\mathbf y)^2}
 {w(B(\mathbf x,d(\mathbf x,\mathbf y)))}
\]
for almost every such pair, with $C$ independent of $\tau$.
\end{proposition}

\begin{proof}
We first estimate the time integral of the heat kernel.
Fix $\mathbf x,\mathbf y$ with
$r=d(\mathbf x,\mathbf y)>0$, and write
\[
 \int_0^\infty h_t(\mathbf x,\mathbf y)\,dt
 =
 \int_0^{r^2}h_t(\mathbf x,\mathbf y)\,dt
 +
 \int_{r^2}^\infty h_t(\mathbf x,\mathbf y)\,dt
 =:I_1+I_2.
\]

For $0<t\leq r^2$, Theorem~\ref{theorem:heat}
and the upper volume growth estimate in
\eqref{eq:growth} give
\[
 h_t(\mathbf x,\mathbf y)
 \leq
 \frac{C}{w(B(\mathbf x,\sqrt t))}
       e^{-cr^2/t}
 \leq
 \frac{C}{w(B(\mathbf x,r))}
       \left(\frac r{\sqrt t}\right)^{\mathbf N}
       e^{-cr^2/t}
 \leq \frac{C}{w(B(\mathbf x,r))}.
\]
The last inequality follows from the boundedness of
$v^{\mathbf N}e^{-cv^2}$ for $v\geq1$.
Consequently,
\[
 I_1\leq C\frac{r^2}{w(B(\mathbf x,r))}.
\]

For $t\geq r^2$, the lower volume growth estimate yields
\[
 w(B(\mathbf x,\sqrt t))
 \geq C^{-1}\left(\frac{\sqrt t}{r}\right)^N
                  w(B(\mathbf x,r)).
\]
Therefore, since $N>2$,
\[
 \begin{split}
 I_2
 &\leq
 \frac{Cr^N}{w(B(\mathbf x,r))}
       \int_{r^2}^\infty t^{-N/2}\,dt\leq C\frac{r^2}{w(B(\mathbf x,r))}.
 \end{split}
\]
Together, these estimates show that
\begin{equation}\label{eq:integrated_heat_bound}
 \int_0^\infty h_t(\mathbf x,\mathbf y)\,dt
 \leq C\frac{r^2}{w(B(\mathbf x,r))}<\infty.
\end{equation}

By \eqref{eq:kernels_compare}, we may thus define
\[
 \widetilde\Gamma(\mathbf x,\mathbf y,\tau)
 :=\int_0^\infty e^{-i\tau t}
                   k_t(\mathbf x,\mathbf y)\,dt
\]
by an absolutely convergent integral off the orbit diagonal.
Furthermore,
\[
 |\widetilde\Gamma(\mathbf x,\mathbf y,\tau)|
 \leq \int_0^\infty h_t(\mathbf x,\mathbf y)\,dt.
\]
It remains to identify $\widetilde\Gamma$ with the
resolvent kernel.

For $\varepsilon>0$, put
\[
 \Gamma_\varepsilon(\mathbf x,\mathbf y,\tau)
 =\int_0^\infty e^{-(\varepsilon+i\tau)t}
                   k_t(\mathbf x,\mathbf y)\,dt.
\]
By dominated convergence,
$\Gamma_\varepsilon\to\widetilde\Gamma$ almost everywhere
off the orbit diagonal. To pass to the limit against
test functions, we also need an integrable majorant
in both spatial variables.

Let $f,\varphi\in C_c^\infty(\mathbb R^N)$, and let
$H_t$ denote the semigroup with kernel $h_t$.
For $0<t\leq1$, its $L^2(dw)$-contractivity gives
\[
 \int_0^1
 \langle H_t|f|,|\varphi|\rangle_{L^2(dw)}\,dt
 \leq \|f\|_{L^2(dw)}\|\varphi\|_{L^2(dw)}.
\]
For $t\geq1$, the heat kernel estimate and lower volume
growth imply
\[
 h_t(\mathbf x,\mathbf y)
 \leq C_{\operatorname{supp}\varphi}\,t^{-N/2},
 \qquad \mathbf x\in\operatorname{supp}\varphi.
\]
Since $N>2$, it follows that
\[
 \int_1^\infty
 \langle H_t|f|,|\varphi|\rangle_{L^2(dw)}\,dt<\infty.
\]
Thus
\[
 \int_0^\infty\int_{\mathbb R^N}\int_{\mathbb R^N}
 h_t(\mathbf x,\mathbf y)
 |f(\mathbf y)|\,|\varphi(\mathbf x)|
 \,dw(\mathbf y)\,dw(\mathbf x)\,dt<\infty.
\]

Fubini's theorem and dominated convergence now yield
\[
 \begin{split}
 &\int_{\mathbb R^N}\int_{\mathbb R^N}
 \widetilde\Gamma(\mathbf x,\mathbf y,\tau)
 f(\mathbf y)\overline{\varphi(\mathbf x)}
 \,dw(\mathbf y)\,dw(\mathbf x)\\
 &\qquad=
 \lim_{\varepsilon\downarrow0}
 \left\langle
 (L+(\varepsilon+i\tau)I)^{-1}f,\varphi
 \right\rangle_{L^2(dw)}\\
 &\qquad=
 \left\langle (L+i\tau I)^{-1}f,\varphi
 \right\rangle_{L^2(dw)}.
 \end{split}
\]
Hence $\widetilde\Gamma$ is a locally integrable kernel
of the resolvent and agrees almost everywhere with
$\Gamma$. The desired estimate follows from
\eqref{eq:integrated_heat_bound}.
\end{proof}

The same heat-kernel majorant also gives absolute convergence,
for almost every pair $(\mathbf x,\mathbf y)$ with
$d(\mathbf x,\mathbf y)>0$, of
\[
 \Gamma(\mathbf x,\mathbf y,0)
 :=\int_0^\infty k_t(\mathbf x,\mathbf y)\,dt.
\]
This defines the Green kernel, which satisfies the bound
in Proposition~\ref{propo:basic_decay} as well.
No bounded inverse of $L$ on all of $L^2(dw)$ is asserted.
All kernel statements below are understood almost everywhere
for each fixed $\tau\in\mathbb R$.

\subsection{The resolvent kernel as a local weak solution}
\begin{lemma}\label{lem:resolvent_local_weak_solution}
Assume $N>2$ and fix $\tau\in\mathbb R$.
For almost every $\mathbf y\in\mathbb R^N$, the function
\[
 u(\mathbf x)=\Gamma(\mathbf x,\mathbf y,\tau)
\]
belongs to
\[
 \mathcal D_{\mathrm{loc}}
 \bigl(\mathbf Q;\mathbb R^N\setminus\mathcal O(\{\mathbf y\})\bigr)
\]
and satisfies
\[
 (-\Delta_k+V+i\tau)u=0
\]
weakly in $\mathbb R^N\setminus\mathcal O(\{\mathbf y\})$.
\end{lemma}

\begin{proof}
We approximate the kernel by solutions with an $L^2(dw)$
right-hand side. Local Caccioppoli estimates allow us to
pass to the limit in the weak equation.

For $a>0$ and $\varepsilon>0$, set
\[
 F_{\varepsilon,\mathbf y}
 =\frac{\mathbf1_{B(\mathbf y,\varepsilon)}}
        {w(B(\mathbf y,\varepsilon))},
 \qquad
 u_{a,\varepsilon}
 =(L+(a+i\tau)I)^{-1}F_{\varepsilon,\mathbf y}.
\]
Since $F_{\varepsilon,\mathbf y}\in L^2(dw)$, we have
$u_{a,\varepsilon}\in\mathcal D(L)$ and
\[
 (L+a+i\tau)u_{a,\varepsilon}=F_{\varepsilon,\mathbf y}.
\]
In particular, the equation is homogeneous on every
invariant open set disjoint from $B(\mathbf y,\varepsilon)$.

The kernel of this regularized resolvent is
\[
 \Gamma_a(\mathbf x,\mathbf z,\tau)
 =\int_0^\infty e^{-(a+i\tau)t}
                    k_t(\mathbf x,\mathbf z)\,dt.
\]
Thus
\[
 u_{a,\varepsilon}(\mathbf x)
 =\frac1{w(B(\mathbf y,\varepsilon))}
   \int_{B(\mathbf y,\varepsilon)}
       \Gamma_a(\mathbf x,\mathbf z,\tau)\,dw(\mathbf z).
\]
By heat-kernel domination and the estimate proved in
Proposition~\ref{propo:basic_decay},
\begin{equation}\label{eq:regularized_green_majorant}
 |\Gamma_a(\mathbf x,\mathbf z,\tau)|
 \leq \int_0^\infty h_t(\mathbf x,\mathbf z)\,dt
 \leq C\frac{d(\mathbf x,\mathbf z)^2}
 {w(B(\mathbf x,d(\mathbf x,\mathbf z)))},
\end{equation}
uniformly in $a>0$.

Fix bounded $G$-invariant open sets $U,W$ such that
\[
 \overline U\subset W,
 \qquad
 \overline W\subset
 \mathbb R^N\setminus\mathcal O(\{\mathbf y\}).
\]
The distance between $\overline W$ and the orbit of
$\mathbf y$ is positive. Consequently, for sufficiently
small $\varepsilon$, the right-hand side of
\eqref{eq:regularized_green_majorant} is uniformly bounded
for $\mathbf x\in W$ and
$\mathbf z\in B(\mathbf y,\varepsilon)$.
Averaging in $\mathbf z$, we obtain
\[
 \|u_{a,\varepsilon}\|_{L^\infty(W)}\leq C_{W,\mathbf y},
\]
uniformly in $a$ and $\varepsilon$.

Applying the local Caccioppoli estimate to
$u_{a,\varepsilon}$ gives
\begin{equation}\label{eq:kernel_local_form_bound}
 \int_U
 \left(
 |u_{a,\varepsilon}|^2
 +\sum_{j=1}^N|T_ju_{a,\varepsilon}|^2
 +V|u_{a,\varepsilon}|^2
 \right)\,dw
 \leq C_{U,W,\mathbf y}.
\end{equation}
The constant is independent of $a>0$: the additional
potential term $a|u_{a,\varepsilon}|^2$ is nonnegative
and may be discarded in the Caccioppoli estimate.

We first let $\varepsilon\downarrow0$, keeping
$a=a_n:=1/n$ fixed. On separated bounded sets, the
majorant \eqref{eq:regularized_green_majorant} makes
\[
 \mathbf z\longmapsto
 \Gamma_{a_n}(\cdot,\mathbf z,\tau)
\]
a locally integrable function with values in $L^2(U,dw)$.
The Lebesgue differentiation theorem for such functions
therefore yields
\begin{equation}\label{eq:kernel_source_approximation}
 u_{a_n,\varepsilon}
 \longrightarrow
 \Gamma_{a_n}(\cdot,\mathbf y,\tau)
 \quad\text{in }L^2(U,dw)
\end{equation}
for almost every admissible $\mathbf y$.

Here and below, the exceptional set can be chosen
independently of $n$ and of the localizations used in the
proof. Indeed, one first applies differentiation to a
countable family of invariant open sets, obtained from
finite unions of $G$-orbits of balls with rational centres
and radii, and then takes the union of the resulting null
sets. This family provides the required localizations
away from each orbit.

We explain how the form bound allows us to pass to the
limit. Let $\chi\in C_c^\infty(U)$ be $G$-invariant.
By the product rule and
\eqref{eq:kernel_local_form_bound}, the functions
$\chi u_{a_n,\varepsilon}$ are bounded in
$\mathcal D(\mathbf Q)$ equipped with the norm
\[
 \|v\|_{\mathbf Q}^2
 :=\|v\|_{L^2(dw)}^2+\mathbf Q(v,v).
\]
They therefore have a weakly convergent subsequence
in this Hilbert space. Their strong $L^2(dw)$ limit,
identified by \eqref{eq:kernel_source_approximation},
is
\[
 \chi\Gamma_{a_n}(\cdot,\mathbf y,\tau).
\]
It follows that this function belongs to
$\mathcal D(\mathbf Q)$. Since the cutoff is arbitrary,
$\Gamma_{a_n}(\cdot,\mathbf y,\tau)$ belongs to the
local form domain away from the orbit of $\mathbf y$.

Recall that, by the definition of $u_{a_n,\varepsilon}$,
\[
 (L+a_n+i\tau)u_{a_n,\varepsilon}
 =F_{\varepsilon,\mathbf y}.
\]
For sufficiently small $\varepsilon$, the support of
$F_{\varepsilon,\mathbf y}$ is disjoint from $U$.
Testing this equation with $\varphi\in C_c^\infty(U)$
therefore gives
\[
 \sum_{j=1}^N
 \int T_ju_{a_n,\varepsilon}\,
             \overline{T_j\varphi}\,dw
 +\int (V+a_n+i\tau)
             u_{a_n,\varepsilon}\overline\varphi\,dw
 =\int F_{\varepsilon,\mathbf y}\overline\varphi\,dw
 =0.
\]
Choose $\chi$ equal to one on a neighbourhood of
$\mathcal O(\operatorname{supp}\varphi)$.
Weak convergence in the form domain then allows us
to pass to the limit in this identity. In particular,
the potential term is written as
\[
 \int V u_{a_n,\varepsilon}\overline\varphi\,dw
 =
 \int (V^{1/2}u_{a_n,\varepsilon})
             \overline{V^{1/2}\varphi}\,dw.
\]
We conclude that $\Gamma_{a_n}(\cdot,\mathbf y,\tau)$
solves the homogeneous equation with potential $V+a_n$.

Finally, we let $n\to\infty$. Absolute convergence
of the heat-kernel integral gives
\[
 \Gamma_{a_n}(\mathbf x,\mathbf y,\tau)
 \longrightarrow\Gamma(\mathbf x,\mathbf y,\tau)
\]
for almost every $\mathbf x$ outside the orbit of
$\mathbf y$. The uniform bound on $U$ obtained from
\eqref{eq:regularized_green_majorant} implies
\[
 \Gamma_{a_n}(\cdot,\mathbf y,\tau)
 \longrightarrow\Gamma(\cdot,\mathbf y,\tau)
 \quad\text{in }L^2(U,dw).
\]
Moreover, the local form bounds remain uniform in $n$,
by \eqref{eq:kernel_local_form_bound} and weak lower
semicontinuity. Repeating the cutoff argument above
shows that $\Gamma(\cdot,\mathbf y,\tau)$ belongs
to the local form domain.

We may therefore pass to the limit in the weak equation
for $\Gamma_{a_n}$. The only additional term vanishes:
\[
 \left|
 a_n\int_U\Gamma_{a_n}(\mathbf x,\mathbf y,\tau)
                    \overline{\varphi(\mathbf x)}\,dw(\mathbf x)
 \right|
 \leq a_n\|\Gamma_{a_n}(\cdot,\mathbf y,\tau)\|_{L^2(U,dw)}
             \|\varphi\|_{L^2(U,dw)}
 \longrightarrow0.
\]
Hence
\[
 \sum_{j=1}^N\int T_j\Gamma(\cdot,\mathbf y,\tau)
                         \overline{T_j\varphi}\,dw
 +\int(V+i\tau)\Gamma(\cdot,\mathbf y,\tau)
                         \overline\varphi\,dw=0.
\]
Since the sets $U,W$ may be chosen around any compact
subset of $\mathbb R^N\setminus\mathcal O(\{\mathbf y\})$,
this proves the assertion.
\end{proof}

\subsection{Decay estimates involving the critical function}

    Combining Propositions~\ref{propo:large_decay} and~\ref{propo:basic_decay} we get the following estimate.

\begin{theorem}\label{theorem:fundamental}
Assume $N>2$, and let $\kappa>0$ be as in
Lemma~\ref{lem:m_growth}.
There exists a constant $C>0$, independent of $\tau$,
such that, for each fixed $\tau\in\mathbb R$,
\begin{equation}\label{eq:fundamendal_est}
 |\Gamma(\mathbf x,\mathbf y,\tau)|
 \leq
 C\bigl(1+m(\mathbf x)d(\mathbf x,\mathbf y)\bigr)^{-1/(\kappa+1)}
 \frac{d(\mathbf x,\mathbf y)^2}
      {w(B(\mathbf x,d(\mathbf x,\mathbf y)))}
\end{equation}
for almost every pair $(\mathbf x,\mathbf y)$ with
$d(\mathbf x,\mathbf y)>0$.
\end{theorem}

\begin{proof}
Fix $\tau\in\mathbb R$ and choose $\mathbf y$ outside
the exceptional null sets in
Proposition~\ref{propo:basic_decay} and
Lemma~\ref{lem:resolvent_local_weak_solution}.
Put
\[
 u(\mathbf z)=\Gamma(\mathbf z,\mathbf y,\tau).
\]
Then $u$ belongs to the local form domain and solves
\[
 (-\Delta_k+V+i\tau)u=0
\]
weakly outside $\mathcal O(\{\mathbf y\})$.
Let $\mathbf x$ be a Lebesgue point of $|u|^2$ outside
this orbit at which the basic kernel bound holds, and set
\[
 r=d(\mathbf x,\mathbf y)>0.
\]
These conditions hold for almost every $\mathbf x$.

If $m(\mathbf x)r\leq10$, the factor
$(1+m(\mathbf x)r)^{-1/(\kappa+1)}$ is bounded below
by a positive constant depending only on $\kappa$.
Thus, the desired estimate follows from
Proposition~\ref{propo:basic_decay}.

Suppose now that $m(\mathbf x)r>10$, and put $R=r/10$.
Then $Rm(\mathbf x)>1$ and
\[
 \overline{U(\mathbf x,2R)}
 \subset\mathbb R^N\setminus\mathcal O(\{\mathbf y\}).
\]
We may therefore apply
Proposition~\ref{propo:large_decay} with $\kappa_1=0$.
Since $\mathbf x$ is a Lebesgue point of $|u|^2$,
\[
 |u(\mathbf x)|^2
 \leq\operatorname*{ess\,sup}_{B(\mathbf x,R)}|u|^2.
\]
Taking square roots in that proposition gives
\begin{equation}\label{eq:fundamental_local_average}
 |u(\mathbf x)|
 \leq
 C(Rm(\mathbf x))^{-1/(\kappa+1)}
 \left(
 \frac1{w(B(\mathbf x,R))}
 \int_{U(\mathbf x,2R)}|u(\mathbf z)|^2\,dw(\mathbf z)
 \right)^{1/2}.
\end{equation}
Because $m(\mathbf x)r>10$ and $R=r/10$,
\[
 (Rm(\mathbf x))^{-1/(\kappa+1)}
 \leq C(1+m(\mathbf x)r)^{-1/(\kappa+1)}.
\]

We next estimate the local average in
\eqref{eq:fundamental_local_average}.
For $\mathbf z\in U(\mathbf x,2R)$, choose $g\in G$
such that
\[
 \|\mathbf z-g(\mathbf x)\|<2R=r/5.
\]
The orbit distance satisfies
\[
 |d(\mathbf z,\mathbf y)-d(g(\mathbf x),\mathbf y)|
 \leq\|\mathbf z-g(\mathbf x)\|.
\]
Since $d(g(\mathbf x),\mathbf y)=r$, we obtain
\[
 \frac45r\leq d(\mathbf z,\mathbf y)\leq\frac65r.
\]
The doubling property and the $G$-invariance of $dw$
also imply
\[
 w(B(\mathbf z,d(\mathbf z,\mathbf y)))
\simeq w(B(g(\mathbf x),r))
 =w(B(\mathbf x,r)).
\]
Indeed, the centres $\mathbf z$ and $g(\mathbf x)$
are at distance less than $r/5$, and the radii
are comparable to $r$.

Consequently, Proposition~\ref{propo:basic_decay} gives
\[
 |u(\mathbf z)|
 \leq C\frac{r^2}{w(B(\mathbf x,r))}
\]
for almost every $\mathbf z\in U(\mathbf x,2R)$.
Moreover,
\[
 w(U(\mathbf x,2R))
 \leq\sum_{g\in G}w(B(g(\mathbf x),2R))
 =|G|w(B(\mathbf x,2R))
 \leq Cw(B(\mathbf x,R)).
\]
It follows that
\[
 \left(
 \frac1{w(B(\mathbf x,R))}
 \int_{U(\mathbf x,2R)}|u(\mathbf z)|^2\,dw(\mathbf z)
 \right)^{1/2}
 \leq C\frac{r^2}{w(B(\mathbf x,r))}.
\]
Substituting this into
\eqref{eq:fundamental_local_average} proves
\eqref{eq:fundamendal_est}.
The constants in the estimates used above are
independent of $\tau$.
\end{proof}

Furthermore, if the potential $V$ is $G$-invariant, the estimate obtained in Theorem~\ref{theorem:fundamental} can be significantly strengthened to an arbitrary polynomial decay.

\begin{theorem}\label{theorem:fundamental_invariant}
Assume $N>2$ and that $V$ is $G$-invariant.
Then, for every $\kappa_1>0$, there exists a constant
$C_{\kappa_1}>0$, independent of $\tau$, such that,
for each fixed $\tau\in\mathbb R$,
\begin{equation}\label{eq:fundamental_est_invariant}
 |\Gamma(\mathbf x,\mathbf y,\tau)|
 \leq C_{\kappa_1}
 \bigl(1+m(\mathbf x)d(\mathbf x,\mathbf y)\bigr)^{-\kappa_1}
 \frac{d(\mathbf x,\mathbf y)^2}
      {w(B(\mathbf x,d(\mathbf x,\mathbf y)))}
\end{equation}
for almost every pair $(\mathbf x,\mathbf y)$ with
$d(\mathbf x,\mathbf y)>0$.
\end{theorem}

\begin{proof}
Fix $\kappa_1>0$ and $\tau\in\mathbb R$.
As in the proof of Theorem~\ref{theorem:fundamental},
choose $\mathbf y$ outside the exceptional null sets and
put
\[
 u(\mathbf z)=\Gamma(\mathbf z,\mathbf y,\tau).
\]
Let $\mathbf x$ be a Lebesgue point of $|u|^2$ outside
$\mathcal O(\{\mathbf y\})$ at which the basic kernel
bound holds, and write
\[
 r=d(\mathbf x,\mathbf y)>0.
\]

If $m(\mathbf x)r\leq10$, the factor
$(1+m(\mathbf x)r)^{-\kappa_1}$ is bounded below by
a positive constant depending only on $\kappa_1$.
Hence \eqref{eq:fundamental_est_invariant} follows
from Proposition~\ref{propo:basic_decay}.

Suppose now that $m(\mathbf x)r>10$, and set $R=r/10$.
Then
\[
 Rm(\mathbf x)>1,
 \qquad
 \overline{U(\mathbf x,2R)}
 \subset\mathbb R^N\setminus\mathcal O(\{\mathbf y\}).
\]
By Lemma~\ref{lem:resolvent_local_weak_solution},
$u$ is a local weak solution on this neighbourhood.
We may therefore apply
Proposition~\ref{propo:large_decay} with parameter
$\kappa_1$.

The $G$-invariance of $V$ and $dw$ implies
$m(g\mathbf x)=m(\mathbf x)$ for every $g\in G$.
Consequently,
\[
 \max_{g\in G}(Rm(g\mathbf x))^{-2\kappa_1}
 =(Rm(\mathbf x))^{-2\kappa_1}.
\]
Taking square roots in
Proposition~\ref{propo:large_decay}, and using the
Lebesgue point property at $\mathbf x$, gives
\[
 |u(\mathbf x)|
 \leq C_{\kappa_1}
 (Rm(\mathbf x))^{-1/(\kappa+1)-\kappa_1}
 \left(
 \frac1{w(B(\mathbf x,R))}
 \int_{U(\mathbf x,2R)}|u(\mathbf z)|^2\,dw(\mathbf z)
 \right)^{1/2}.
\]
Since $Rm(\mathbf x)>1$, we may discard the additional
factor $(Rm(\mathbf x))^{-1/(\kappa+1)}$.
Moreover, $R=r/10$ and $m(\mathbf x)r>10$ imply
\[
 (Rm(\mathbf x))^{-\kappa_1}
 \leq C_{\kappa_1}(1+m(\mathbf x)r)^{-\kappa_1}.
\]

It remains to estimate the local average.
For $\mathbf z\in U(\mathbf x,2R)$, the same geometric
comparisons as in the proof of
Theorem~\ref{theorem:fundamental} give
\[
 d(\mathbf z,\mathbf y)\simeq r,
 \qquad
 w(B(\mathbf z,d(\mathbf z,\mathbf y)))
 \simeq w(B(\mathbf x,r)).
\]
Thus Proposition~\ref{propo:basic_decay} implies
\[
 |u(\mathbf z)|
 \leq C\frac{r^2}{w(B(\mathbf x,r))}
\]
for almost every $\mathbf z\in U(\mathbf x,2R)$.
Using also
\[
 w(U(\mathbf x,2R))
 \leq |G|w(B(\mathbf x,2R))
 \leq Cw(B(\mathbf x,R)),
\]
we obtain
\[
 \left(
 \frac1{w(B(\mathbf x,R))}
 \int_{U(\mathbf x,2R)}|u(\mathbf z)|^2\,dw(\mathbf z)
 \right)^{1/2}
 \leq C\frac{r^2}{w(B(\mathbf x,r))}.
\]
Combining these estimates proves
\eqref{eq:fundamental_est_invariant}.
All constants are independent of $\tau$.
\end{proof}

\section{Potential estimates and fractional powers}

    \begin{theorem}\label{thm:V_boundedness}
Assume $q > \frac{\mathbf{N}}{2}$, $N>2$, and, $V \in {\rm{RH}}^q(dw)$, and that $V$ is $G$-invariant. Then for any $1 \le p \le q$, there exists a constant $C > 0$ such that for all $f \in L^p(dw)$ we have
\[
\|V(-\Delta_k + V)^{-1}f\|_{L^p(dw)} \le C \|f\|_{L^p(dw)}.
\]
\end{theorem}

    The next lemma is an easy consequence of the basic properties of the auxiliary function $m$ (see Lemma~\ref{lem:m_growth}).

    \subsection{Local estimates for the potential operator}

\begin{lemma}\label{lem:x_y_switch}
There exists a constant $C_1>0$ such that
\[
 \|\mathbf x-\mathbf y\|>\frac1{m(\mathbf x)}
 \quad\Longrightarrow\quad
 \|\mathbf x-\mathbf y\|>\frac{C_1}{m(\mathbf y)}.
\]
\end{lemma}

\begin{proof}
Put $r=\|\mathbf x-\mathbf y\|$ and $t=m(\mathbf x)r>1$.
By \eqref{eq:Shen_C} and $1+t\leq2t$,
\[
 rm(\mathbf y)
 \geq C^{-1}t(1+t)^{-\kappa/(\kappa+1)}
 \geq C^{-1}2^{-\kappa/(\kappa+1)}
             t^{1/(\kappa+1)}
 > C^{-1}2^{-\kappa/(\kappa+1)}.
\]
Thus the assertion holds with
$C_1=C^{-1}2^{-\kappa/(\kappa+1)}$.
\end{proof}

The following lemma is a standard estimate which holds on the spaces of homogeneous type.

\begin{lemma}\label{lem:homogeneous}
    Let $\delta>0$. There is a constant $C>0$ such that for all $\mathbf{x} \in \mathbb{R}^N$ and $r>0$ we have:
\begin{enumerate}[(A)]
    \item \label{numitem:outside_orbit} $\int_{\mathcal{O}(B(\mathbf{x},r))^{c}}\left(\frac{d(\mathbf{x},\mathbf{y})}{r}\right)^{-\delta}w(B(\mathbf{x},d(\mathbf{x},\mathbf{y})))^{-1}\,dw(\mathbf{y}) \leq C$;
    \item \label{numitem:inside_orbit} $\int_{\mathcal{O}(B(\mathbf{x},r))}\left(\frac{d(\mathbf{x},\mathbf{y})}{r}\right)^{\delta}w(B(\mathbf{x},d(\mathbf{x},\mathbf{y})))^{-1}\,dw(\mathbf{y}) \leq C$.
\end{enumerate}
\end{lemma}

\begin{lemma}\label{lem:L_1}
    Assume $q > \frac{\mathbf{N}}{2}$, $N>2$, and and $V \in {\rm{RH}}^{q}(dw)$. Let $\tau \in \mathbb{R}$ and $g \in G$. There is a constant $C>0$ such that for all $f \in L^1(dw)$ we have
    \begin{equation}
        \int_{\mathbb{R}^N}|V(\mathbf{x})|\int_{\{\mathbf{y}\in\mathbb{R}^N\::\: d(\mathbf{x},\mathbf{y})=\|\mathbf{x}-g(\mathbf{y})\|<\frac{1}{m(\mathbf{x})}\}}|\Gamma(\mathbf{x},\mathbf{y},\tau)||f(\mathbf{y})|\,dw(\mathbf{y})\,dw(\mathbf{x}) \leq C\|f\|_{L^1(dw)}.
    \end{equation}
\end{lemma}

\begin{proof}
Put
\[
 \delta=2-\frac{\mathbf N}{q}>0.
\]
All constants below are independent of $\tau$ and $g$.

If $r=\|\mathbf x-g(\mathbf y)\|<m(\mathbf x)^{-1}$,
then \eqref{eq:Shen_A} gives
$m(g(\mathbf y))\leq C_0m(\mathbf x)$, where $C_0\geq1$.
Consequently,
\[
 r<\frac{C_0}{m(g(\mathbf y))}.
\]
On the integration region in the statement, we also have
$d(\mathbf x,\mathbf y)=r$. Hence
Proposition~\ref{propo:basic_decay} and Tonelli's theorem
show that the left-hand side is bounded by
\[
 C\int_{\mathbb R^N}|f(\mathbf y)|\,J(g(\mathbf y))\,dw(\mathbf y),
\]
where
\[
 J(\mathbf z)
 :=
 \int_{\|\mathbf x-\mathbf z\|<C_0/m(\mathbf z)}
 \frac{\|\mathbf x-\mathbf z\|^2}
      {w(B(\mathbf x,\|\mathbf x-\mathbf z\|))}
 V(\mathbf x)\,dw(\mathbf x).
\]
It therefore suffices to prove that $J(\mathbf z)\leq C$
uniformly in $\mathbf z$.

Fix $\mathbf z$, and write
\[
 \rho=\frac1{m(\mathbf z)},\qquad R=C_0\rho,
 \qquad
 A_{\mathbf z}(s)
 =\frac{s^2}{w(B(\mathbf z,s))}
       \int_{B(\mathbf z,s)}V\,dw.
\]
The critical-radius normalization gives
$A_{\mathbf z}(\rho)\leq C$.
Since $V\,dw$ is doubling and $C_0$ is fixed, it follows
that
\begin{equation}\label{eq:near_critical_average}
 A_{\mathbf z}(R)\leq C.
\end{equation}

For $0<s\leq R$, Hölder's inequality and
$V\in{\rm RH}^q(dw)$ yield
\[
 \begin{split}
 \int_{B(\mathbf z,s)}V\,dw
 &\leq
 w(B(\mathbf z,s))^{1-1/q}
 \left(\int_{B(\mathbf z,R)}V^q\,dw\right)^{1/q}\leq C
 \left(\frac{w(B(\mathbf z,s))}
            {w(B(\mathbf z,R))}\right)^{1-1/q}
 \int_{B(\mathbf z,R)}V\,dw.
 \end{split}
\]
Multiplying the preceding estimate by
$s^2/w(B(\mathbf z,s))$ and using
\[
 \int_{B(\mathbf z,R)}V\,dw
 =\frac{w(B(\mathbf z,R))}{R^2}A_{\mathbf z}(R),
\]
we obtain
\[
 \begin{split}
 A_{\mathbf z}(s)
 &\leq
 C\frac{s^2}{w(B(\mathbf z,s))}
 \left(\frac{w(B(\mathbf z,s))}
            {w(B(\mathbf z,R))}\right)^{1-1/q}
 \frac{w(B(\mathbf z,R))}{R^2}A_{\mathbf z}(R)\\
 &=C\left(\frac sR\right)^2
 \left(\frac{w(B(\mathbf z,R))}
            {w(B(\mathbf z,s))}\right)^{1/q}
 A_{\mathbf z}(R).
 \end{split}
\]
Since $0<s\leq R$, the upper volume growth estimate
in \eqref{eq:growth} gives
\[
 \frac{w(B(\mathbf z,R))}{w(B(\mathbf z,s))}
 \leq C\left(\frac Rs\right)^{\mathbf N}.
\]
Consequently,
\[
 \begin{split}
 A_{\mathbf z}(s)
 &\leq C\left(\frac sR\right)^2
          \left(\frac Rs\right)^{\mathbf N/q}
          A_{\mathbf z}(R)=C\left(\frac sR\right)^{2-\mathbf N/q}
          A_{\mathbf z}(R)
 =C\left(\frac sR\right)^\delta A_{\mathbf z}(R).
 \end{split}
\]
Finally, \eqref{eq:near_critical_average} implies
\[
 A_{\mathbf z}(s)\leq C\left(\frac sR\right)^\delta,
 \qquad 0<s\leq R.
\]

Now set $r_j=2^{-j}R$ and decompose $B(\mathbf z,R)$
into the annuli
\[
 E_j=\{\mathbf x:r_j/2\leq\|\mathbf x-\mathbf z\|<r_j\},
 \qquad j\geq0.
\]
For $\mathbf x\in E_j$, the doubling property gives
\[
 w(B(\mathbf x,\|\mathbf x-\mathbf z\|))
 \simeq w(B(\mathbf z,r_j)).
\]
Therefore,
\[
 \begin{split}
 J(\mathbf z)
 &\leq C\sum_{j=0}^\infty
 \frac{r_j^2}{w(B(\mathbf z,r_j))}
       \int_{B(\mathbf z,r_j)}V\,dw=C\sum_{j=0}^\infty A_{\mathbf z}(r_j)
 \leq C\sum_{j=0}^\infty 2^{-j\delta}
 \leq C.
 \end{split}
\]
The series converges because $\delta>0$.
Substituting this uniform bound into the initial
Tonelli estimate proves the assertion.
\end{proof}

\begin{lemma}\label{lem:L_q_large}
    Assume $q > \frac{\mathbf{N}}{2}$, $N>2$, and $V \in {\rm{RH}}^{q}(dw)$. Let $\tau \in \mathbb{R}$ and $g \in G$. There is a constant $C>0$ such that for all $f \in L^q(dw)$ we have
    \begin{equation}\label{eq:L_q_large}
        \int_{\mathbb{R}^N}|V(\mathbf{x})|^q\left|\int_{\{\mathbf{y}\in\mathbb{R}^N\::\: d(\mathbf{x},\mathbf{y})=\|\mathbf{x}-g(\mathbf{y})\|<\frac{1}{m(\mathbf{x})}\}}\Gamma(\mathbf{x},\mathbf{y},\tau)f(\mathbf{y})\,dw(\mathbf{y})\right|^q\,dw(\mathbf{x}) \leq C\|f\|_{L^q(dw)}^q.
    \end{equation}
\end{lemma}

\begin{proof}
Choose
\[
 0<\varepsilon<2-\frac{\mathbf N}{q},
 \qquad
 \alpha=q(2-\varepsilon)>\mathbf N,
 \qquad q'=\frac q{q-1}.
\]
Write
\[
 E_{\mathbf x}
 =\{\mathbf y:\|\mathbf x-g(\mathbf y)\|<m(\mathbf x)^{-1}\},
 \qquad
 r=\|\mathbf x-g(\mathbf y)\|.
\]
Let $T_gf$ denote the inner integral in
\eqref{eq:L_q_large}.
On its integration region, $d(\mathbf x,\mathbf y)=r$.
By Proposition~\ref{propo:basic_decay}, we may therefore
estimate the kernel and then enlarge the integration
region to obtain
\[
 |T_gf(\mathbf x)|
 \leq C\int_{E_{\mathbf x}}
       \frac{r^2}{w(B(\mathbf x,r))}
       |f(\mathbf y)|\,dw(\mathbf y).
\]
Hölder's inequality gives
\[
 \begin{split}
 |T_gf(\mathbf x)|^q
 &\leq C
 \left(
 \int_{E_{\mathbf x}}
 \frac{r^{q(2-\varepsilon)}}{w(B(\mathbf x,r))}
 |f(\mathbf y)|^q\,dw(\mathbf y)
 \right)
 \left(
 \int_{E_{\mathbf x}}
 \frac{r^{q'\varepsilon}}{w(B(\mathbf x,r))}
 \,dw(\mathbf y)
 \right)^{q/q'}.
 \end{split}
\]
 Applying
Lemma~\ref{lem:homogeneous}(\ref{numitem:inside_orbit}),
we obtain
\[
 \left(
 \int_{E_{\mathbf x}}
 \frac{r^{q'\varepsilon}}{w(B(\mathbf x,r))}
 \,dw(\mathbf y)
 \right)^{q/q'}
 \leq C m(\mathbf x)^{-q\varepsilon}.
\]
Consequently, Tonelli's theorem yields
\[
 \int_{\mathbb R^N}V(\mathbf x)^q|T_gf(\mathbf x)|^q
 \,dw(\mathbf x)
 \leq C\int_{\mathbb R^N}|f(\mathbf y)|^q
                       H(g(\mathbf y))\,dw(\mathbf y),
\]
where
\[
 H(\mathbf z)
 :=
 \int_{\|\mathbf x-\mathbf z\|<m(\mathbf x)^{-1}}
 \frac{m(\mathbf x)^{-q\varepsilon}
       \|\mathbf x-\mathbf z\|^\alpha}
      {w(B(\mathbf x,\|\mathbf x-\mathbf z\|))}
 V(\mathbf x)^q\,dw(\mathbf x).
\]
It remains to bound $H(\mathbf z)$ uniformly in
$\mathbf z$.

Fix $\mathbf z$ and put $\rho=m(\mathbf z)^{-1}$.
On the integration region defining $H(\mathbf z)$,
\eqref{eq:Shen_A} implies
\[
 m(\mathbf x)\asymp m(\mathbf z),
 \qquad
 \|\mathbf x-\mathbf z\|<C_0\rho,
\]
where $C_0\geq1$ is fixed.
Replacing $m(\mathbf x)^{-q\varepsilon}$ on this
region and then enlarging the region, we obtain
\[
 H(\mathbf z)
 \leq C\rho^{q\varepsilon}
 \int_{\|\mathbf x-\mathbf z\|<R}
 \frac{\|\mathbf x-\mathbf z\|^\alpha}
      {w(B(\mathbf x,\|\mathbf x-\mathbf z\|))}
 V(\mathbf x)^q\,dw(\mathbf x),
 \qquad R=C_0\rho.
\]

For $r=\|\mathbf x-\mathbf z\|\in(0,R)$,
comparison of balls and the upper volume growth estimate
give
\[
 \frac1{w(B(\mathbf x,r))}
 \leq \frac{C}{w(B(\mathbf z,r))}
 \leq
 \frac{C}{w(B(\mathbf z,R))}
 \left(\frac Rr\right)^{\mathbf N}.
\]
Since $\alpha>\mathbf N$, it follows that
\[
 \frac{r^\alpha}{w(B(\mathbf x,r))}
 \leq
 \frac{CR^{\mathbf N}r^{\alpha-\mathbf N}}
      {w(B(\mathbf z,R))}
 \leq \frac{CR^\alpha}{w(B(\mathbf z,R))}.
\]
Thus, using $\rho\leq R$ and
$q\varepsilon+\alpha=2q$, we obtain
\[
 H(\mathbf z)
 \leq
 C\frac{R^{2q}}{w(B(\mathbf z,R))}
       \int_{B(\mathbf z,R)}V^q\,dw.
\]
The reverse Hölder inequality now yields
\[
 H(\mathbf z)
 \leq C
 \left(
 \frac{R^2}{w(B(\mathbf z,R))}
       \int_{B(\mathbf z,R)}V\,dw
 \right)^q.
\]
The expression in parentheses is uniformly bounded:
at radius $\rho=m(\mathbf z)^{-1}$ this follows
from the critical-radius normalization, and passage
to $R=C_0\rho$ follows from the doubling property
of $V\,dw$.
Hence $H(\mathbf z)\leq C$ uniformly in $\mathbf z$.

Substituting this bound into the Tonelli estimate gives
\[
 \int_{\mathbb R^N}V(\mathbf x)^q|T_gf(\mathbf x)|^q
 \,dw(\mathbf x)
 \leq C\|f\|_{L^q(dw)}^q,
\]
as required.
\end{proof}

\subsection{Far-field estimates and boundedness of \texorpdfstring{$VL^{-1}$}{VL-1}}

  \begin{lemma}\label{lem:L_p_large}
    Assume $q > \frac{\mathbf{N}}{2}$, $N>2$, and, $V \in {\rm{RH}}^{q}(dw)$, and that $V$ is $G$-invariant. Let $\tau \in \mathbb{R}$, $g \in G$, and $1 \leq p \leq q$. There exists a constant $C>0$ such that for all $f \in L^p(dw)$ we have
    \begin{equation}\label{eq:L_p_large}
        \int_{\mathbb{R}^N}|V(\mathbf{x})|^p\left|\int_{\{\mathbf{y}\in\mathbb{R}^N\::\: d(\mathbf{x},\mathbf{y})=\|\mathbf{x}-g(\mathbf{y})\| \geq \frac{1}{m(\mathbf{x})}\}}\Gamma(\mathbf{x},\mathbf{y},\tau)f(\mathbf{y})\,dw(\mathbf{y})\right|^p\,dw(\mathbf{x}) \leq C\|f\|_{L^p(dw)}^p.
    \end{equation}
\end{lemma}

\begin{proof}
Fix $\ell>0$, to be chosen later, and apply
Theorem~\ref{theorem:fundamental_invariant}
with decay exponent $2\ell$.
Let $T_gf$ denote the inner integral in
\eqref{eq:L_p_large}, and put
\[
 E_{\mathbf x}
 =\{\mathbf y:\|\mathbf x-g(\mathbf y)\|
                  \geq m(\mathbf x)^{-1}\},
 \qquad
 r=\|\mathbf x-g(\mathbf y)\|.
\]
On the integration region defining $T_gf$, we have
$d(\mathbf x,\mathbf y)=r$.
Moreover, the doubling property gives
\[
 w(B(\mathbf x,r))\simeq w(B(g(\mathbf y),r)).
\]
We may therefore estimate the kernel on this region
and then enlarge the region to $E_{\mathbf x}$,
obtaining
\[
 |T_gf(\mathbf x)|
 \leq C_\ell\int_{E_{\mathbf x}}
 \frac{r^2|f(\mathbf y)|}
 {w(B(g(\mathbf y),r))(1+m(\mathbf x)r)^{2\ell}}
 \,dw(\mathbf y).
\]

Suppose first that $1<p\leq q$, and put $p'=p/(p-1)$.
Hölder's inequality yields
\[
 |T_gf(\mathbf x)|^p
 \leq C_\ell A_f(\mathbf x)B(\mathbf x)^{p/p'},
\]
where
\[
 A_f(\mathbf x)
 =
 \int_{E_{\mathbf x}}
 \frac{r^{2p}|f(\mathbf y)|^p}
 {w(B(g(\mathbf y),r))(1+m(\mathbf x)r)^{\ell p}}
 \,dw(\mathbf y)
\]
and
\[
 B(\mathbf x)
 =
 \int_{E_{\mathbf x}}
 \frac{dw(\mathbf y)}
 {w(B(g(\mathbf y),r))(1+m(\mathbf x)r)^{\ell p'}}.
\]

We first bound $B(\mathbf x)$ uniformly in $\mathbf x$.
Set $\rho_{\mathbf x}=m(\mathbf x)^{-1}$.
The change of variables $\mathbf z=g(\mathbf y)$,
the $G$-invariance of $dw$, and comparison of balls give
\[
 B(\mathbf x)
 \leq C
 \int_{\|\mathbf x-\mathbf z\|\geq\rho_{\mathbf x}}
 \frac{dw(\mathbf z)}
 {w(B(\mathbf x,\|\mathbf x-\mathbf z\|))
  (1+m(\mathbf x)\|\mathbf x-\mathbf z\|)^{\ell p'}}.
\]
Decomposing this region into dyadic annuli and using
the doubling property, we obtain
\[
 \begin{split}
 B(\mathbf x)
 &\leq C\sum_{j=0}^\infty
 2^{-j\ell p'}
 \frac{w(B(\mathbf x,2^{j+1}\rho_{\mathbf x}))}
      {w(B(\mathbf x,2^j\rho_{\mathbf x}))}\leq C\sum_{j=0}^\infty 2^{-j\ell p'}
 \leq C.
 \end{split}
\]
Multiplying the Hölder estimate by $V(\mathbf x)^p$,
integrating, and applying Tonelli's theorem therefore
gives
\begin{equation}\label{eq:large_kernel_reduction}
\begin{split}
 &\int_{\mathbb R^N}V(\mathbf x)^p|T_gf(\mathbf x)|^p
 \,dw(\mathbf x)\\
 &\quad\leq C_\ell\int_{\mathbb R^N}|f(\mathbf y)|^p
 \left(
 \int_{\|\mathbf x-g(\mathbf y)\|\geq m(\mathbf x)^{-1}}
 \frac{\|\mathbf x-g(\mathbf y)\|^{2p}V(\mathbf x)^p}
 {w(B(g(\mathbf y),\|\mathbf x-g(\mathbf y)\|))
  (1+m(\mathbf x)\|\mathbf x-g(\mathbf y)\|)^{\ell p}}
 \,dw(\mathbf x)
 \right)\,dw(\mathbf y).
\end{split}
\end{equation}
For $p=1$, the same estimate follows directly from
the initial kernel bound and Tonelli's theorem, since
\[
 (1+m(\mathbf x)r)^{-2\ell}
 \leq(1+m(\mathbf x)r)^{-\ell}.
\]
Thus \eqref{eq:large_kernel_reduction} holds for every
$1\leq p\leq q$.

We now estimate the inner integral uniformly in
$\mathbf y$. Put $\mathbf z=g(\mathbf y)$.
Lemma~\ref{lem:x_y_switch} gives a constant $c_0>0$
such that
\[
 \|\mathbf x-\mathbf z\|\geq m(\mathbf x)^{-1}
 \quad\Longrightarrow\quad
 \|\mathbf x-\mathbf z\|\geq c_0m(\mathbf z)^{-1}.
\]
Furthermore, \eqref{eq:Shen_B} implies
\[
 m(\mathbf z)
 \leq C m(\mathbf x)
       (1+m(\mathbf x)\|\mathbf x-\mathbf z\|)^\kappa,
\]
and hence
\[
 1+m(\mathbf z)\|\mathbf x-\mathbf z\|
 \leq C
 (1+m(\mathbf x)\|\mathbf x-\mathbf z\|)^{\kappa+1}.
\]
Consequently, setting
\[
 \theta=\frac{\ell}{\kappa+1},
\]
we obtain
\[
 (1+m(\mathbf x)\|\mathbf x-\mathbf z\|)^{-\ell p}
 \leq C_\ell
 (1+m(\mathbf z)\|\mathbf x-\mathbf z\|)^{-\theta p}.
\]
It therefore suffices to show that
\begin{equation}\label{eq:dyadic_sum}
 \int_{\|\mathbf x-\mathbf z\|\geq c_0m(\mathbf z)^{-1}}
 \frac{\|\mathbf x-\mathbf z\|^{2p}V(\mathbf x)^p}
 {w(B(\mathbf z,\|\mathbf x-\mathbf z\|))
  (1+m(\mathbf z)\|\mathbf x-\mathbf z\|)^{\theta p}}
 \,dw(\mathbf x)
 \leq C
\end{equation}
uniformly in $\mathbf z$.

Put $\rho=m(\mathbf z)^{-1}$ and define
\[
 A_j(\mathbf z)
 =
 \{\mathbf x:2^jc_0\rho\leq\|\mathbf x-\mathbf z\|
                      <2^{j+1}c_0\rho\},
 \qquad
 R_j=2^{j+1}c_0\rho,
 \qquad j\geq0.
\]
On $A_j(\mathbf z)$, the doubling property gives
\[
 w(B(\mathbf z,\|\mathbf x-\mathbf z\|))
 \simeq w(B(\mathbf z,R_j)),
\]
while
\[
 (1+m(\mathbf z)\|\mathbf x-\mathbf z\|)^{-\theta p}
 \leq C2^{-j\theta p}.
\]
Thus the integral in \eqref{eq:dyadic_sum} is bounded by
\begin{equation}\label{eq:large_annuli_potential_sum}
 C\sum_{j=0}^\infty 2^{-j\theta p}
 \frac{R_j^{2p}}{w(B(\mathbf z,R_j))}
 \int_{B(\mathbf z,R_j)}V^p\,dw.
\end{equation}

For every ball $B$, Hölder's inequality and
$V\in{\rm RH}^q(dw)$ imply
\[
 \left(\frac1{w(B)}\int_B V^p\,dw\right)^{1/p}
 \leq
 \left(\frac1{w(B)}\int_B V^q\,dw\right)^{1/q}
 \leq C\frac1{w(B)}\int_B V\,dw.
\]
Consequently,
\[
 \frac{R_j^{2p}}{w(B(\mathbf z,R_j))}
 \int_{B(\mathbf z,R_j)}V^p\,dw
 \leq C
 \left(
 \frac{R_j^2}{w(B(\mathbf z,R_j))}
 \int_{B(\mathbf z,R_j)}V\,dw
 \right)^p.
\]

To estimate the last expression, write $d\mu=V\,dw$.
The critical-radius normalization gives
\[
 \frac{\rho^2}{w(B(\mathbf z,\rho))}
       \mu(B(\mathbf z,\rho))\leq C.
\]
By Lemma~\ref{lem:mu_doubling}, there exists $\nu>0$
such that
\[
 \mu(B(\mathbf z,R))
 \leq C\left(\frac R\rho\right)^\nu
          \mu(B(\mathbf z,\rho)),
 \qquad R\geq\rho.
\]
Since $w(B(\mathbf z,R))\geq w(B(\mathbf z,\rho))$,
we obtain
\[
 \frac{R^2}{w(B(\mathbf z,R))}
       \int_{B(\mathbf z,R)}V\,dw
 \leq C\left(\frac R\rho\right)^{2+\nu},
 \qquad R\geq\rho.
\]
For $0<R<\rho$, Lemma~\ref{lem:Rr} gives
\[
 \frac{R^2}{w(B(\mathbf z,R))}
       \int_{B(\mathbf z,R)}V\,dw
 \leq C\left(\frac R\rho\right)^{2-\mathbf N/q}
 \leq C.
\]
Therefore, with $N_0=2+\nu$,
\[
 \frac{R^2}{w(B(\mathbf z,R))}
       \int_{B(\mathbf z,R)}V\,dw
 \leq C(1+Rm(\mathbf z))^{N_0},
 \qquad R>0.
\]
Since $R_jm(\mathbf z)=2^{j+1}c_0$, it follows that
\[
 \frac{R_j^{2p}}{w(B(\mathbf z,R_j))}
 \int_{B(\mathbf z,R_j)}V^p\,dw
 \leq C2^{jpN_0}.
\]
Substituting this into
\eqref{eq:large_annuli_potential_sum}, we obtain
\[
 C\sum_{j=0}^\infty 2^{-jp(\theta-N_0)}.
\]
Choose
\[
 \ell>(\kappa+1)N_0.
\]
Then $\theta>N_0$, so the series converges and proves
\eqref{eq:dyadic_sum}.
Finally, \eqref{eq:large_kernel_reduction} yields
\[
 \int_{\mathbb R^N}V(\mathbf x)^p|T_gf(\mathbf x)|^p
 \,dw(\mathbf x)
 \leq C\|f\|_{L^p(dw)}^p.
\]
The constant is independent of $\tau$ and $g$.
\end{proof}

\begin{remark}
The proof of Lemma~\ref{lem:L_p_large} uses
$G$-invariance through the arbitrary-order kernel decay
in Theorem~\ref{theorem:fundamental_invariant}.
This allows us to choose
$\ell>(\kappa+1)N_0$ and sum the resulting dyadic series.
The fixed decay exponent in
Theorem~\ref{theorem:fundamental} alone does not provide
this summability argument.
\end{remark}

The preceding lemmas provide the near-field and far-field
bounds needed to prove Theorem~\ref{thm:V_boundedness}.
For the near-field contribution, we interpolate between
the $L^1(dw)$ and $L^q(dw)$ estimates, while the far-field
contribution is bounded directly throughout the range
$1\leq p\leq q$.
Summing over the finitely many group elements then yields
the boundedness of $VL^{-1}$.
This estimate will also be the starting point for the
remaining results of this section: the identity
$-\Delta_kL^{-1}=I-VL^{-1}$ gives the second-order
derivative bounds, and analytic interpolation yields
the fractional potential estimates and the initial
range of boundedness for the first-order Riesz transform.

\begin{proof}[Proof of Theorem~\ref{thm:V_boundedness}]
Write
\[
 L=-\Delta_k+V,
 \qquad
 \Gamma(\mathbf x,\mathbf y)=\Gamma(\mathbf x,\mathbf y,0).
\]
Recall that $\Gamma\geq0$, since it is the time integral
of the nonnegative heat kernel of $L$.

For $g\in G$, define
\[
 \begin{split}
 E_{1,g}(\mathbf x)
 &=\left\{\mathbf y:
 d(\mathbf x,\mathbf y)=\|\mathbf x-g(\mathbf y)\|
 <m(\mathbf x)^{-1}\right\},\\
 E_{2,g}(\mathbf x)
 &=\left\{\mathbf y:
 d(\mathbf x,\mathbf y)=\|\mathbf x-g(\mathbf y)\|
 \geq m(\mathbf x)^{-1}\right\},
 \end{split}
\]
and introduce the operators
\[
 S_{i,g}f(\mathbf x)
 =V(\mathbf x)
  \int_{E_{i,g}(\mathbf x)}
       \Gamma(\mathbf x,\mathbf y)f(\mathbf y)\,dw(\mathbf y),
 \qquad i=1,2.
\]
These operators are initially considered on bounded,
compactly supported functions.

By Lemma~\ref{lem:L_1},
\[
 \|S_{1,g}f\|_{L^1(dw)}
 \leq C\|f\|_{L^1(dw)},
\]
whereas Lemma~\ref{lem:L_q_large} gives
\[
 \|S_{1,g}f\|_{L^q(dw)}
 \leq C\|f\|_{L^q(dw)}.
\]
The Riesz--Thorin interpolation theorem therefore yields
\[
 \|S_{1,g}f\|_{L^p(dw)}
 \leq C_p\|f\|_{L^p(dw)},
 \qquad 1\leq p\leq q.
\]
For the remaining operators,
Lemma~\ref{lem:L_p_large}, applied with $\tau=0$,
gives directly
\[
 \|S_{2,g}f\|_{L^p(dw)}
 \leq C_p\|f\|_{L^p(dw)},
 \qquad 1\leq p\leq q.
\]
Here we use the assumed $G$-invariance of $V$.

Now let $f\in C_c^\infty(\mathbb R^N)$ and define
\[
 Tf(\mathbf x)
 =V(\mathbf x)\int_{\mathbb R^N}
       \Gamma(\mathbf x,\mathbf y)f(\mathbf y)\,dw(\mathbf y).
\]
For every pair $(\mathbf x,\mathbf y)$, at least one
element $g\in G$ realizes the minimum defining
$d(\mathbf x,\mathbf y)$.
Thus the sets $E_{i,g}(\mathbf x)$ cover $\mathbb R^N$.
They need not be disjoint, but positivity of $\Gamma$
and $V$ gives
\[
 |Tf(\mathbf x)|
 \leq\sum_{g\in G}
       \bigl(S_{1,g}|f|(\mathbf x)
             +S_{2,g}|f|(\mathbf x)\bigr).
\]
The triangle inequality in $L^p(dw)$ and the preceding
bounds imply
\[
 \begin{split}
 \|Tf\|_{L^p(dw)}
 &\leq\sum_{g\in G}
 \left(
 \|S_{1,g}|f|\|_{L^p(dw)}
 +\|S_{2,g}|f|\|_{L^p(dw)}
 \right)\leq C_p\|f\|_{L^p(dw)},
 \qquad 1\leq p\leq q,
 \end{split}
\]
where the finite factor $|G|$ is absorbed into $C_p$.

Since $C_c^\infty(\mathbb R^N)$ is dense in $L^p(dw)$
for $1\leq p<\infty$, $T$ extends uniquely to a bounded
operator on $L^p(dw)$ for every $1\leq p\leq q$.
This extension is the operator denoted by $VL^{-1}$,
and it satisfies
\[
 \|VL^{-1}f\|_{L^p(dw)}
 \leq C_p\|f\|_{L^p(dw)}.
\]
\end{proof}

\subsection{Second-order Dunkl derivatives}

We write
\[
 \nabla_k^2f=(T_iT_jf)_{1\leq i,j\leq N}
\]
for the Dunkl Hessian and use the convention
\[
 \|\nabla_k^2f\|_{L^p(dw)}
 :=\sum_{i,j=1}^N\|T_iT_jf\|_{L^p(dw)}.
\]

\begin{lemma}\label{lem:Riesz_second_order}
Let $1<p<\infty$.
For every $1\leq i,j\leq N$, the second-order
Dunkl--Riesz transform
\[
 S_{ij}=T_iT_j(-\Delta_k)^{-1},
\]
initially defined on $L^2(dw)$ by the multiplier
$-\xi_i\xi_j/\|\boldsymbol\xi\|^2$, extends to a bounded
operator on $L^p(dw)$.
Moreover, if $v\in L^p(dw)$ and
$-\Delta_kv\in L^p(dw)$ in the distributional sense, then
\[
 T_iT_jv=S_{ij}(-\Delta_kv)
\]
and
\[
 \|T_iT_jv\|_{L^p(dw)}
 \leq C_p\|-\Delta_kv\|_{L^p(dw)}.
\]
\end{lemma}

\begin{proof}
The first-order Dunkl--Riesz transforms
\[
 R_i=T_i(-\Delta_k)^{-1/2}
\]
are bounded on $L^p(dw)$ for $1<p<\infty$;
see~\cite{AS}.
Their Dunkl multipliers are
$i\xi_i/\|\boldsymbol\xi\|$.
Multiplication of these symbols gives
\[
 S_{ij}=R_iR_j
\]
on $L^2(dw)$ and hence, by density, on $L^p(dw)$.
This proves the boundedness of $S_{ij}$.

For completeness, put $A=-\Delta_k$ and suppose that
$v,Av\in L^p(dw)$.
For $b>0$, the free resolvent identity gives
\[
 v=(A+bI)^{-1}(Av+bv).
\]
Applying the corresponding Dunkl multipliers yields,
in the distributional sense,
\[
 T_iT_jv
 =S_{ij}A(A+bI)^{-1}(Av+bv).
\]
As $b\downarrow0$, the operators $A(A+bI)^{-1}$
converge strongly to the identity on $L^p(dw)$ and
are uniformly bounded there. This follows from
\[
 A(A+bI)^{-1}=I-b(A+bI)^{-1}
\]
and the strong decay of the heat semigroup on
$L^p(dw)$ for $1<p<\infty$.
Passing to the limit proves
$T_iT_jv=S_{ij}Av$ and the asserted estimate.
\end{proof}

\begin{theorem}\label{thm:Riesz_boundedness}
Assume $N>2$, $q>\mathbf N/2$, and
$V\in{\rm RH}^q(dw)$.
Suppose also that $V$ is $G$-invariant.
Then, for every $1<p\leq q$, the operator
$\nabla_k^2L^{-1}$, initially defined through the
Green kernel on $C_c^\infty(\mathbb R^N)$, extends
to a bounded operator on $L^p(dw)$.
In particular,
\[
 \|\nabla_k^2L^{-1}f\|_{L^p(dw)}
 \leq C_p\|f\|_{L^p(dw)}.
\]
\end{theorem}

\begin{proof}
Fix $1<p\leq q$.
We first work with $f\in C_c^\infty(\mathbb R^N)$.
For $a>0$, put
\[
 u_a=(L+aI)^{-1}f,
 \qquad
 u(\mathbf x)=\int_{\mathbb R^N}
       \Gamma(\mathbf x,\mathbf y,0)f(\mathbf y)\,dw(\mathbf y).
\]
We shall estimate the second-order derivatives of $u_a$
uniformly in $a$ and then let $a\downarrow0$.

Since the semigroup generated by $L$ is contractive
on $L^p(dw)$,
\[
 \|u_a\|_{L^p(dw)}
 \leq\int_0^\infty e^{-at}\|K_tf\|_{L^p(dw)}\,dt
 \leq a^{-1}\|f\|_{L^p(dw)}.
\]
Moreover, positivity of the heat kernel implies
\[
 |u_a(\mathbf x)|
 \leq \int_{\mathbb R^N}
       \Gamma(\mathbf x,\mathbf y,0)|f(\mathbf y)|\,dw(\mathbf y).
\]
Theorem~\ref{thm:V_boundedness}, applied to $|f|$,
therefore gives
\[
 \|Vu_a\|_{L^p(dw)}
 \leq C_p\|f\|_{L^p(dw)},
\]
uniformly in $a>0$.

The equation $(L+aI)u_a=f$ yields
\[
 -\Delta_ku_a=f-Vu_a-au_a.
\]
Consequently,
\[
 \|-\Delta_ku_a\|_{L^p(dw)}
 \leq (2+C_p)\|f\|_{L^p(dw)}.
\]
Since both $u_a$ and $-\Delta_ku_a$ belong to $L^p(dw)$,
Lemma~\ref{lem:Riesz_second_order} applies and gives
\begin{equation}\label{eq:regularized_hessian_bound}
 \sum_{i,j=1}^N\|T_iT_ju_a\|_{L^p(dw)}
 \leq C_p\|f\|_{L^p(dw)},
\end{equation}
with a constant independent of $a$.

We next pass to the Green potential.
The heat-kernel representation and dominated convergence
give
\[
 u_a\longrightarrow u
 \quad\text{in }L^1_{\mathrm{loc}}(dw)
 \qquad\text{as }a\downarrow0.
\]
Indeed, the absolute values are dominated by the Green
potential of $|f|$, which is locally integrable by the
heat-kernel majorant established in
Proposition~\ref{propo:basic_decay}.
Thus $T_iT_ju_a\to T_iT_ju$ in the distributional sense.

Because $1<p<\infty$, the bound
\eqref{eq:regularized_hessian_bound} and weak compactness
in $L^p(dw)$ allow us to choose a sequence $a_n\downarrow0$
such that all $T_iT_ju_{a_n}$ converge weakly in $L^p(dw)$.
Their distributional limits identify the weak limits
with $T_iT_ju$.
Weak lower semicontinuity therefore gives
\[
 \|\nabla_k^2u\|_{L^p(dw)}
 \leq C_p\|f\|_{L^p(dw)}.
\]

Finally, density of $C_c^\infty(\mathbb R^N)$ in
$L^p(dw)$ gives a unique bounded extension of
$\nabla_k^2L^{-1}$ to all of $L^p(dw)$.
\end{proof}

\subsection{Imaginary powers and fractional potential estimates}

We first record a consequence of the symmetric contraction semigroup
generated by $L$. This argument does not require pointwise estimates for
spatial derivatives of the kernel.

\begin{theorem}\label{thm:imaginary_powers_Lp}
Assume $q>\frac{\mathbf{N}}{2}$, $N>2$, $V\in {\rm{RH}}^q(dw)$, and that
$V$ is $G$-invariant. Then, for every $\gamma\in\mathbb{R}$ and every
$1<p<\infty$, there is a constant $C_{p,\gamma}>0$ such that
\begin{equation}\label{eq:imaginary_powers_Lp}
    \|(-\Delta_k+V)^{i\gamma}f\|_{L^p(dw)}
    \leq C_{p,\gamma}\|f\|_{L^p(dw)}.
\end{equation}
More precisely, if
\[
    \phi>\arcsin\left|1-\frac{2}{p}\right|,
\]
then one may take
\[
    C_{p,\gamma}\leq C_{p,\phi}e^{\phi|\gamma|}.
\]
\end{theorem}

\begin{proof}
Recall that $L=-\Delta_k+V$. Since $\{e^{-tL}\}_{t > 0}$ forms a symmetric contraction semigroup on $L^p(dw)$ for all $1 \le p \le \infty$, applying the holomorphic functional calculus for symmetric contraction semigroups (see \cite[Proposition 11]{CarbonaroDragicevic2017}) yields that for every $\phi > \arcsin \left| 1 - \frac{2}{p} \right|$ and any bounded holomorphic function $F$ defined on $\{z \neq 0 : |\arg z| < \phi\}$, 
\[
    \|F(L)f\|_{L^p(dw)}
    \leq C_{p,\phi}
       \sup_{|\arg z|<\phi}|F(z)|\,\|f\|_{L^p(dw)}.
\]
Apply this with $F(z)=z^{i\gamma}$. Since
\[
    \sup_{|\arg z|<\phi}|z^{i\gamma}|\leq e^{\phi|\gamma|},
\]
we obtain~\eqref{eq:imaginary_powers_Lp}. On the null space of $L$,
$L^{i\gamma}$ is defined to be zero. This convention does not affect the
estimate.
\end{proof}

The following interpolation argument yields full estimates for the operator $V^{1/2}(-\Delta_k+V)^{-1/2}$.

\begin{theorem}\label{thm:fractional_potential_boundedness}
Assume $q>\frac{\mathbf{N}}{2}$, $N>2$, $V\in {\rm{RH}}^q(dw)$, and that
$V$ is $G$-invariant. Let $0<\theta<1$. Then, for every
\[
    1<p<\frac{q}{\theta},
\]
there is a constant $C_{p,\theta}>0$ such that
\begin{equation}\label{eq:fractional_potential_boundedness}
    \|V^\theta (-\Delta_k+V)^{-\theta}f\|_{L^p(dw)}
    \leq C_{p,\theta}\|f\|_{L^p(dw)}.
\end{equation}
In particular, for every $1<p<2q$,
\begin{equation}\label{eq:half_potential_boundedness}
    \|V^{1/2}(-\Delta_k+V)^{-1/2}f\|_{L^p(dw)}
    \leq C_p\|f\|_{L^p(dw)}.
\end{equation}
\end{theorem}

\begin{proof}
Fix $1<p<q/\theta$. Choose $1<p_0<\infty$ and $1<p_1\leq q$ so that
\begin{equation}\label{eq:interpolation_indices_potential}
    \frac{1}{p}=\frac{1-\theta}{p_0}+\frac{\theta}{p_1}.
\end{equation}
Such a choice is possible precisely because $1<p<q/\theta$.

To justify the analytic interpolation, put $L=-\Delta_k+V$
and first consider the regularized families
\[
 F_{a,n}(z)
 =M_{\chi_n V^z}(L+aI)^{-z},
 \qquad
 \chi_n=\mathbf1_{\{n^{-1}\leq V\leq n\}},
 \qquad a>0,\quad n\geq1,
\]
on the strip $0\leq\operatorname{Re}z\leq1$.
Here $M_{\chi_n V^z}$ denotes multiplication by
$\chi_n V^z$, defined to vanish outside the support
of $\chi_n$.
For fixed $a,n$, these families are analytic in the
open strip and strongly continuous on its boundary
on $L^2(dw)$.

The boundary estimates below hold uniformly in $a,n$:
the imaginary-power bounds apply uniformly to $L+aI$,
while positivity gives
$|(L+aI)^{-1}f|\leq L^{-1}|f|$.
After interpolation, let $n\to\infty$ and then
$a\downarrow0$. For compactly supported smooth $f$,
the semigroup representation of the negative powers
and heat-kernel domination justify the latter limit
almost everywhere. Fatou's lemma and density then
give the asserted estimate for $V^\theta L^{-\theta}$.

For $z=i\tau$, multiplication by $V^{i\tau}$ has norm at most one and
Theorem~\ref{thm:imaginary_powers_Lp} gives
\[
    \|V^{i\tau}(-\Delta_k+V)^{-i\tau}f\|_{L^{p_0}(dw)}
    \leq C_{p_0,\tau}\|f\|_{L^{p_0}(dw)}.
\]
For $z=1+i\tau$, we use the identity
\[
    V^{1+i\tau}(-\Delta_k+V)^{-1-i\tau}
    =V^{i\tau}(V(-\Delta_k+V)^{-1})(-\Delta_k+V)^{-i\tau}.
\]
Theorem~\ref{thm:V_boundedness} and
Theorem~\ref{thm:imaginary_powers_Lp} therefore imply
\[
    \|V^{1+i\tau}(-\Delta_k+V)^{-1-i\tau}f\|_{L^{p_1}(dw)}
    \leq C_{p_1,\tau}\|f\|_{L^{p_1}(dw)}.
\]
The exponential growth in $|\tau|$ is admissible in Stein's interpolation
theorem \cite{Stein1956}. Interpolation at $z=\theta$, together
with~\eqref{eq:interpolation_indices_potential}, proves
\eqref{eq:fractional_potential_boundedness}. 
\end{proof}

\subsection{An interpolation estimate for the Riesz transform}

The same argument applied to the kinetic part of $(-\Delta_k+V)$ gives a non-sharp but
unconditional range for the first-order transform.

\begin{theorem}\label{thm:Riesz_boundedness_interpolation}
Assume $q>\frac{\mathbf{N}}{2}$, $N>2$, $V\in {\rm{RH}}^q(dw)$, and that
$V$ is $G$-invariant. Then, for every $1<p<2q$, there is a constant
$C_p>0$ such that
\begin{equation}\label{eq:Riesz_boundedness_interpolation}
    \|\nabla_k(-\Delta_k+V)^{-1/2}f\|_{L^p(dw)}
    \leq C_p\|f\|_{L^p(dw)}.
\end{equation}
By duality,
\begin{equation}\label{eq:adjoint_Riesz_boundedness_interpolation}
    \|(-\Delta_k+V)^{-1/2}\nabla_k f\|_{L^p(dw)}
    \leq C_p\|f\|_{L^p(dw)}
\end{equation}
for $(2q)'<p<\infty$.
\end{theorem}

\begin{proof}
Since
\[
    -\Delta_k(-\Delta_k+V)^{-1}=I-V(-\Delta_k+V)^{-1},
\]
Theorem~\ref{thm:V_boundedness} shows that $(-\Delta_k)(-\Delta_k+V)^{-1}$ is bounded
on $L^p(dw)$ for $1<p\leq q$. The operator $-\Delta_k$ also generates a
symmetric contraction semigroup. Hence its imaginary powers are bounded
on $L^p(dw)$ for $1<p<\infty$, by the argument used in
Theorem~\ref{thm:imaginary_powers_Lp}.

Consider the analytic family
\[
    (-\Delta_k)^z(-\Delta_k+V)^{-z},\qquad 0\leq\operatorname{Re}z\leq1.
\]
The analytic family below is justified by the same
regularization and interpolation argument as in the proof
of Theorem~\ref{thm:fractional_potential_boundedness}.
More precisely, with $A=-\Delta_k$, we first use
\[
 F_{a,b}(z)=e^{-bA}A^z(L+aI)^{-z},
 \qquad a,b>0.
\]
The boundary estimates are uniform in $a,b$, since
\[
 A(L+aI)^{-1}
 =I-V(L+aI)^{-1}-a(L+aI)^{-1}.
\]
After interpolation, we let $b\downarrow0$ and then
$a\downarrow0$, using $L^2$ convergence on test functions,
Fatou's lemma, and density.
For simplicity, we write the argument below for
the unregularized family $A^zL^{-z}$.
On the lower boundary both imaginary powers are bounded. On the upper
boundary we use
\[
 (-\Delta_k)^{1+i\tau}(-\Delta_k+V)^{-1-i\tau}
 =(-\Delta_k)^{i\tau}\big((-\Delta_k)(-\Delta_k+V)^{-1}\big)L^{-i\tau}.
\]
The interpolation argument in the proof of
Theorem~\ref{thm:fractional_potential_boundedness}, with $\theta=1/2$,
therefore gives
\begin{equation}\label{eq:kinetic_square_root}
    \|(-\Delta_k)^{1/2}(-\Delta_k+V)^{-1/2}f\|_{L^p(dw)}
    \leq C_p\|f\|_{L^p(dw)},\qquad 1<p<2q.
\end{equation}
The Dunkl Riesz transforms $T_j(-\Delta_k)^{-1/2}$ are bounded on
$L^p(dw)$ for every $1<p<\infty$. On a dense subspace,
\[
    T_j(-\Delta_k+V)^{-1/2}
    =T_j(-\Delta_k)^{-1/2}(-\Delta_k)^{1/2}L^{-1/2}.
\]
Combining this identity with~\eqref{eq:kinetic_square_root} and summing over
$1\leq j\leq N$ proves~\eqref{eq:Riesz_boundedness_interpolation}.

Every $T_j$ is skew-symmetric on $L^2(dw)$, while $(-\Delta_k+V)^{-1/2}$ is
self-adjoint. Thus
\[
    (T_j(-\Delta_k+V)^{-1/2})^*=-(-\Delta_k+V)^{-1/2}T_j.
\]
Hence, duality gives~\eqref{eq:adjoint_Riesz_boundedness_interpolation}.
\end{proof}

\begin{corollary}\label{cor:mixed_unconditional}
Under the assumptions of Theorem~\ref{thm:Riesz_boundedness_interpolation},
for every $(2q)'<p<2q$,
\begin{align}
 \|\nabla_k(-\Delta_k+V)^{-1}\nabla_k f\|_{L^p(dw)}
 &\leq C_p\|f\|_{L^p(dw)},
 \label{eq:mixed_second_order_unconditional}\\
 \|V^{1/2}(-\Delta_k+V)^{-1}\nabla_k f\|_{L^p(dw)}
 &\leq C_p\|f\|_{L^p(dw)}.
 \label{eq:mixed_potential_unconditional}
\end{align}
\end{corollary}

\begin{proof}
Put $L=-\Delta_k+V$.
We interpret the operators componentwise through
the compositions
\[
 T_iL^{-1}T_j
 :=(T_iL^{-1/2})(L^{-1/2}T_j),
 \qquad 1\leq i,j\leq N,
\]
and
\[
 V^{1/2}L^{-1}T_j
 :=(V^{1/2}L^{-1/2})(L^{-1/2}T_j).
\]
The left factors are bounded on $L^p(dw)$ for
$1<p<2q$, by
Theorems~\ref{thm:Riesz_boundedness_interpolation}
and~\ref{thm:fractional_potential_boundedness}.
The right factors are bounded for $(2q)'<p<\infty$
by duality.
Thus both compositions are bounded whenever
$(2q)'<p<2q$.
Summing the component estimates proves the assertion.
\end{proof}

Combining the unconditional estimates proved above with
Theorems~\ref{thm:V_boundedness} and~\ref{thm:Riesz_boundedness} gives the
following statement.

\begin{corollary}\label{cor:equation_unconditional}
Assume that $N>2$, $q>\mathbf N/2$, and that
$V\geq0$, $V\not\equiv0$, is $G$-invariant and belongs to
${\rm RH}^q(dw)\cap L^2_{\mathrm{loc}}(dw)$.
Write $L=-\Delta_k+V$.

For every $f\in L^p(dw)$, $1\leq p\leq q$, the Green
potential $u=L^{-1}f$ is well defined as an element of
$L^p_{\mathrm{loc}}(dw)$ and satisfies
\[
 Lu=f
 \quad\text{in the distributional sense}.
\]
Moreover,
\begin{align*}
 \|Vu\|_{L^p(dw)}
 &\leq C_p\|f\|_{L^p(dw)},
 &&1\leq p\leq q,\\
 \|\nabla_k^2u\|_{L^p(dw)}
 &\leq C_p\|f\|_{L^p(dw)},
 &&1<p\leq q.
\end{align*}

For every vector field
$\mathbf f=(f_1,\ldots,f_N)\in L^p(dw;\mathbb C^N)$,
where $(2q)'<p<2q$, the canonical Green solution
\[
 u=L^{-1}\operatorname{div}_k\mathbf f,
 \qquad
 \operatorname{div}_k\mathbf f=\sum_{j=1}^N T_jf_j,
\]
belongs to $L^p_{\mathrm{loc}}(dw)$, satisfies
\[
 Lu=\operatorname{div}_k\mathbf f
 \quad\text{in the distributional sense},
\]
and obeys
\[
 \|\nabla_k u\|_{L^p(dw)}
 +\|V^{1/2}u\|_{L^p(dw)}
 \leq C_p\|\mathbf f\|_{L^p(dw;\mathbb C^N)}.
\]
Here
\[
 \|\mathbf f\|_{L^p(dw;\mathbb C^N)}
 =
 \left\|
 \left(\sum_{j=1}^N|f_j|^2\right)^{1/2}
 \right\|_{L^p(dw)}.
\]

For general data, the canonical solutions are obtained
as limits in $L^p_{\mathrm{loc}}(dw)$ of Green solutions
with smooth compactly supported data. These limits
are independent of the approximating sequence.
\end{corollary}

\begin{proof}
We first establish two auxiliary bounds ensuring
local convergence of the Green solutions:
\begin{equation}\label{eq:critical_weight_green_bounds}
 \|m^2L^{-1}g\|_{L^p(dw)}
 +\|mL^{-1/2}g\|_{L^p(dw)}
 \leq C_p\|g\|_{L^p(dw)},
 \qquad 1\leq p<\infty.
\end{equation}

Let $G_2(\mathbf x,\mathbf y)=\Gamma(\mathbf x,\mathbf y,0)$
and define the kernel of $L^{-1/2}$ by
\[
 G_1(\mathbf x,\mathbf y)
 =\frac1{\sqrt\pi}\int_0^\infty
 t^{-1/2}k_t(\mathbf x,\mathbf y)\,dt.
\]
The latter representation agrees with the spectral
definition on its natural domain, by the semigroup
formula for negative powers.

Put $\rho=d(\mathbf x,\mathbf y)>0$.
Integrating the heat-kernel estimate gives
\[
 \int_0^\infty t^{-1}h_t(\mathbf x,\mathbf y)\,dt
 \leq \frac{C}{w(B(\mathbf x,\rho))}.
\]
Consequently, Cauchy--Schwarz and $0\leq k_t\leq h_t$
yield
\[
 G_1(\mathbf x,\mathbf y)
 \leq C
 \left(\Gamma(\mathbf x,\mathbf y,0)\right)^{1/2}
 \left(
 \int_0^\infty t^{-1}h_t(\mathbf x,\mathbf y)\,dt
 \right)^{1/2}.
\]
Together with
Theorem~\ref{theorem:fundamental_invariant}, this proves,
for $\alpha\in\{1,2\}$ and every $M>0$,
\begin{equation}\label{eq:green_fractional_critical_decay}
 0\leq G_\alpha(\mathbf x,\mathbf y)
 \leq
 C_M\frac{\rho^\alpha}{w(B(\mathbf x,\rho))}
       (1+m(\mathbf x)\rho)^{-M}.
\end{equation}
Both kernels are symmetric.

Consider
\[
 H_\alpha(\mathbf x,\mathbf y)
 =m(\mathbf x)^\alpha G_\alpha(\mathbf x,\mathbf y).
\]
Decomposing into orbit annuli of radii
$2^\ell/m(\mathbf x)$, $\ell\in\mathbb Z$, gives
\[
 \sup_{\mathbf x}
 \int_{\mathbb R^N}H_\alpha(\mathbf x,\mathbf y)
 \,dw(\mathbf y)
 \leq
 C_M\sum_{\ell\in\mathbb Z}
 2^{\ell\alpha}(1+2^\ell)^{-M}
 <\infty
\]
when $M>\alpha$.

To estimate the integral in the other variable,
Lemma~\ref{lem:m_growth} and the $G$-invariance of $m$
give
\[
 m(\mathbf x)
 \leq C m(\mathbf y)
       (1+m(\mathbf y)d(\mathbf x,\mathbf y))^\kappa.
\]
Using the symmetry of $G_\alpha$, followed by
\eqref{eq:green_fractional_critical_decay} with the
variables interchanged, we obtain
\[
 \sup_{\mathbf y}
 \int_{\mathbb R^N}H_\alpha(\mathbf x,\mathbf y)
 \,dw(\mathbf x)<\infty
\]
by choosing $M>\alpha(\kappa+1)$.
The two integral bounds and interpolation prove
\eqref{eq:critical_weight_green_bounds}.
Since $m$ is bounded below by a positive constant
on each compact set, these bounds also give
\[
 \|L^{-1}g\|_{L^p(K,dw)}
 +\|L^{-1/2}g\|_{L^p(K,dw)}
 \leq C_{K,p}\|g\|_{L^p(dw)}
\]
for every compact set $K$.

We now consider scalar data. Fix $1\leq p\leq q$.
Let $f_n\in C_c^\infty(\mathbb R^N)$ converge to
$f$ in $L^p(dw)$, and put $u_n=L^{-1}f_n$.
By \eqref{eq:critical_weight_green_bounds},
the sequence $u_n$ converges in $L^p_{\mathrm{loc}}(dw)$
to a limit $u$, independently of the approximation.
This limit agrees with the Green potential of $f$.

Since $1\leq p\leq q$, Theorem~\ref{thm:V_boundedness}
gives
\[
 \|V(u_n-u_m)\|_{L^p(dw)}
 \leq C_p\|f_n-f_m\|_{L^p(dw)}.
\]
Thus $(Vu_n)$ is a Cauchy sequence in $L^p(dw)$
and consequently converges in that space.
Local convergence of $u_n$, followed by passage to an
almost everywhere convergent subsequence, identifies
this limit with $Vu$.
Thus
\[
 \|Vu\|_{L^p(dw)}\leq C_p\|f\|_{L^p(dw)}.
\]
Passing to the limit in $Lu_n=f_n$ proves $Lu=f$
distributionally.
For $1<p\leq q$,
Theorem~\ref{thm:Riesz_boundedness} likewise gives
convergence of $T_iT_ju_n$ in $L^p(dw)$.
Their distributional limits are $T_iT_ju$, which
proves the asserted Hessian estimate.

Next, let $(2q)'<p<2q$ and choose
$\mathbf f_n\in C_c^\infty(\mathbb R^N;\mathbb C^N)$
converging to $\mathbf f$ in $L^p(dw;\mathbb C^N)$.
Set
\[
 w_n=L^{-1/2}\operatorname{div}_k\mathbf f_n,
 \qquad
 u_n=L^{-1}\operatorname{div}_k\mathbf f_n.
\]
For smooth compactly supported data, the identity
\[
 u_n=L^{-1/2}w_n
\]
is understood through spectral regularization and
agrees with the Green potential.
The adjoint estimate in
Theorem~\ref{thm:Riesz_boundedness_interpolation}
implies that $w_n$ converges in $L^p(dw)$.
Applying \eqref{eq:critical_weight_green_bounds} to
$L^{-1/2}(w_n-w_m)$ therefore shows that $u_n$
converges in $L^p_{\mathrm{loc}}(dw)$ to a limit $u$,
independently of the approximating sequence.

Corollary~\ref{cor:mixed_unconditional} gives
\[
 \|\nabla_k(u_n-u_m)\|_{L^p(dw)}
 +\|V^{1/2}(u_n-u_m)\|_{L^p(dw)}
 \leq C_p
 \|\mathbf f_n-\mathbf f_m\|_{L^p(dw;\mathbb C^N)}.
\]
The limits are identified with $\nabla_k u$
distributionally and with $V^{1/2}u$ almost everywhere.
Passing to the limit proves the required estimates.

Finally, $p>(2q)'$ implies $p'<2q$.
Hence $V^{1/2}\varphi\in L^{p'}(dw)$ for every
$\varphi\in C_c^\infty(\mathbb R^N)$, and
\[
 \left|\int_{\mathbb R^N}V(u_n-u)\varphi\,dw\right|
 \leq
 \|V^{1/2}(u_n-u)\|_{L^p(dw)}
 \|V^{1/2}\varphi\|_{L^{p'}(dw)}
 \longrightarrow0.
\]
Together with local convergence of $u_n$ and
$\mathbf f_n\to\mathbf f$ in $L^p(dw;\mathbb C^N)$,
this allows us to pass to the limit in
\[
 Lu_n=\operatorname{div}_k\mathbf f_n.
\]
Thus $Lu=\operatorname{div}_k\mathbf f$
in the distributional sense.
\end{proof}

\section{Gradient and annular kernel estimates}

In the preceding results we use the estimates already proved. The range in Shen's first-order theorem and the pointwise
Calder\'on--Zygmund assertions require additional spatial regularity. In the
Euclidean proof this regularity is obtained from a local gradient estimate,
the classical fractional integral theorem, and Morrey's embedding theorem.
For the Dunkl Laplacian, $T_j$ contains reflection terms, so that the same
conclusion does not follow from Lemma~\ref{lem:subharmonic} or
Theorem~\ref{theorem:fundamental_invariant}. We isolate below the exact
estimates which close the argument.

\subsection{Derivatives of the free resolvent kernel}

Let $K_j$ and $K_j^0$ denote the distribution kernels of
$T_jL^{-1/2}$ and $T_j(-\Delta_k)^{-1/2}$, respectively,
where $L=-\Delta_k+V$.
We use the principal branch, so that
\[
 (-i\tau)^{-1/2}
 =|\tau|^{-1/2}
   e^{i\pi\operatorname{sgn}(\tau)/4},
 \qquad \tau\neq0.
\]
The scalar identity
\[
 \lambda^{-1/2}
 =\frac1{2\pi}\int_{\mathbb R}
   (-i\tau)^{-1/2}(\lambda+i\tau)^{-1}\,d\tau,
 \qquad \lambda>0,
\]
leads to the representations
\begin{align}
 K_j(\mathbf x,\mathbf y)
 &=\frac1{2\pi}\int_{\mathbb R}
 (-i\tau)^{-1/2}T_{j,\mathbf x}
 \Gamma(\mathbf x,\mathbf y,\tau)\,d\tau,
 \label{eq:Riesz_kernel_representation}\\
 K_j^0(\mathbf x,\mathbf y)
 &=\frac1{2\pi}\int_{\mathbb R}
 (-i\tau)^{-1/2}T_{j,\mathbf x}
 \Gamma_0(\mathbf x,\mathbf y,\tau)\,d\tau.
 \label{eq:free_Riesz_kernel_representation}
\end{align}
Here $\Gamma_0$ denotes the resolvent kernel of
$-\Delta_k+i\tau I$.

At this stage, these identities are understood weakly.
More precisely, one first uses the shifted resolvents
$(L+aI+i\tau I)^{-1}$ and
$(-\Delta_k+aI+i\tau I)^{-1}$, with $a>0$,
applies $T_j$ in the distributional sense, and then
lets $a\downarrow0$.
Pointwise interpretations of these formulas away from
the orbit diagonal require the derivative estimates
established below.

\begin{lemma}\label{lem:estimates_before_annular_Riesz}
Assume that $V$ is $G$-invariant, $N>2$, and
$V\in {\rm RH}^{q_1}(dw)$ with $q_1>\mathbf N/2$. Then for every integer $M\geq0$, there exists a constant $C_M>0$ such
that, whenever $d(\mathbf{a},\mathbf{b})>0$,
\begin{align}
 |T_{j,\mathbf{a}}\Gamma_0(\mathbf{a},\mathbf{b},\tau)|
 &\leq
 C_M
 \frac{d(\mathbf{a},\mathbf{b})}
 {w(B(\mathbf{a},d(\mathbf{a},\mathbf{b})))}
 \bigl(1+|\tau|^{1/2}d(\mathbf{a},\mathbf{b})\bigr)^{-M},
 \label{eq:free_resolvent_first_derivative}\\
 |T_{i,\mathbf{a}}T_{j,\mathbf{b}}
       \Gamma_0(\mathbf{a},\mathbf{b},\tau)|
 &\leq
 \frac{C_M}
 {w(B(\mathbf{a},d(\mathbf{a},\mathbf{b})))}
 \bigl(1+|\tau|^{1/2}d(\mathbf{a},\mathbf{b})\bigr)^{-M}.
 \label{eq:free_resolvent_second_derivative}
\end{align}

Moreover, for every $M,\ell>0$, there exists a constant
$C_{M,\ell}>0$ such that
\begin{equation}\label{eq:fundamental_tau_decay}
\begin{split}
 |\Gamma(\mathbf{a},\mathbf{b},\tau)|
 &\leq
 C_{M,\ell}
 \bigl(1+|\tau|^{1/2}d(\mathbf{a},\mathbf{b})\bigr)^{-M}
 \bigl(1+m(\mathbf{a})d(\mathbf{a},\mathbf{b})\bigr)^{-\ell}
 \frac{d(\mathbf{a},\mathbf{b})^2}
 {w(B(\mathbf{a},d(\mathbf{a},\mathbf{b})))}.
\end{split}
\end{equation}
The estimate for $\Gamma$ holds for almost every pair
$(\mathbf a,\mathbf b)$ with $d(\mathbf a,\mathbf b)>0$,
for each fixed $\tau\in\mathbb R$.
\end{lemma}

\begin{proof}
We first prove the free resolvent estimates. The heat-kernel derivative
estimates in
\cite[Theorem~4.1(c)]{ADzH} give, for every
integer $n\geq0$,
\begin{equation}\label{eq:free_heat_first_derivative}
\left|
T_{j,\mathbf{a}}\partial_t^n
h_t(\mathbf{a},\mathbf{b})
\right|
\leq
Ct^{-n-1/2}
w(B(\mathbf{a},\sqrt{t}))^{-1}
\exp\left(-c\frac{d(\mathbf{a},\mathbf{b})^2}{t}\right).
\end{equation}
Using the semigroup identity
\begin{equation}\label{eq:heat_semigroup_decomposition}
h_t(\mathbf{a},\mathbf{b})
=
\int_{\mathbb{R}^N}
h_{t/2}(\mathbf{a},\mathbf{z})
h_{t/2}(\mathbf{z},\mathbf{b})\,dw(\mathbf{z})
\end{equation}
together with the
symmetry of the heat kernel, give
\begin{equation}\label{eq:symmetric_free_heat_first_derivative}
\begin{split}
&\left|
T_{j,\mathbf{a}}\partial_t^m
h_t(\mathbf{a},\mathbf{b})
\right|
+
\left|
T_{j,\mathbf{b}}\partial_t^m
h_t(\mathbf{a},\mathbf{b})
\right|
\leq
Ct^{-m-1/2}
\frac{
\exp\left(-c\,d(\mathbf{a},\mathbf{b})^2/t\right)}
{\max\left\{
w(B(\mathbf{a},\sqrt{t})),
w(B(\mathbf{b},\sqrt{t}))
\right\}}.
\end{split}
\end{equation}
The heat kernel is smooth, and its Gaussian derivative estimates
provide integrable majorants locally uniformly in the external
variables and in $t>0$. We may therefore apply
$T_{i,\mathbf a}T_{j,\mathbf b}\partial_t^n$ under the integral
sign in \eqref{eq:heat_semigroup_decomposition}. By the Leibniz
formula,
\begin{equation}\label{eq:second_heat_derivative_semigroup_formula}
\begin{split}
&T_{i,\mathbf{a}}T_{j,\mathbf{b}}\partial_t^n
h_t(\mathbf{a},\mathbf{b})=
2^{-n}\sum_{m=0}^{n}\binom{n}{m}
\int_{\mathbb{R}^N}
\left.
T_{i,\mathbf{a}}\partial_s^m
h_s(\mathbf{a},\mathbf{z})
\right|_{s=t/2}
\left.
T_{j,\mathbf{b}}\partial_s^{n-m}
h_s(\mathbf{z},\mathbf{b})
\right|_{s=t/2}
\,dw(\mathbf{z}).
\end{split}
\end{equation}
The factor $2^{-n}$ appears because every differentiation of
$h_{t/2}$ with respect to $t$ produces a factor $1/2$.

Applying \eqref{eq:symmetric_free_heat_first_derivative} to both
factors on the right-hand side of
\eqref{eq:second_heat_derivative_semigroup_formula}, we obtain
\begin{equation}\label{eq:second_heat_derivative_before_convolution}
\begin{split}
&\left|
T_{i,\mathbf{a}}T_{j,\mathbf{b}}\partial_t^n
h_t(\mathbf{a},\mathbf{b})
\right|\leq
Ct^{-n-1}
\int_{\mathbb{R}^N}
\frac{
\exp\left(-c\,d(\mathbf{a},\mathbf{z})^2/t\right)}
{\max\left\{
w(B(\mathbf{a},\sqrt{t})),
w(B(\mathbf{z},\sqrt{t}))
\right\}}
\\
&\hspace{37mm}\times
\frac{
\exp\left(-c\,d(\mathbf{z},\mathbf{b})^2/t\right)}
{\max\left\{
w(B(\mathbf{z},\sqrt{t})),
w(B(\mathbf{b},\sqrt{t}))
\right\}}
\,dw(\mathbf{z}).
\end{split}
\end{equation}

Now the integral in
\eqref{eq:second_heat_derivative_before_convolution} is bounded by
\begin{equation}\label{eq:Gaussian_product_integral}
\begin{split}
&\frac{1}
{w(B(\mathbf{a},\sqrt{t}))
 w(B(\mathbf{b},\sqrt{t}))}
\int_{\mathbb{R}^N}
\exp\left(
-c\frac{
d(\mathbf{a},\mathbf{z})^2+
d(\mathbf{z},\mathbf{b})^2}{t}
\right)\,dw(\mathbf{z}).
\end{split}
\end{equation}
The orbit distance satisfies the triangle inequality. Therefore,
\[
d(\mathbf{a},\mathbf{z})^2+
d(\mathbf{z},\mathbf{b})^2
\geq
\frac12d(\mathbf{a},\mathbf{b})^2.
\]
It follows that
\begin{equation}\label{eq:Gaussian_product_separation}
\begin{split}
&\exp\left(
-c\frac{
d(\mathbf{a},\mathbf{z})^2+
d(\mathbf{z},\mathbf{b})^2}{t}
\right)\leq
\exp\left(
-\frac{c}{4}
\frac{d(\mathbf{a},\mathbf{b})^2}{t}
\right)
\exp\left(
-\frac{c}{2}
\frac{
d(\mathbf{a},\mathbf{z})^2+
d(\mathbf{z},\mathbf{b})^2}{t}
\right).
\end{split}
\end{equation}
The doubling property of $dw$ and a decomposition into dyadic orbit
annuli give
\begin{equation}\label{eq:Gaussian_orbit_integral}
\int_{\mathbb{R}^N}
\exp\left(-c\frac{d(\mathbf{x},\mathbf{z})^2}{t}\right)
\,dw(\mathbf{z})
\leq
Cw(B(\mathbf{x},\sqrt{t})).
\end{equation}
Using \eqref{eq:Gaussian_product_separation} and then
\eqref{eq:Gaussian_orbit_integral}, once with $\mathbf{x}=\mathbf{a}$
and once with $\mathbf{x}=\mathbf{b}$, we obtain
\begin{align*}
&\int_{\mathbb{R}^N}
\exp\left(
-c\frac{
d(\mathbf{a},\mathbf{z})^2+
d(\mathbf{z},\mathbf{b})^2}{t}
\right)\,dw(\mathbf{z})
\\
&\quad\leq
C\exp\left(
-c'\frac{d(\mathbf{a},\mathbf{b})^2}{t}
\right)
\min\left\{
w(B(\mathbf{a},\sqrt{t})),
w(B(\mathbf{b},\sqrt{t}))
\right\}.
\end{align*}
Substitution into \eqref{eq:Gaussian_product_integral} gives
\begin{align*}
&\int_{\mathbb{R}^N}
\frac{
\exp\left(-c\,d(\mathbf{a},\mathbf{z})^2/t\right)}
{\max\left\{
w(B(\mathbf{a},\sqrt{t})),
w(B(\mathbf{z},\sqrt{t}))
\right\}}
\frac{
\exp\left(-c\,d(\mathbf{z},\mathbf{b})^2/t\right)}
{\max\left\{
w(B(\mathbf{z},\sqrt{t})),
w(B(\mathbf{b},\sqrt{t}))
\right\}}
\,dw(\mathbf{z})
\\
&\quad\leq
C
\frac{
\exp\left(-c'\,d(\mathbf{a},\mathbf{b})^2/t\right)}
{
w(B(\mathbf{a},\sqrt{t}))}.
\end{align*}
Combining this estimate with
\eqref{eq:second_heat_derivative_before_convolution}, we conclude that
\begin{equation}\label{eq:free_heat_second_derivative}
\begin{split}
&\left|
T_{i,\mathbf{a}}T_{j,\mathbf{b}}\partial_t^n
h_t(\mathbf{a},\mathbf{b})
\right|\leq
Ct^{-n-1}
w(B(\mathbf{a},\sqrt{t}))^{-1}
\exp\left(-c\frac{d(\mathbf{a},\mathbf{b})^2}{t}\right).
\end{split}
\end{equation}

Put $\rho=d(\mathbf{a},\mathbf{b})$. The doubling and reverse
doubling properties of $dw$ imply
\begin{align}
 &\int_0^\infty
 t^{-n-1/2}w(B(\mathbf{a},\sqrt{t}))^{-1}
 \exp\left(-c\frac{\rho^2}{t}\right)\,dt
 \leq
 C_n\frac{\rho^{1-2n}}{w(B(\mathbf{a},\rho))},
 \label{eq:integrated_first_heat_derivative}\\
 &\int_0^\infty
 t^{-n-1}w(B(\mathbf{a},\sqrt{t}))^{-1}
 \exp\left(-c\frac{\rho^2}{t}\right)\,dt
 \leq
 C_n\frac{\rho^{-2n}}{w(B(\mathbf{a},\rho))}.
 \label{eq:integrated_second_heat_derivative}
\end{align}
The free resolvent kernel is understood in the sense:
\[
\Gamma_0(\mathbf{a},\mathbf{b},\tau)
=
\lim_{\delta\to 0+}
\int_0^\infty
e^{-(\delta+i\tau)t}
h_t(\mathbf{a},\mathbf{b})\,dt.
\]
For $\tau\neq0$, integration by parts $n$ times in the time variable,
followed by
\eqref{eq:integrated_first_heat_derivative} and
\eqref{eq:integrated_second_heat_derivative}, gives
\begin{align*}
 |T_{j,\mathbf{a}}\Gamma_0(\mathbf{a},\mathbf{b},\tau)|
 &\leq
 C_n|\tau|^{-n}
 \frac{\rho^{1-2n}}{w(B(\mathbf{a},\rho))},\\
 |T_{i,\mathbf{a}}T_{j,\mathbf{b}}
 \Gamma_0(\mathbf{a},\mathbf{b},\tau)|
 &\leq
 C_n|\tau|^{-n}
 \frac{\rho^{-2n}}{w(B(\mathbf{a},\rho))}.
\end{align*}
The estimates corresponding to $n=0$ give
\[
 |T_{j,\mathbf{a}}\Gamma_0(\mathbf{a},\mathbf{b},\tau)|
 \leq
 C\frac{\rho}{w(B(\mathbf{a},\rho))}
\]
and
\[
 |T_{i,\mathbf{a}}T_{j,\mathbf{b}}
 \Gamma_0(\mathbf{a},\mathbf{b},\tau)|
 \leq
 \frac{C}{w(B(\mathbf{a},\rho))}.
\]
Combining the estimates for $n=0$ with those for sufficiently large
$n$ proves
\eqref{eq:free_resolvent_first_derivative} and
\eqref{eq:free_resolvent_second_derivative}.

We next prove~\eqref{eq:fundamental_tau_decay}.
Fix $\tau\in\mathbb R$. By
Lemma~\ref{lem:resolvent_local_weak_solution}, for almost every
$\mathbf b\in\mathbb R^N$, the function
\[
 u(\mathbf z)=\Gamma(\mathbf z,\mathbf b,\tau)
\]
belongs to
$\mathcal D_{\mathrm{loc}}(\mathbf Q;\Omega_{\mathbf b})$, where
\[
 \Omega_{\mathbf b}
 =\mathbb R^N\setminus\mathcal O(\{\mathbf b\}),
\]
and satisfies
\[
 (-\Delta_k+V+i\tau)u=0
 \quad\text{weakly in }\Omega_{\mathbf b}.
\]
Fix such a point $\mathbf b$ for which the conclusion of
Proposition~\ref{propo:basic_decay} also holds for almost every
point in the first variable.

Let $\mathbf a\in\Omega_{\mathbf b}$ be a Lebesgue point of
$|u|^2$. Choose a fixed constant $0<c_0<1/8$ and put
\[
 \rho=d(\mathbf a,\mathbf b),
 \qquad R=c_0\rho.
\]
Then
\[
 \overline{U(\mathbf a,2R)}\subset\Omega_{\mathbf b},
\]
so the local Caccioppoli estimate applies on all the orbit
neighbourhoods used below.

Fix an integer $n\geq1$ and introduce the radii
\[
 r_i=R\left(1+\frac{i}{n}\right),
 \qquad i=0,\ldots,n.
\]
Using $G$-invariant cutoff functions between
$U(\mathbf a,r_i)$ and $U(\mathbf a,r_{i+1})$,
Lemma~\ref{lem:cacc} gives
\[
 |\tau|\int_{U(\mathbf a,r_i)}|u|^2\,dw
 \leq \frac{C_n}{R^2}
       \int_{U(\mathbf a,r_{i+1})}|u|^2\,dw,
 \qquad i=0,\ldots,n-1.
\]
For $\tau\neq0$, iteration therefore yields
\[
 \int_{U(\mathbf a,R)}|u|^2\,dw
 \leq C_n(|\tau|R^2)^{-n}
       \int_{U(\mathbf a,2R)}|u|^2\,dw.
\]
Combining this with the trivial estimate coming from
$U(\mathbf a,R)\subset U(\mathbf a,2R)$, we obtain
\begin{equation}\label{eq:tau_iteration_simple}
 \int_{U(\mathbf a,R)}|u|^2\,dw
 \leq C_n(1+|\tau|R^2)^{-n}
       \int_{U(\mathbf a,2R)}|u|^2\,dw.
\end{equation}
For $\tau=0$, the same inequality follows directly from
the inclusion of these sets.

Corollary~\ref{coro:c2}, applied with inner radius $R/2$
and outer radius $R$, together with the doubling property
of $dw$, gives
\[
 |u(\mathbf a)|^2
 \leq \frac{C}{w(B(\mathbf a,R))}
       \int_{U(\mathbf a,R)}|u(\mathbf z)|^2\,dw(\mathbf z).
\]
Here the value at $\mathbf a$ is understood through the
Lebesgue representative. Using
\eqref{eq:tau_iteration_simple}, we obtain
\[
 |u(\mathbf a)|^2
 \leq
 \frac{C_n(1+|\tau|R^2)^{-n}}{w(B(\mathbf a,R))}
 \int_{U(\mathbf a,2R)}|u(\mathbf z)|^2\,dw(\mathbf z).
\]
Since $dw$ is $G$-invariant and doubling,
\[
 w(U(\mathbf a,2R))
 \leq |G|\,w(B(\mathbf a,2R))
 \leq Cw(B(\mathbf a,R)).
\]
Consequently, taking square roots gives
\begin{equation}\label{eq:tau_mean_value_simple}
 |\Gamma(\mathbf a,\mathbf b,\tau)|
 \leq
 C_n(1+|\tau|R^2)^{-n/2}
 \operatorname*{ess\,sup}_{\mathbf z\in U(\mathbf a,2R)}
 |\Gamma(\mathbf z,\mathbf b,\tau)|.
\end{equation}

We now estimate the essential supremum on the right-hand side.
For $\mathbf z\in U(\mathbf a,2R)$, the triangle inequality
for the orbit distance gives
\[
 (1-2c_0)\rho
 \leq d(\mathbf z,\mathbf b)
 \leq (1+2c_0)\rho.
\]
Moreover, there exists $g\in G$ such that
$\|\mathbf z-g(\mathbf a)\|<2R$.
The $G$-invariance and the doubling property of $dw$
therefore imply
\[
 w(B(\mathbf z,d(\mathbf z,\mathbf b)))
 \simeq w(B(g(\mathbf a),\rho))
 =w(B(\mathbf a,\rho)).
\]
Applying Proposition~\ref{propo:basic_decay}, we conclude that
\[
 \operatorname*{ess\,sup}_{\mathbf z\in U(\mathbf a,2R)}
 |\Gamma(\mathbf z,\mathbf b,\tau)|
 \leq C\frac{\rho^2}{w(B(\mathbf a,\rho))}.
\]
Substitution into~\eqref{eq:tau_mean_value_simple} yields
\[
 |\Gamma(\mathbf a,\mathbf b,\tau)|
 \leq
 C_n(1+|\tau|R^2)^{-n/2}
 \frac{\rho^2}{w(B(\mathbf a,\rho))}.
\]
Since $R=c_0\rho$ and $c_0$ is fixed, this implies
\[
 |\Gamma(\mathbf a,\mathbf b,\tau)|
 \leq
 C_n(1+|\tau|^{1/2}\rho)^{-n}
 \frac{\rho^2}{w(B(\mathbf a,\rho))}.
\]

Finally, fix $M,\ell>0$ and choose an integer $n\geq2M$.
The preceding estimate gives
\[
 |\Gamma(\mathbf a,\mathbf b,\tau)|
 \leq
 C_M(1+|\tau|^{1/2}\rho)^{-2M}
 \frac{\rho^2}{w(B(\mathbf a,\rho))}.
\]
On the other hand, since $V$ is $G$-invariant,
Theorem~\ref{theorem:fundamental_invariant}, applied with
exponent $2\ell$, gives
\[
 |\Gamma(\mathbf a,\mathbf b,\tau)|
 \leq
 C_\ell(1+m(\mathbf a)\rho)^{-2\ell}
 \frac{\rho^2}{w(B(\mathbf a,\rho))}.
\]
Multiplying these two inequalities and taking square roots,
we obtain
\[
 |\Gamma(\mathbf a,\mathbf b,\tau)|
 \leq
 C_{M,\ell}
 (1+|\tau|^{1/2}\rho)^{-M}
 (1+m(\mathbf a)\rho)^{-\ell}
 \frac{\rho^2}{w(B(\mathbf a,\rho))}.
\]
This proves~\eqref{eq:fundamental_tau_decay} for almost every
pair $(\mathbf a,\mathbf b)$ off the orbit diagonal,
for each fixed $\tau\in\mathbb R$.
\end{proof}

\subsection{A localized fractional integral estimate}

Recall that
\[
 U(\mathbf a,r)=\bigcup_{g\in G}B(g(\mathbf a),r).
\]

\begin{lemma}\label{lem:estimates_before_annular_Riesz2}
Assume that $V$ is $G$-invariant,
$\frac{\mathbf{N}}{2}<q_1<\mathbf{N}$,  $N>2$, and
$V\in {\rm RH}^{q_1}(dw)$. Define $s$ by
\[
    \frac1s=1-\frac1{q_1}+\frac1{\mathbf{N}},
\]
and let $s'$ be its conjugate exponent. Then for every $\mathbf{a}\in\mathbb{R}^N$, $R>0$, and every function
$F\in L^{q_1}(dw)$ supported in
$U(\mathbf{a},2R)$, we have
\begin{equation}\label{eq:localized_fractional_integral}
\begin{split}
 &\left\|
 \int_{U(\mathbf{a},2R)}
 \frac{d(\,\cdot\,,\mathbf{z})}
 {w(B(\,\cdot\,,d(\,\cdot\,,\mathbf{z})))}
 |F(\mathbf{z})|\,dw(\mathbf{z})
 \right\|_{
 L^{s'}(U(\mathbf{a},R),dw)}
 \leq
 CR\,w(B(\mathbf{a},R))^{-1/\mathbf{N}}
 \|F\|_{
 L^{q_1}(U(\mathbf{a},2R),dw)}.
\end{split}
\end{equation}
Finally, let $\tau\in\mathbb R$ and suppose that
$u\in\mathcal D_{\mathrm{loc}}(\mathbf Q;U(\mathbf a,4R))$
is a local weak solution of
\[
    -\Delta_k u+(V+i\tau)u=0
\]
in $U(\mathbf{a},4R)$. Then
\begin{equation}\label{eq:local_resolvent_gradient}
\begin{split}
 &\|T_j u\|_{
 L^{s'}(U(\mathbf{a},R),dw)}
 \\
 &\quad\leq
 C\Bigg(
 Rw(B(\mathbf{a},R))^{-1/\mathbf{N}}
 \|V\|_{
 L^{q_1}(U(\mathbf{a},2R),dw)}
 +
 R^{-1}w(B(\mathbf{a},R))^{1/s'}
 \Bigg)
 \operatorname*{ess\,sup}_{\mathbf{z}\in
 U(\mathbf{a},4R)}
 |u(\mathbf{z})|.
\end{split}
\end{equation}
Consequently, there exists $\kappa_0>0$ such that
\begin{equation}\label{eq:local_resolvent_gradient_simplified}
\begin{split}
 &\|T_j u\|_{
 L^{s'}(U(\mathbf{a},R),dw)}\leq
 CR^{-1}w(B(\mathbf{a},R))^{1/s'}
 \bigl(1+Rm(\mathbf{a})\bigr)^{\kappa_0}
 \operatorname*{ess\,sup}_{\mathbf{z}\in
 U(\mathbf{a},4R)}
 |u(\mathbf{z})|.
\end{split}
\end{equation}
All constants are independent of $\mathbf a$, $R$, $\tau$, and the functions involved.
\end{lemma}

\begin{proof}
Let $M$ denote the Hardy--Littlewood maximal operator on
$(\mathbb R^N,dw)$, defined using Euclidean balls.
Fix $\mathbf{y}\in U(\mathbf{a},R)$. Then there exists $g_0\in G$ such that $\mathbf{y}\in B(g_0(\mathbf{a}),R)$. The $G$-invariance and doubling property of $dw$ imply
\[
    w(B(\mathbf{y},R))
    \simeq w(B(g_0(\mathbf{a}),R))
    =w(B(\mathbf{a},R)).
\]
Consequently, by~\eqref{eq:growth}, for every $0<r\leq3R$,
\begin{equation}\label{eq:local_volume_lower_bound}
 w(B(\mathbf{y},r))
 \geq
 C^{-1}
 \left(\frac{r}{R}\right)^{\mathbf{N}}
 w(B(\mathbf{a},R)).
\end{equation}

Fix $0<\delta\leq3R$ and split the integral into $d(\mathbf{y},\mathbf{z})<\delta$ and $d(\mathbf{y},\mathbf{z})\geq\delta$. For the first part, consider the dyadic orbit annuli
\[
  2^{-\ell-1}\delta
  \leq d(\mathbf{y},\mathbf{z})
  <2^{-\ell}\delta,
  \qquad \ell\geq0.
\]
The doubling property and the $G$-invariance of $dw$ give
\begin{align*}
&\int_{\{2^{-\ell-1}\delta
\leq d(\mathbf{y},\mathbf{z})<2^{-\ell}\delta\}}
\frac{d(\mathbf{y},\mathbf{z})}
{w(B(\mathbf{y},d(\mathbf{y},\mathbf{z})))}
|F(\mathbf{z})|\,dw(\mathbf{z})
\\
&\quad\leq
C2^{-\ell}\delta
\sum_{g\in G}
\frac{1}{w(B(g(\mathbf{y}),2^{-\ell}\delta))}
\int_{B(g(\mathbf{y}),2^{-\ell}\delta)}
|F(\mathbf{z})|\,dw(\mathbf{z})
\leq
C2^{-\ell}\delta
\sum_{g\in G}M(|F|)(g(\mathbf{y})).
\end{align*}
Summing over $\ell\geq0$, we obtain
\begin{equation}\label{eq:localized_fractional_near_part}
\begin{split}
&\int_{\{d(\mathbf{y},\mathbf{z})<\delta\}}
\frac{d(\mathbf{y},\mathbf{z})}
{w(B(\mathbf{y},d(\mathbf{y},\mathbf{z})))}
|F(\mathbf{z})|\,dw(\mathbf{z})
\leq
C\delta
\sum_{g\in G}M(|F|)(g(\mathbf{y})).
\end{split}
\end{equation}

For the second part, decompose into the annuli $2^\ell\delta
  \leq d(\mathbf{y},\mathbf{z})
  <2^{\ell+1}\delta$, for those $\ell\geq0$ for which $2^\ell\delta\leq3R$.
The restriction $2^\ell\delta\leq3R$ is sufficient because $\mathbf{y}\in U(\mathbf{a},R),
 \mathbf{z}\in U(\mathbf{a},2R) $
imply $d(\mathbf{y},\mathbf{z})<3R $.
By Hölder's inequality,
\begin{align*}
&\int_{\{2^\ell\delta
\leq d(\mathbf{y},\mathbf{z})<2^{\ell+1}\delta\}}
\frac{d(\mathbf{y},\mathbf{z})}
{w(B(\mathbf{y},d(\mathbf{y},\mathbf{z})))}
|F(\mathbf{z})|\,dw(\mathbf{z})
\leq
C2^\ell\delta\,
w(B(\mathbf{y},2^\ell\delta))^{-1/q_1}
\|F\|_{
L^{q_1}(U(\mathbf{a},2R),dw)}.
\end{align*}
Using~\eqref{eq:local_volume_lower_bound}, we obtain
\begin{align*}
&\int_{\{2^\ell\delta
\leq d(\mathbf{y},\mathbf{z})<2^{\ell+1}\delta\}}
\frac{d(\mathbf{y},\mathbf{z})}
{w(B(\mathbf{y},d(\mathbf{y},\mathbf{z})))}
|F(\mathbf{z})|\,dw(\mathbf{z})
\\
&\quad\leq
C R^{\mathbf{N}/q_1}
w(B(\mathbf{a},R))^{-1/q_1}
(2^\ell\delta)^{1-\mathbf{N}/q_1}
\|F\|_{
L^{q_1}(U(\mathbf{a},2R),dw)}.
\end{align*}
Since $q_1<\mathbf{N}$, we have $ 1-{\mathbf{N}}/{q_1}<0$. Thus the sum over $\ell$ is convergent and gives
\begin{equation}\label{eq:localized_fractional_far_part}
\begin{split}
&\int_{\substack{
\mathbf{z}\in U(\mathbf{a},2R)\\
d(\mathbf{y},\mathbf{z})\geq\delta}}
\frac{d(\mathbf{y},\mathbf{z})}
{w(B(\mathbf{y},d(\mathbf{y},\mathbf{z})))}
|F(\mathbf{z})|\,dw(\mathbf{z})
\leq
C R^{\mathbf{N}/q_1}
w(B(\mathbf{a},R))^{-1/q_1}
\delta^{1-\mathbf{N}/q_1}
\|F\|_{
L^{q_1}(U(\mathbf{a},2R),dw)}.
\end{split}
\end{equation}

Combining~\eqref{eq:localized_fractional_near_part} and
\eqref{eq:localized_fractional_far_part}, we obtain
\begin{align*}
&\int_{U(\mathbf{a},2R)}
\frac{d(\mathbf{y},\mathbf{z})}
{w(B(\mathbf{y},d(\mathbf{y},\mathbf{z})))}
|F(\mathbf{z})|\,dw(\mathbf{z})
\\
&\quad\leq
C\delta\sum_{g\in G}M(|F|)(g(\mathbf{y}))+
C R^{\mathbf{N}/q_1}
w(B(\mathbf{a},R))^{-1/q_1}
\delta^{1-\mathbf{N}/q_1}
\|F\|_{
L^{q_1}(U(\mathbf{a},2R),dw)}.
\end{align*}
Assume first that $F\not\equiv0$. We work at points $\mathbf y$
where $\sum_{g\in G}M(|F|)(g(\mathbf y))$ is finite,
which holds almost everywhere. Then
\[
    \sum_{g\in G}M(|F|)(g(\mathbf{y}))>0.
\]
Choose
\begin{equation}\label{eq:balancing_delta_fractional}
\begin{split}
 \delta_0
 &=
 Rw(B(\mathbf{a},R))^{-1/\mathbf{N}}
 \left(
 \frac{
 \|F\|_{
 L^{q_1}(U(\mathbf{a},2R),dw)}
 }{
 \displaystyle\sum_{g\in G}M(|F|)(g(\mathbf{y}))
 }
 \right)^{q_1/\mathbf{N}}.
\end{split}
\end{equation}
This value is obtained by equating the two terms in the preceding
estimate. Suppose first that $\delta_0\leq3R$. Taking
$\delta=\delta_0$, both terms in the preceding estimate are equal.
Moreover,
\begin{align*}
&\delta_0\sum_{g\in G}M(|F|)(g(\mathbf{y}))
=
Rw(B(\mathbf{a},R))^{-1/\mathbf{N}}
\|F\|_{
L^{q_1}(U(\mathbf{a},2R),dw)
}^{q_1/\mathbf{N}}
\left(
\sum_{g\in G}M(|F|)(g(\mathbf{y}))
\right)^{1-q_1/\mathbf{N}}.
\end{align*}
Consequently,
\begin{align*}
\int_{U(\mathbf{a},2R)}
\frac{d(\mathbf{y},\mathbf{z})}
{w(B(\mathbf{y},d(\mathbf{y},\mathbf{z})))}
|F(\mathbf{z})|\,dw(\mathbf{z})&\leq
CRw(B(\mathbf{a},R))^{-1/\mathbf{N}}
\|F\|_{
L^{q_1}(U(\mathbf{a},2R),dw)
}^{q_1/\mathbf{N}}
\\
&\qquad\times
\left(
\sum_{g\in G}M(|F|)(g(\mathbf{y}))
\right)^{1-q_1/\mathbf{N}}.
\end{align*}

Suppose now that $\delta_0>3R$. Since
\[
 \mathbf{y}\in U(\mathbf{a},R)
 \quad\text{and}\quad
 \operatorname{supp}F
 \subseteq
 U(\mathbf{a},2R),
\]
we have $ d(\mathbf{y},\mathbf{z})<3R$ for every $\mathbf{z}\in\operatorname{supp}F$. Taking $\delta=3R$, the far part is therefore empty, and
\[
\begin{split}
&\int_{U(\mathbf{a},2R)}
\frac{d(\mathbf{y},\mathbf{z})}
{w(B(\mathbf{y},d(\mathbf{y},\mathbf{z})))}
|F(\mathbf{z})|\,dw(\mathbf{z})\leq
CR\sum_{g\in G}M(|F|)(g(\mathbf{y})).
\end{split}
\]
On the other hand, $\delta_0>3R$ and
\eqref{eq:balancing_delta_fractional} imply
\[
\begin{split}
&Rw(B(\mathbf{a},R))^{-1/\mathbf{N}}
\|F\|_{
L^{q_1}(U(\mathbf{a},2R),dw)
}^{q_1/\mathbf{N}}
\left(
\sum_{g\in G}M(|F|)(g(\mathbf{y}))
\right)^{1-q_1/\mathbf{N}}>
3R\sum_{g\in G}M(|F|)(g(\mathbf{y})).
\end{split}
\]
Hence the same pointwise estimate holds also when
$\delta_0>3R$. If $F\equiv0$, the estimate is immediate.

Now, $s'\left(1-{q_1}/{\mathbf{N}}\right)=q_1$. Therefore, taking the $L^{s'}$-norm and using the
$L^{q_1}(dw)$-boundedness of the Hardy--Littlewood maximal operator,
we obtain
\begin{align*}
&\left\|
\sum_{g\in G}M(|F|)(g(\,\cdot\,))
\right\|_{
L^{q_1}(U(\mathbf{a},R),dw)}
\leq
\sum_{g\in G}
\left\|
M(|F|)(g(\,\cdot\,))
\right\|_{L^{q_1}(dw)}
\\
&\quad=
|G|\|M(|F|)\|_{L^{q_1}(dw)}
\leq
C\|F\|_{L^{q_1}(dw)}.
\end{align*}
This proves~\eqref{eq:localized_fractional_integral}.

We next prove~\eqref{eq:local_resolvent_gradient}.
The local mean-value estimate implies that $u$ is essentially
bounded on $U(\mathbf a,2R)$, whose closure is contained in
$U(\mathbf a,4R)$. Define
\[
 v(\mathbf{y})
 =
 \int_{U(\mathbf{a},2R)}
 \Gamma_0(\mathbf{y},\mathbf{z},\tau)
 V(\mathbf{z})u(\mathbf{z})\,dw(\mathbf{z})
\]
and put $h=u+v$. The equation for $v$ is understood
in the distributional sense, using the fundamental-solution
identity for $\Gamma_0$; at $\tau=0$ we use the free Green kernel.
The local form regularity needed below is justified after
estimating $v$. Since $ (-\Delta_k+i\tau)u=-Vu$ and $(-\Delta_k+i\tau)v
 =
 V u\,
 \chi_{U(\mathbf{a},2R)}$,
we have $(-\Delta_k+i\tau)h=0 $
in $U(\mathbf{a},2R)$.

We first estimate $v$. For $\mathbf y\in U(\mathbf a,2R)$,
H\"older's inequality and the free-kernel version of
Proposition~\ref{propo:basic_decay}, obtained by integrating
$h_t$, give
\begin{align*}
 |v(\mathbf{y})|
 &\leq
 C
 \operatorname*{ess\,sup}_{\mathbf{z}\in
 U(\mathbf{a},2R)}
 |u(\mathbf{z})|
 \\
 &\quad\times
 \|V\|_{
 L^{q_1}(U(\mathbf{a},2R),dw)}
 \left(
 \int_{\substack{
 \mathbf{z}\in U(\mathbf{a},2R)\\
 d(\mathbf{y},\mathbf{z})<4R}}
 \left(
 \frac{d(\mathbf{y},\mathbf{z})^2}
 {w(B(\mathbf{y},d(\mathbf{y},\mathbf{z})))}
 \right)^{q_1'}
 dw(\mathbf{z})
 \right)^{1/q_1'}.
\end{align*}
A decomposition into the dyadic orbit annuli
$2^{-\ell-1}4R
 \leq d(\mathbf{y},\mathbf{z})
 <2^{-\ell}4R,
  \ell\geq0$, 
together with the doubling property, gives
\[
\begin{split}
&\left(
 \int_{\substack{
 \mathbf{z}\in U(\mathbf{a},2R)\\
 d(\mathbf{y},\mathbf{z})<4R}}
 \left(
 \frac{d(\mathbf{y},\mathbf{z})^2}
 {w(B(\mathbf{y},d(\mathbf{y},\mathbf{z})))}
 \right)^{q_1'}
 dw(\mathbf{z})
 \right)^{1/q_1'}\leq
CR^2w(B(\mathbf{a},R))^{-1/q_1}.
\end{split}
\]
Here the volume lower bound is used with centres in
$U(\mathbf a,2R)$ and radii at most $4R$, which follows
in the same way as \eqref{eq:local_volume_lower_bound}.
The geometric series converges because
\[
 2q_1'-\mathbf N(q_1'-1)
 =\frac{2q_1-\mathbf N}{q_1-1}>0.
\]
Consequently,
\begin{equation}\label{eq:local_free_potential_supremum}
\begin{split}
 &\operatorname*{ess\,sup}_{\mathbf{y}\in
 U(\mathbf{a},2R)}
 |v(\mathbf{y})|\leq
 CR^2w(B(\mathbf{a},R))^{-1/q_1}
 \|V\|_{
 L^{q_1}(U(\mathbf{a},2R),dw)}
 \operatorname*{ess\,sup}_{\mathbf{z}\in
 U(\mathbf{a},4R)}
 |u(\mathbf{z})|.
\end{split}
\end{equation}

We also justify the local form regularity of $v$ and the
first-order derivative representation used below.
Approximate $F=Vu\,\mathbf1_{U(\mathbf a,2R)}$ in $L^{q_1}(dw)$
by functions in $C_c^\infty(U(\mathbf a,2R))$ and consider their
free potentials. The preceding kernel-integrability estimate
gives uniform convergence of the potentials on $U(\mathbf a,2R)$.
Moreover, \eqref{eq:free_resolvent_first_derivative} and
\eqref{eq:localized_fractional_integral}, applied with radius
$2R$, give convergence of their Dunkl derivatives in
$L^{s'}(U(\mathbf a,2R),dw)$.
Since
\[
 s'=\frac{\mathbf Nq_1}{\mathbf N-q_1}>\mathbf N>2,
\]
these derivatives also converge in $L^2$ on this bounded set.
Integration by parts identifies the limits with the
Dunkl derivatives of $v$. Passing to the limit in the
weak equations shows that $v$, and hence $h$, belong to
the local form domain of the free Dunkl Laplacian and that
$(-\Delta_k+i\tau)h=0$ weakly in $U(\mathbf a,2R)$.

We next use an interior estimate for the free equation.
To justify applying the mean-value inequality to a derivative,
we first regularize $h$. Choose a $G$-invariant cutoff
$\chi\in C_c^\infty(U(\mathbf a,2R))$, with $0\leq\chi\leq1$,
equal to one on a neighbourhood of
$\overline{U(\mathbf a,7R/4)}$.
Let $\rho_\varepsilon$ be a nonnegative normalized smooth radial
Dunkl approximate identity supported in $B(0,\varepsilon)$,
and set
\[
 h_\varepsilon=(\chi h)*\rho_\varepsilon.
\]
For sufficiently small $\varepsilon$, this is a smooth solution of
\[
 (-\Delta_k+i\tau)h_\varepsilon=0
 \quad\text{in }U(\mathbf a,7R/4).
\]
Indeed, the Dunkl Laplacian commutes with convolution, and the
convolution kernel vanishes when the orbit distance exceeds
$\varepsilon$. Its positivity and normalization also imply
\[
 \|h_\varepsilon\|_{L^\infty(dw)}
 \leq \|h\|_{L^\infty(U(\mathbf a,2R),dw)}.
\]
Applying Lemma~\ref{lem:cacc} with zero potential,
inner radius $3R/2$, and outer radius $7R/4$, and using
the $G$-invariance and doubling property of $dw$, we obtain
\[
 \int_{U(\mathbf a,3R/2)}|T_jh_\varepsilon|^2\,dw
 \leq CR^{-2}w(B(\mathbf a,R))
       \|h\|_{L^\infty(U(\mathbf a,2R),dw)}^2.
\]
The Dunkl operators commute with $\Delta_k$, and hence
\[
 (-\Delta_k+i\tau)T_jh_\varepsilon=0
 \quad\text{in }U(\mathbf a,7R/4).
\]
By Lemma~\ref{lem:subharmonic} with zero potential,
$|T_jh_\varepsilon|^2$ is $\Delta_k$-subharmonic.
Applying the local mean-value estimate on orbit neighbourhoods
of radius comparable to $R$ contained in $U(\mathbf a,3R/2)$,
and using the preceding Caccioppoli bound, gives
\[
 \sup_{\mathbf y\in U(\mathbf a,R)}|T_jh_\varepsilon(\mathbf y)|
 \leq CR^{-1}\|h\|_{L^\infty(U(\mathbf a,2R),dw)}.
\]
The constants are independent of $\varepsilon$ and $\tau$.
Since $h_\varepsilon\to h$ locally in $L^2(dw)$, integration
by parts identifies the weak limits of $T_jh_\varepsilon$
with $T_jh$. Passing to the limit yields
\[
 \operatorname*{ess\,sup}_{\mathbf y\in U(\mathbf a,R)}
 |T_jh(\mathbf y)|
 \leq CR^{-1}
 \operatorname*{ess\,sup}_{\mathbf z\in U(\mathbf a,2R)}
 |h(\mathbf z)|.
\]
It follows that
\begin{equation}\label{eq:free_local_gradient_supremum}
\begin{split}
 &\|T_jh\|_{
 L^{s'}(U(\mathbf{a},R),dw)}\leq
 CR^{-1}w(B(\mathbf{a},R))^{1/s'}
 \operatorname*{ess\,sup}_{\mathbf{z}\in
 U(\mathbf{a},2R)}
 |h(\mathbf{z})|.
\end{split}
\end{equation}

Since $h=u+v$, estimate
\eqref{eq:local_free_potential_supremum} gives
\begin{align*}
&\operatorname*{ess\,sup}_{\mathbf{z}\in
U(\mathbf{a},2R)}
|h(\mathbf{z})|
\leq
\Bigg(
1+
CR^2w(B(\mathbf{a},R))^{-1/q_1}
\|V\|_{
L^{q_1}(U(\mathbf{a},2R),dw)}
\Bigg)
\operatorname*{ess\,sup}_{\mathbf{z}\in
U(\mathbf{a},4R)}
|u(\mathbf{z})|.
\end{align*}
Substituting this into
\eqref{eq:free_local_gradient_supremum}, and using $ 1/{s'}-1/{q_1}
    =-1/{\mathbf{N}}$, we obtain
\begin{align}
&\|T_jh\|_{
L^{s'}(U(\mathbf{a},R),dw)}
\nonumber\\
&\quad\leq
C\Bigg(
R^{-1}w(B(\mathbf{a},R))^{1/s'}
+
Rw(B(\mathbf{a},R))^{-1/\mathbf{N}}
\|V\|_{
L^{q_1}(U(\mathbf{a},2R),dw)}
\Bigg)
\operatorname*{ess\,sup}_{\mathbf{z}\in
U(\mathbf{a},4R)}
|u(\mathbf{z})|.
\label{eq:local_gradient_free_part}
\end{align}

On the other hand, the first-order free resolvent estimate \eqref{eq:free_resolvent_first_derivative} gives
\[
 |T_jv(\mathbf{y})|
 \leq
 C\int_{U(\mathbf{a},2R)}
 \frac{d(\mathbf{y},\mathbf{z})}
 {w(B(\mathbf{y},d(\mathbf{y},\mathbf{z})))}
 V(\mathbf{z})|u(\mathbf{z})|\,dw(\mathbf{z}).
\]
Applying~\eqref{eq:localized_fractional_integral} with
\[
    F=Vu\,
    \chi_{U(\mathbf{a},2R)},
\]
we obtain
\begin{align}
&\|T_jv\|_{
L^{s'}(U(\mathbf{a},R),dw)}
\nonumber\leq
CRw(B(\mathbf{a},R))^{-1/\mathbf{N}}
\|V\|_{
L^{q_1}(U(\mathbf{a},2R),dw)}
\operatorname*{ess\,sup}_{\mathbf{z}\in
U(\mathbf{a},4R)}
|u(\mathbf{z})|.
\label{eq:local_gradient_potential_part}
\end{align}

Finally, since $u=h-v$ the triangle inequality gives
\begin{align*}
 &\|T_j u\|_{
 L^{s'}(U(\mathbf{a},R),dw)}
 \\
 &\quad\leq
 C\Bigg(
 Rw(B(\mathbf{a},R))^{-1/\mathbf{N}}
 \|V\|_{
 L^{q_1}(U(\mathbf{a},2R),dw)}
 +
 R^{-1}w(B(\mathbf{a},R))^{1/s'}
 \Bigg)
 \operatorname*{ess\,sup}_{\mathbf{z}\in
 U(\mathbf{a},4R)}
 |u(\mathbf{z})|.
\end{align*}
This proves~\eqref{eq:local_resolvent_gradient}.

It remains to prove
\eqref{eq:local_resolvent_gradient_simplified}. The reverse
H\"older inequality, the doubling property of the measure
$V\,dw$, the definition of $m$, and the critical-radius estimates in Lemmas \ref{lem:mu_doubling} and \ref{lem:Rr} imply that there exists
$\kappa_0>0$ such that
\[
 \frac{1}{w(B(\mathbf{a},2R))}
 \int_{B(\mathbf{a},2R)}V(\mathbf{z})\,dw(\mathbf{z})
 \leq
 CR^{-2}\bigl(1+Rm(\mathbf{a})\bigr)^{\kappa_0}.
\]
By the reverse H\"older inequality,
\[
\begin{split}
 \|V\|_{L^{q_1}(B(\mathbf{a},2R),dw)}
 &\leq
 Cw(B(\mathbf{a},R))^{1/q_1}
 R^{-2}
 \bigl(1+Rm(\mathbf{a})\bigr)^{\kappa_0}.
\end{split}
\]
Since $V$ and $dw$ are $G$-invariant,
\[
\begin{split}
&\|V\|_{
L^{q_1}(U(\mathbf{a},2R),dw)}\leq
CR^{-2}w(B(\mathbf{a},R))^{1/q_1}
\bigl(1+Rm(\mathbf{a})\bigr)^{\kappa_0}.
\end{split}
\]
Finally, $1/{s'}
    =1/{q_1}-1/{\mathbf{N}}$, and therefore
\[
 Rw(B(\mathbf{a},R))^{-1/\mathbf{N}}
 \|V\|_{
 L^{q_1}(U(\mathbf{a},2R),dw)}
 \leq
 CR^{-1}w(B(\mathbf{a},R))^{1/s'}
 \bigl(1+Rm(\mathbf{a})\bigr)^{\kappa_0}.
\]
Substituting this into
\eqref{eq:local_resolvent_gradient}
proves~\eqref{eq:local_resolvent_gradient_simplified}.
\end{proof}

\subsection{Comparison with the free resolvent and annular estimates for the Riesz kernels}

Recall that $K_j(\mathbf y,\mathbf x)$ and $K_j^0(\mathbf y,\mathbf x)$
denote the integral kernels of $T_jL^{-1/2}$ and
$T_j(-\Delta_k)^{-1/2}$, respectively (see~\eqref{eq:Riesz_kernel_representation} and~\eqref{eq:free_Riesz_kernel_representation}).

\begin{lemma}\label{lem:annular_Riesz_kernel_bounds}
Assume that $N>2$, that $V$ is $G$-invariant, and that
$\mathbf N/2<q_1<\mathbf N$ and
$V\in{\rm RH}^{q_1}(dw)$. Let $s$ be defined by
\[
    \frac1s=1-\frac1{q_1}+\frac1{\mathbf{N}},
\]
and let $s'$ be its conjugate exponent.For almost every $\mathbf x\in\mathbb R^N$, put
$r=m(\mathbf x)^{-1}$.
There exists $\varepsilon>0$ such that, for every integer $\ell\geq1$,
\begin{equation}\label{eq:far_annular_Riesz_kernel}
\begin{split}
 &\left(
 \int_{\{2^{\ell-1}r<d(\mathbf{x},\mathbf{y})\leq2^\ell r\}}
 |K_j(\mathbf{y},\mathbf{x})|^{s'}\,dw(\mathbf{y})
 \right)^{1/s'}
\leq
 C2^{-\ell\varepsilon}
 w(B(\mathbf{x},2^\ell r))^{-1/s}.
\end{split}
\end{equation}
Moreover, for every integer $\ell\leq0$,
\begin{equation}\label{eq:local_annular_Riesz_kernel}
\begin{split}
 &\left(
 \int_{\{2^{\ell-1}r<d(\mathbf{x},\mathbf{y})\leq2^\ell r\}}
 |K_j(\mathbf{y},\mathbf{x})
      -K_j^0(\mathbf{y},\mathbf{x})|^{s'}\,dw(\mathbf{y})
 \right)^{1/s'}
\leq
 C2^{\ell(2-\mathbf{N}/q_1)}
 w(B(\mathbf{x},2^\ell r))^{-1/s}.
\end{split}
\end{equation}
\end{lemma}

\begin{proof}

We first record the following consequence of the resolvent identity:
for $0<R\leq r$ and every $M>1$,
\begin{equation}\label{eq:local_resolvent_difference_annular_short}
\begin{split}
&\left\|
T_{j,\mathbf{y}}
\bigl(
\Gamma(\mathbf{y},\mathbf{x},\tau)
-\Gamma_0(\mathbf{y},\mathbf{x},\tau)
\bigr)
\right\|_{
L^{s'}(\{R/2<d(\mathbf{x},\mathbf{y})\leq R\},dw)}
\\
&\qquad\leq
C_M R
\bigl(Rm(\mathbf{x})\bigr)^{2-\mathbf{N}/q_1}
w(B(\mathbf{x},R))^{-1/s}
\bigl(1+|\tau|^{1/2}R\bigr)^{-M}.
\end{split}
\end{equation}
Indeed, the semigroup perturbation formula
(see \cite[proof of Lemma~16]{Hejna_Schrodinger}) gives
\[
 h_t(\mathbf y,\mathbf x)-k_t(\mathbf y,\mathbf x)
 =
 \int_0^t\int_{\mathbb R^N}
 h_{t-s}(\mathbf y,\mathbf z)V(\mathbf z)
 k_s(\mathbf z,\mathbf x)\,dw(\mathbf z)\,ds.
\]
For the possibly unbounded potential $V$, this identity
is understood initially in the quadratic-form sense,
by approximation with bounded nonnegative potentials.
Positivity then yields the corresponding identity
for the kernels.

Integrating in $t$ and applying Tonelli's theorem, we obtain
\[
\begin{split}
 &\int_{\mathbb R^N}
 \Gamma_0(\mathbf y,\mathbf z,0)V(\mathbf z)
 \Gamma(\mathbf z,\mathbf x,0)\,dw(\mathbf z)
 \\
 &\qquad=
 \Gamma_0(\mathbf y,\mathbf x,0)
 -\Gamma(\mathbf y,\mathbf x,0)
 \leq \Gamma_0(\mathbf y,\mathbf x,0)<\infty
\end{split}
\]
for almost every pair $(\mathbf y,\mathbf x)$ off the
orbit diagonal. The last assertion follows from $N>2$
and the integrated heat-kernel bound.
Thus the Duhamel formula remains absolutely integrable
after multiplication by $e^{-i\tau t}$.
Fubini's theorem and the change of variables $t=s+r$
therefore give, for each fixed $\tau\in\mathbb R$,
\[
\begin{split}
 &\Gamma(\mathbf y,\mathbf x,\tau)
 -\Gamma_0(\mathbf y,\mathbf x,\tau)=
 -\int_{\mathbb R^N}
 \Gamma_0(\mathbf y,\mathbf z,\tau)V(\mathbf z)
 \Gamma(\mathbf z,\mathbf x,\tau)\,dw(\mathbf z).
\end{split}
\]
The integral is absolutely convergent, since
$|\Gamma_0(\mathbf y,\mathbf z,\tau)|
 \leq\Gamma_0(\mathbf y,\mathbf z,0)$ and
$|\Gamma(\mathbf z,\mathbf x,\tau)|
 \leq\Gamma(\mathbf z,\mathbf x,0)$.

Put
\[
 E_R=
 \{\mathbf{y}:R/2<d(\mathbf{x},\mathbf{y})\leq R\}.
\]
We assume that
\[
    0<R\leq r=m(\mathbf{x})^{-1}.
\]

The resolvent identity gives
\begin{equation}\label{eq:resolvent_identity_before_differentiation}
\begin{split}
&\Gamma(\mathbf{y},\mathbf{x},\tau)
-\Gamma_0(\mathbf{y},\mathbf{x},\tau)=
-\int_{\mathbb{R}^N}
\Gamma_0(\mathbf{y},\mathbf{z},\tau)
V(\mathbf{z})
\Gamma(\mathbf{z},\mathbf{x},\tau)\,dw(\mathbf{z}).
\end{split}
\end{equation}

{
Applying the Dunkl operator $T_{j,\mathbf y}$ distributionally to
\eqref{eq:resolvent_identity_before_differentiation}, we obtain
\begin{equation}\label{eq:differentiated_resolvent_identity_proper}
\begin{split}
&T_{j,\mathbf y}
\bigl(
\Gamma(\mathbf y,\mathbf x,\tau)
-\Gamma_0(\mathbf y,\mathbf x,\tau)
\bigr)=
-\int_{\mathbb R^N}
T_{j,\mathbf y}\Gamma_0(\mathbf y,\mathbf z,\tau)
V(\mathbf z)
\Gamma(\mathbf z,\mathbf x,\tau)\,dw(\mathbf z)
\end{split}
\end{equation}
in the sense of distributions.

We emphasize that, at this point, the integral on the right-hand side
is not asserted to converge pointwise.  We shall prove below that,
when restricted to $E_R$.

it defines a function in $L^{s'}(E_R,dw)$.  Consequently,
\eqref{eq:differentiated_resolvent_identity_proper} then holds almost
everywhere on $E_R$ and the distributional Dunkl derivative on the
left-hand side is represented by this $L^{s'}$ function.

To verify this, we split the right-hand side of
\eqref{eq:differentiated_resolvent_identity_proper} according to
\[
d(\mathbf x,\mathbf z)<R/4,\qquad
R/4\le d(\mathbf x,\mathbf z)\le4R,\qquad
d(\mathbf x,\mathbf z)>4R,
\]
and denote the resulting terms by
$I_1$, $I_2$, and $I_3$, respectively.}

We shall repeatedly use the following consequences of the reverse
H\"older assumptions and of the definition of the auxiliary function
$m$.  Put
\[
    \delta=2-\frac{\mathbf N}{q_1}.
\]
By Lemma~\ref{lem:Rr}, applied with exponent $q_1$, and
\eqref{eq:m}, for $0<t\leq m(\mathbf x)^{-1}$ we have
\[
\frac{t^2}{w(B(\mathbf{x},t))}
\int_{B(\mathbf{x},t)}
V(\mathbf{z})\,dw(\mathbf{z})
\leq
C\bigl(tm(\mathbf{x})\bigr)^\delta.
\]
Since $V$ and $dw$ are $G$-invariant, the same estimate, up to a
constant depending only on $G$, holds with $B(\mathbf x,t)$ replaced
by its orbit $\mathcal O(B(\mathbf x,t))$.  On the other hand, the
doubling property of the measure $d\mu=V\,dw$ from
Lemma~\ref{lem:mu_doubling}, together with the doubling properties of
$dw$ and the definition \eqref{eq:m}, yields a number
$\kappa_1>0$ such that
\begin{equation}\label{eq:potential_scale_used_in_resolvent_difference}
\frac{t^2}{w(B(\mathbf{x},t))}
\int_{\mathcal{O}(B(\mathbf{x},t))}
V(\mathbf{z})\,dw(\mathbf{z})
\leq
\begin{cases}
C\bigl(tm(\mathbf{x})\bigr)^\delta,
&0<t\leq m(\mathbf x)^{-1},\\[2mm]
C\bigl(tm(\mathbf{x})\bigr)^{\kappa_1},
&t>m(\mathbf x)^{-1}.
\end{cases}
\end{equation}

Furthermore, by the reverse H\"older inequality
\eqref{eq:reverse_Holder}, applied with exponent $q_1$, and the
first estimate in
\eqref{eq:potential_scale_used_in_resolvent_difference}, we obtain,
whenever $0<R\leq m(\mathbf x)^{-1}$,
\begin{equation}\label{eq:Lq_potential_scale_used_in_difference}
\begin{split}
&\|V\|_{
L^{q_1}(\mathcal{O}(B(\mathbf{x},4R)),dw)}\leq
CR^{-2}w(B(\mathbf{x},R))^{1/q_1}
\bigl(Rm(\mathbf{x})\bigr)^\delta.
\end{split}
\end{equation}
Here and below the harmless change from $R$ to $4R$ is absorbed by
the doubling properties of $dw$ and $\mu$.

We first estimate $I_1$. If $\mathbf{y}\in E_R$ and
$d(\mathbf{x},\mathbf{z})<R/4$, then $d(\mathbf{y},\mathbf{z})\simeq R$.
Therefore, by \eqref{eq:free_resolvent_first_derivative},
\[
\begin{split}
|T_{j,\mathbf{y}}\Gamma_0(\mathbf{y},\mathbf{z},\tau)|
&\leq
C_M\frac{R}{w(B(\mathbf{x},R))}
\bigl(1+|\tau|^{1/2}R\bigr)^{-M}.
\end{split}
\]
The symmetry of Proposition~\ref{propo:basic_decay}  give
\[
 |\Gamma(\mathbf{z},\mathbf{x},\tau)|
 \leq
 C\frac{d(\mathbf{x},\mathbf{z})^2}
 {w(B(\mathbf{x},d(\mathbf{x},\mathbf{z})))}.
\]
Consequently,
\begin{align*}
|I_1(\mathbf{y})|
&\leq
C_M\frac{R}{w(B(\mathbf{x},R))}
\bigl(1+|\tau|^{1/2}R\bigr)^{-M}
\int_{\{d(\mathbf{x},\mathbf{z})<R/4\}}
\frac{d(\mathbf{x},\mathbf{z})^2}
{w(B(\mathbf{x},d(\mathbf{x},\mathbf{z})))}
V(\mathbf{z})\,dw(\mathbf{z}).
\end{align*}
Decomposing the last integral into
\[
 2^{-\ell-1}R
 \leq d(\mathbf{x},\mathbf{z})<2^{-\ell}R,
 \qquad \ell\geq2,
\]
and using
\eqref{eq:potential_scale_used_in_resolvent_difference}, we obtain
\begin{align*}
&\int_{\{d(\mathbf{x},\mathbf{z})<R/4\}}
\frac{d(\mathbf{x},\mathbf{z})^2}
{w(B(\mathbf{x},d(\mathbf{x},\mathbf{z})))}
V(\mathbf{z})\,dw(\mathbf{z})\leq
C\sum_{\ell=2}^{\infty}
\bigl(2^{-\ell}Rm(\mathbf{x})\bigr)^\delta
\leq
C\bigl(Rm(\mathbf{x})\bigr)^\delta.
\end{align*}
Hence
\[
\begin{split}
|I_1(\mathbf{y})|
&\leq
C_M\frac{R}{w(B(\mathbf{x},R))}
\bigl(Rm(\mathbf{x})\bigr)^\delta
\bigl(1+|\tau|^{1/2}R\bigr)^{-M}.
\end{split}
\]
Since
\[
    w(E_R)\leq Cw(B(\mathbf{x},R)),
\]
it follows that
\begin{equation}\label{eq:first_part_resolvent_difference}
\begin{split}
\|I_1\|_{L^{s'}(E_R,dw)}
&\leq
C_MR\,w(B(\mathbf{x},R))^{-1/s}
\bigl(Rm(\mathbf{x})\bigr)^\delta
\bigl(1+|\tau|^{1/2}R\bigr)^{-M}.
\end{split}
\end{equation}

We next estimate $I_2$. On its domain of integration,
\[
    d(\mathbf{x},\mathbf{z})\simeq R.
\]
Therefore, by the symmetry of $\Gamma$ and
\eqref{eq:fundamental_tau_decay},
\[
\begin{split}
|\Gamma(\mathbf{z},\mathbf{x},\tau)|
&\leq
C_M\frac{R^2}{w(B(\mathbf{x},R))}
\bigl(1+|\tau|^{1/2}R\bigr)^{-M}.
\end{split}
\]
It follows from
\eqref{eq:free_resolvent_first_derivative} that
\begin{align*}
|I_2(\mathbf{y})|
&\leq
C_M\frac{R^2}{w(B(\mathbf{x},R))}
\bigl(1+|\tau|^{1/2}R\bigr)^{-M}
\int_{\mathcal{O}(B(\mathbf{x},4R))}
\frac{d(\mathbf{y},\mathbf{z})}
{w(B(\mathbf{y},d(\mathbf{y},\mathbf{z})))}
V(\mathbf{z})\,dw(\mathbf{z}).
\end{align*}
Applying \eqref{eq:localized_fractional_integral}, with radius $2R$,
gives
\begin{align*}
\|I_2\|_{L^{s'}(E_R,dw)}
&\leq
C_M\frac{R^2}{w(B(\mathbf{x},R))}
\bigl(1+|\tau|^{1/2}R\bigr)^{-M}
R\,w(B(\mathbf{x},R))^{-1/\mathbf{N}}
\|V\|_{
L^{q_1}(\mathcal{O}(B(\mathbf{x},4R)),dw)}.
\end{align*}

Using \eqref{eq:Lq_potential_scale_used_in_difference} and
\[
    \frac1{s'}=\frac1{q_1}-\frac1{\mathbf{N}},
\]
we get
\begin{equation}\label{eq:second_part_resolvent_difference}
\begin{split}
\|I_2\|_{L^{s'}(E_R,dw)}
&\leq
C_MR\,w(B(\mathbf{x},R))^{-1/s}
\bigl(Rm(\mathbf{x})\bigr)^\delta
\bigl(1+|\tau|^{1/2}R\bigr)^{-M}.
\end{split}
\end{equation}

Finally, decompose $I_3$ as
\[
 I_3(\mathbf{y})
 =
 \sum_{\ell=2}^{\infty}I_{3,\ell}(\mathbf{y}),
\]
where the integral defining $I_{3,\ell}$ is taken over
\[
 2^\ell R<d(\mathbf{x},\mathbf{z})\leq2^{\ell+1}R.
\]
For $\mathbf{y}\in E_R$ and $\mathbf{z}$ in this region,
\[
    d(\mathbf{y},\mathbf{z})\simeq2^\ell R.
\]
Hence, using
\eqref{eq:free_resolvent_first_derivative} and
\eqref{eq:fundamental_tau_decay}, we obtain
\begin{align*}
|I_{3,\ell}(\mathbf{y})|
&\leq
C_M
\frac{2^\ell R}{w(B(\mathbf{x},2^\ell R))}
\bigl(1+|\tau|^{1/2}R\bigr)^{-M}
\bigl(1+2^\ell Rm(\mathbf{x})\bigr)^{-\kappa}
\frac{(2^\ell R)^2}{w(B(\mathbf{x},2^\ell R))}
\int_{\mathcal{O}(B(\mathbf{x},2^{\ell+1}R))}
V(\mathbf{z})\,dw(\mathbf{z}).
\end{align*}
Using
\eqref{eq:potential_scale_used_in_resolvent_difference}, this gives
\[
\begin{split}
|I_{3,\ell}(\mathbf{y})|
&\leq
C_M\frac{2^\ell R}{w(B(\mathbf{x},2^\ell R))}
\bigl(1+|\tau|^{1/2}R\bigr)^{-M}
\times
\begin{cases}
\bigl(2^\ell Rm(\mathbf{x})\bigr)^\delta,
&2^\ell R\leq r,\\[1mm]
\bigl(2^\ell Rm(\mathbf{x})\bigr)^{\kappa_1-\kappa},
&2^\ell R>r.
\end{cases}
\end{split}
\]
The reverse doubling property (see~\eqref{eq:growth}) implies
\[
 \frac{2^\ell R}{w(B(\mathbf{x},2^\ell R))}
 \leq
 C2^{-\ell(N-1)}
 \frac{R}{w(B(\mathbf{x},R))}.
\]
Since
\[
    N-1-\delta>0,
\]
the terms for which $2^\ell R\leq r$ form a convergent series and are
bounded by
\[
C_M\frac{R}{w(B(\mathbf{x},R))}
\bigl(Rm(\mathbf{x})\bigr)^\delta
\bigl(1+|\tau|^{1/2}R\bigr)^{-M}.
\]
Choose
\[
    \kappa>\kappa_1+N.
\]
The terms for which $2^\ell R>r$ are then bounded by
\[
C_M\frac{R}{w(B(\mathbf{x},R))}
\bigl(Rm(\mathbf{x})\bigr)^{N-1}
\bigl(1+|\tau|^{1/2}R\bigr)^{-M}.
\]
Because
\[
    Rm(\mathbf{x})\leq1
    \quad\text{and}\quad
    N-1>\delta,
\]
we have
\[
    \bigl(Rm(\mathbf{x})\bigr)^{N-1}
    \leq
    \bigl(Rm(\mathbf{x})\bigr)^\delta.
\]
Consequently,
\begin{equation}\label{eq:third_part_resolvent_difference}
\begin{split}
\|I_3\|_{L^{s'}(E_R,dw)}
&\leq
C_M R\,w(B(\mathbf{x},R))^{-1/s}
\bigl(Rm(\mathbf{x})\bigr)^{2-\mathbf{N}/q_1}
\bigl(1+|\tau|^{1/2}R\bigr)^{-M}.
\end{split}
\end{equation}

Combining
\eqref{eq:first_part_resolvent_difference},
\eqref{eq:second_part_resolvent_difference}, and
\eqref{eq:third_part_resolvent_difference}, we conclude that
\[
\begin{split}
&\left\|
T_{j,\mathbf{y}}
\bigl(
\Gamma(\mathbf{y},\mathbf{x},\tau)
-\Gamma_0(\mathbf{y},\mathbf{x},\tau)
\bigr)
\right\|_{L^{s'}(E_R,dw)}
\\
&\qquad\leq
C_MR\,w(B(\mathbf{x},R))^{-1/s}
\bigl(Rm(\mathbf{x})\bigr)^{2-\mathbf{N}/q_1}
\bigl(1+|\tau|^{1/2}R\bigr)^{-M}.
\end{split}
\]
This proves \eqref{eq:local_resolvent_difference_annular_short}.

We now prove the far estimate. Let $\ell\geq1$, put
$R=2^\ell r$, and write
\[
 E_R=\{\mathbf y:R/2<d(\mathbf x,\mathbf y)\leq R\}.
\]
Cover $E_R$ by a bounded number of balls
$B(\mathbf a,cR)$ with centres $\mathbf a\in E_R$,
where $c>0$ is a sufficiently small fixed constant.
The orbit neighbourhoods $U(\mathbf a,4cR)$
do not meet $\mathcal O(\{\mathbf x\})$, and
$d(\mathbf x,\mathbf z)\simeq R$ throughout these
neighbourhoods.

Applying \eqref{eq:local_resolvent_gradient_simplified}
to $u(\mathbf y)=\Gamma(\mathbf y,\mathbf x,\tau)$,
and then using the symmetry of $\Gamma$ and
\eqref{eq:fundamental_tau_decay}, and choosing the spatial decay exponent sufficiently large,
we obtain, for every $M>0$ and some $\varepsilon>0$,
\begin{equation}\label{eq:far_resolvent_gradient_short}
\begin{split}
&\left\|
T_{j,\mathbf{y}}\Gamma(\mathbf{y},\mathbf{x},\tau)
\right\|_{
L^{s'}(\{R/2<d(\mathbf{x},\mathbf{y})\leq R\},dw)}
\leq
C_MR\,w(B(\mathbf{x},R))^{-1/s}
\bigl(1+Rm(\mathbf{x})\bigr)^{-\varepsilon}
\bigl(1+|\tau|^{1/2}R\bigr)^{-M}.
\end{split}
\end{equation}
Indeed, the local gradient estimate contributes the factor
$(1+cRm(\mathbf a))^{\kappa_0}$.
Choosing $g\in G$ such that
$\|\mathbf a-g(\mathbf x)\|=d(\mathbf a,\mathbf x)\leq R$,
we obtain from \eqref{eq:Shen_B} and the $G$-invariance of $m$
that
\[
 m(\mathbf a)
 \leq Cm(\mathbf x)(1+Rm(\mathbf x))^\kappa.
\]
Hence
\[
 (1+cRm(\mathbf a))^{\kappa_0}
 \leq C(1+Rm(\mathbf x))^{(\kappa+1)\kappa_0}.
\]
This factor is absorbed by choosing the spatial decay exponent
in \eqref{eq:fundamental_tau_decay} sufficiently large.
We also use
$w(B(\mathbf a,cR))\simeq w(B(\mathbf x,R))$
and $1/s'-1=-1/s$, and sum over the bounded number
of balls in the covering.

By \eqref{eq:Riesz_kernel_representation}, Minkowski's inequality,
and \eqref{eq:far_resolvent_gradient_short},
\begin{align*}
&\left(
\int_{\{R/2<d(\mathbf{x},\mathbf{y})\leq R\}}
|K_j(\mathbf{y},\mathbf{x})|^{s'}\,dw(\mathbf{y})
\right)^{1/s'}
\\
&\quad\leq
C_MR\,w(B(\mathbf{x},R))^{-1/s}
\bigl(1+Rm(\mathbf{x})\bigr)^{-\varepsilon}
\int_{\mathbb{R}}
|\tau|^{-1/2}
\bigl(1+|\tau|^{1/2}R\bigr)^{-M}\,d\tau.
\end{align*}
For $M>1$,
\[
\int_{\mathbb{R}}
|\tau|^{-1/2}
\bigl(1+|\tau|^{1/2}R\bigr)^{-M}\,d\tau
=
\frac{C_M}{R}.
\]
Hence,
\[
\begin{split}
&\left(
\int_{\{R/2<d(\mathbf{x},\mathbf{y})\leq R\}}
|K_j(\mathbf{y},\mathbf{x})|^{s'}\,dw(\mathbf{y})
\right)^{1/s'}\leq
Cw(B(\mathbf{x},R))^{-1/s}
\bigl(1+Rm(\mathbf{x})\bigr)^{-\varepsilon}.
\end{split}
\]
Since $Rm(\mathbf{x})=2^\ell$, this proves
\eqref{eq:far_annular_Riesz_kernel}.

Finally, let $\ell\leq0$ and again put $R=2^\ell r$.
Now $R\leq r$, so
\eqref{eq:local_resolvent_difference_annular_short} applies.
Using
\eqref{eq:Riesz_kernel_representation},
\eqref{eq:free_Riesz_kernel_representation}, and Minkowski's
inequality, we obtain
\begin{align*}
&\left(
\int_{\{R/2<d(\mathbf{x},\mathbf{y})\leq R\}}
|K_j(\mathbf{y},\mathbf{x})
-K_j^0(\mathbf{y},\mathbf{x})|^{s'}\,dw(\mathbf{y})
\right)^{1/s'}
\\
&\quad\leq
C_MR
\bigl(Rm(\mathbf{x})\bigr)^{2-\mathbf{N}/q_1}
w(B(\mathbf{x},R))^{-1/s}
\int_{\mathbb{R}}
|\tau|^{-1/2}
\bigl(1+|\tau|^{1/2}R\bigr)^{-M}\,d\tau
\\
&\quad\leq
C\bigl(Rm(\mathbf{x})\bigr)^{2-\mathbf{N}/q_1}
w(B(\mathbf{x},R))^{-1/s}.
\end{align*}
This proves \eqref{eq:local_annular_Riesz_kernel}.
\end{proof}

\section{First-order and mixed estimates}

\subsection{Maximal-function bounds for the kernel terms}

\begin{lemma}\label{lem:sharp_Riesz_maximal_estimates} Let $V$ be $G$-invariant.
Assume $\frac{\mathbf{N}}{2}<q<\mathbf{N}$,  $N>2$, and let $q_1>q$ be supplied by
the self-improvement of the reverse Holder inequality, with
$q_1<\mathbf{N}$ (see Lemma~\ref{lem: self improve RH}). Define $s$ by
\begin{equation*}
    \frac1s=1-\frac1{q_1}+\frac1{\mathbf{N}}.
\end{equation*}
Then, for $r=m(\mathbf{x})^{-1}$ and every compactly supported bounded
function $f$, the following two estimates hold:
\begin{align}
 &\left|\int_{d(\mathbf{x},\mathbf{y})>r}
 K_j(\mathbf{y},\mathbf{x})f(\mathbf{y})\,dw(\mathbf{y})\right|
 \leq C
 \left(\sum_{g\in G}M(|f|^s)(g(\mathbf{x}))\right)^{1/s},
 \label{eq:far_Riesz_missing}\\
 &\left|\int_{d(\mathbf{x},\mathbf{y})<r}
 \big(K_j(\mathbf{y},\mathbf{x})-K_j^0(\mathbf{y},\mathbf{x})\big)
 f(\mathbf{y})\,dw(\mathbf{y})\right|
 \leq C
 \left(\sum_{g\in G}M(|f|^s)(g(\mathbf{x}))\right)^{1/s}.
 \label{eq:local_Riesz_missing}
\end{align}
Here, we recall that $M$ is the Hardy--Littlewood maximal operator associated with $dw$.
\end{lemma}

\begin{proof}
Let $s'$ denote the conjugate exponent of $s$. Then
\[
    \frac1{s'}=\frac1{q_1}-\frac1{\mathbf{N}}.
\]

We first prove \eqref{eq:far_Riesz_missing}. Decompose
$\{\mathbf{y}:d(\mathbf{x},\mathbf{y})>r\}$ into dyadic orbit annuli:
\[
 \{\mathbf{y}:d(\mathbf{x},\mathbf{y})>r\}
 =
 \bigcup_{\ell=1}^{\infty}
 \{\mathbf{y}:2^{\ell-1}r<d(\mathbf{x},\mathbf{y})
                    \leq2^\ell r\}.
\]
By Hölder's inequality and
\eqref{eq:far_annular_Riesz_kernel},
\begin{align*}
 &\int_{\{2^{\ell-1}r<d(\mathbf{x},\mathbf{y})\leq2^\ell r\}}
 |K_j(\mathbf{y},\mathbf{x})f(\mathbf{y})|\,dw(\mathbf{y})
 \\
 &\quad\leq
 C2^{-\ell\varepsilon}
 w(B(\mathbf{x},2^\ell r))^{-1/s}
 \left(
 \int_{\{d(\mathbf{x},\mathbf{y})\leq2^\ell r\}}
 |f(\mathbf{y})|^s\,dw(\mathbf{y})
 \right)^{1/s}.
\end{align*}
Since
\[
 \{\mathbf{y}:d(\mathbf{x},\mathbf{y})\leq2^\ell r\}
 \subseteq
 \bigcup_{g\in G}B(g(\mathbf{x}),2^\ell r),
\]
the $G$-invariance of $dw$ gives
\begin{align*}
 \int_{\{d(\mathbf{x},\mathbf{y})\leq2^\ell r\}}
 |f(\mathbf{y})|^s\,dw(\mathbf{y})
 &\leq
 \sum_{g\in G}
 \int_{B(g(\mathbf{x}),2^\ell r)}
 |f(\mathbf{y})|^s\,dw(\mathbf{y})\leq
 w(B(\mathbf{x},2^\ell r))
 \sum_{g\in G}M(|f|^s)(g(\mathbf{x})).
\end{align*}
Consequently,
\[
\begin{split}
 &\int_{\{2^{\ell-1}r<d(\mathbf{x},\mathbf{y})\leq2^\ell r\}}
 |K_j(\mathbf{y},\mathbf{x})f(\mathbf{y})|\,dw(\mathbf{y})\leq
 C2^{-\ell\varepsilon}
 \left(\sum_{g\in G}M(|f|^s)(g(\mathbf{x}))\right)^{1/s}.
\end{split}
\]
Summing over $\ell\geq1$ proves
\eqref{eq:far_Riesz_missing}.

For the local term, write
\[
 \{\mathbf{y}:0<d(\mathbf{x},\mathbf{y})<r\}
 =
 \bigcup_{\ell=-\infty}^{0}
 \{\mathbf{y}:2^{\ell-1}r<d(\mathbf{x},\mathbf{y})
                    \leq2^\ell r\}.
\]
By Hölder's inequality and
\eqref{eq:local_annular_Riesz_kernel},
\begin{align*}
 &\int_{\{2^{\ell-1}r<d(\mathbf{x},\mathbf{y})\leq2^\ell r\}}
 |K_j(\mathbf{y},\mathbf{x})-K_j^0(\mathbf{y},\mathbf{x})|
 |f(\mathbf{y})|\,dw(\mathbf{y})
 \\
 &\quad\leq
 C2^{\ell(2-\mathbf{N}/q_1)}
 w(B(\mathbf{x},2^\ell r))^{-1/s}
 \left(
 \int_{\{d(\mathbf{x},\mathbf{y})\leq2^\ell r\}}
 |f(\mathbf{y})|^s\,dw(\mathbf{y})
 \right)^{1/s}
 \\
 &\quad\leq
 C2^{\ell(2-\mathbf{N}/q_1)}
 \left(\sum_{g\in G}
 M(|f|^s)(g(\mathbf{x}))\right)^{1/s}.
\end{align*}
Since $q_1>\mathbf{N}/2$, we have $2-\frac{\mathbf{N}}{q_1}>0$. Therefore,
\[
    \sum_{\ell=-\infty}^{0}
    2^{\ell(2-\mathbf{N}/q_1)}<\infty.
\]
Summing over $\ell\leq0$ proves
\eqref{eq:local_Riesz_missing}.
\end{proof}

\subsection{Free maximal truncations and operator decomposition}

\begin{lemma}\label{lem:free_Riesz_maximal_truncation}
Recall that $K_j(\mathbf y,\mathbf x)$ and $K_j^0(\mathbf y,\mathbf x)$
denote the integral kernels of $T_jL^{-1/2}$ and
$T_j(-\Delta_k)^{-1/2}$, respectively (see~\eqref{eq:Riesz_kernel_representation} and~\eqref{eq:free_Riesz_kernel_representation}). For bounded compactly supported $f$, define
\[
 \mathcal R_{j,*}^{0,d}f(\mathbf x)
 =
 \sup_{\varepsilon>0}
 \left|
 \int_{\{\mathbf y:\,d(\mathbf x,\mathbf y)>\varepsilon\}}
 K_j^0(\mathbf y,\mathbf x)f(\mathbf y)\,dw(\mathbf y)
 \right|.
\]
Then $\mathcal R_{j,*}^{0,d}$ extends to a bounded sublinear
operator on $L^p(dw)$ for every $1<p<\infty$:
\begin{equation}\label{eq:free_Riesz_maximal_truncation}
 \|\mathcal R_{j,*}^{0,d}f\|_{L^p(dw)}
 \leq C_p\|f\|_{L^p(dw)}.
\end{equation}
\end{lemma}

\begin{proof}
The Dunkl--Riesz transform $\mathcal R_j^0$ is a
Dunkl--Calder\'on--Zygmund operator in the sense of
\cite{TanHanHanLeeLi}; see also the kernel estimates in
\cite{HanLeeLiWick}. Thus Corollary~3.2 of
\cite{TanHanLi} applies to its maximal truncation with
respect to the orbit distance $d$.

We check the orientation of the kernel. The free heat
kernel is symmetric and satisfies
\[
 T_{j,\mathbf x}h_t(\mathbf x,\mathbf y)
 =
 \frac{y_j-x_j}{2t}h_t(\mathbf x,\mathbf y).
\]
Its integral representation therefore gives, whenever
$d(\mathbf x,\mathbf y)>0$,
\[
 K_j^0(\mathbf y,\mathbf x)
 =-K_j^0(\mathbf x,\mathbf y).
\]
Consequently,
\[
 \mathcal R_{j,*}^{0,d}f(\mathbf x)
 =
 \sup_{\varepsilon>0}
 \left|
 \int_{\{\mathbf y:\,d(\mathbf x,\mathbf y)>\varepsilon\}}
 K_j^0(\mathbf x,\mathbf y)f(\mathbf y)\,dw(\mathbf y)
 \right|.
\]
This is the orbit-distance maximal truncation covered by
\cite[Corollary~3.2]{TanHanLi}. Its $L^p(dw)$ boundedness
proves~\eqref{eq:free_Riesz_maximal_truncation}.
\end{proof}

\begin{lemma}\label{lem:Riesz_operator_decomposition}
Assume that $N>2$, $\mathbf N/2<q_1<\mathbf N$, and that
$V\in{\rm RH}^{q_1}(dw)$ is $G$-invariant.
Let $K_j$ and $K_j^0$ be the kernels of
$T_jL^{-1/2}$ and $T_j(-\Delta_k)^{-1/2}$,
respectively, with the first variable being the output
variable.

For every bounded compactly supported function $f$, the
following identity holds for almost every $\mathbf x$,
where $r=m(\mathbf x)^{-1}$:
\begin{equation}\label{eq:sharp_Riesz_operator_decomposition}
\begin{split}
 (T_jL^{-1/2})^*f(\mathbf x)
 ={}&
 \int_{d(\mathbf x,\mathbf y)>r}
 \overline{K_j(\mathbf y,\mathbf x)}
 f(\mathbf y)\,dw(\mathbf y)
 \\
 &+
 \int_{d(\mathbf x,\mathbf y)<r}
 \overline{(K_j-K_j^0)(\mathbf y,\mathbf x)}
 f(\mathbf y)\,dw(\mathbf y)
 \\
 &+(\mathcal R_j^0)^*f(\mathbf x)
 -
 \int_{d(\mathbf x,\mathbf y)>r}
 \overline{K_j^0(\mathbf y,\mathbf x)}
 f(\mathbf y)\,dw(\mathbf y).
\end{split}
\end{equation}
All three integrals on the right-hand side converge
absolutely.
\end{lemma}

\begin{proof}
The quadratic forms of $L=-\Delta_k+V$ and $-\Delta_k$ satisfy
\[
 \sum_{i=1}^N\|T_i h\|_{L^2(dw)}^2
 \leq \|L^{1/2}h\|_{L^2(dw)}^2,
 \qquad
 \sum_{i=1}^N\|T_i h\|_{L^2(dw)}^2
 =\|(-\Delta_k)^{1/2}h\|_{L^2(dw)}^2.
\]
Applying these inequalities to $h=L^{-1/2}f$ and
$h=(-\Delta_k)^{-1/2}f$, initially on suitable spectral
subspaces, shows that both $T_jL^{-1/2}$ and
$\mathcal R_j^0=T_j(-\Delta_k)^{-1/2}$ extend to bounded
operators on $L^2(dw)$, with norm at most one.
For integers $n\geq2$, define
\[
 S_{j,n}
 =
 \frac1{2\pi}
 \int_{\{n^{-1}<|\tau|<n\}}
 (-i\tau)^{-1/2}T_j(L+i\tau)^{-1}\,d\tau,
\]
and define $S_{j,n}^0$ by replacing $L$ with
$-\Delta_k$. We use the branch
\[
 (-i\tau)^{-1/2}
 =
 |\tau|^{-1/2}
 \exp\left(\frac{i\pi}{4}\operatorname{sgn}\tau\right).
\]
For fixed $n$, these integrals converge in operator norm
on $L^2(dw)$.

To verify convergence as $n\to\infty$, put
\[
 b_n(\lambda)
 =
 \frac{\lambda^{1/2}}{2\pi}
 \int_{\{n^{-1}<|\tau|<n\}}
 \frac{(-i\tau)^{-1/2}}{\lambda+i\tau}\,d\tau,
 \qquad \lambda>0.
\]
Set $b_n(0)=0$. The scalar functional-calculus identity gives
$b_n(\lambda)\to1$ for every $\lambda>0$. Moreover,
\[
 |b_n(\lambda)|
 \leq
 C\lambda^{1/2}
 \int_0^\infty
 \frac{t^{-1/2}}{(\lambda^2+t^2)^{1/2}}\,dt
 \leq C.
\]
Since
\[
 S_{j,n}=T_jL^{-1/2}b_n(L),
 \qquad
 S_{j,n}^0
 =T_j(-\Delta_k)^{-1/2}b_n(-\Delta_k),
\]
the spectral theorem yields
\begin{equation}\label{eq:Riesz_spectral_truncations_L2}
 S_{j,n}^*f\longrightarrow(T_jL^{-1/2})^*f,
 \qquad
 (S_{j,n}^0)^*f
 \longrightarrow(\mathcal R_j^0)^*f
 \quad\text{in }L^2(dw).
\end{equation}
The Riesz transforms are understood to vanish on the
corresponding null spaces.

Let
\[
 K_{j,n}(\mathbf y,\mathbf x)
 =
 \frac1{2\pi}
 \int_{\{n^{-1}<|\tau|<n\}}
 (-i\tau)^{-1/2}
 T_{j,\mathbf y}\Gamma(\mathbf y,\mathbf x,\tau)\,d\tau,
\]
and define $K_{j,n}^0$ analogously using $\Gamma_0$.
These are the kernels of $S_{j,n}$ and $S_{j,n}^0$,
respectively. We first justify their local integrability
across the orbit diagonal.

For almost every fixed $\mathbf x$, the basic kernel
bound (see Proposition~\ref{propo:basic_decay}) gives
\[
 \int_{U(\mathbf x,\rho)}
 |\Gamma(\mathbf y,\mathbf x,\tau)|\,dw(\mathbf y)
 \leq C\rho^2.
\]
Fix $0<\rho<c\,m(\mathbf x)^{-1}$, where $c<1$ is a
sufficiently small fixed constant, and put
\[
 E_\rho(\mathbf x)
 =\{\mathbf y:\rho/2<d(\mathbf x,\mathbf y)\leq\rho\},
 \qquad R=\rho/32.
\]
We can cover $E_\rho(\mathbf x)$ by at most $C_{N,G}$
balls $B(\mathbf a_\nu,R)$ with
$\mathbf a_\nu\in E_\rho(\mathbf x)$. For every such centre,
\[
 U(\mathbf a_\nu,4R)
 \subset
 \{\mathbf z:3\rho/8<d(\mathbf x,\mathbf z)<9\rho/8\}.
\]
In particular, these orbit neighbourhoods do not meet
$\mathcal O(\{\mathbf x\})$, so
$\Gamma(\,\cdot\,,\mathbf x,\tau)$ is a local weak solution
on each of them.

Since $\mathbf a_\nu$ lies within distance $\rho$ of
$g(\mathbf x)$ for some $g\in G$, Lemma~\ref{lem:m_growth}
and the $G$-invariance of $m$ give
$m(\mathbf a_\nu)\simeq m(\mathbf x)$. Hence
$Rm(\mathbf a_\nu)\leq C$, and
\eqref{eq:local_resolvent_gradient_simplified} gives
\[
 \|T_{j,\mathbf y}\Gamma(\mathbf y,\mathbf x,\tau)\|
 _{L^{s'}(B(\mathbf a_\nu,R),dw)}
 \leq
 CR^{-1}w(B(\mathbf a_\nu,R))^{1/s'}
 \operatorname*{ess\,sup}_{\mathbf z\in U(\mathbf a_\nu,4R)}
 |\Gamma(\mathbf z,\mathbf x,\tau)|.
\]
On $U(\mathbf a_\nu,4R)$, the orbit distance
$d(\mathbf x,\mathbf z)$ is comparable to $\rho$.
Proposition~\ref{propo:basic_decay}, the doubling property,
and the $G$-invariance of $dw$ therefore imply
\[
 \operatorname*{ess\,sup}_{\mathbf z\in U(\mathbf a_\nu,4R)}
 |\Gamma(\mathbf z,\mathbf x,\tau)|
 \leq C\frac{\rho^2}{w(B(\mathbf x,\rho))},
 \qquad
 w(B(\mathbf a_\nu,R))
 \simeq w(B(\mathbf x,\rho)).
\]
As $R=\rho/32$ and $1/s+1/s'=1$, we conclude that
\[
 \|T_{j,\mathbf y}\Gamma(\mathbf y,\mathbf x,\tau)\|
 _{L^{s'}(B(\mathbf a_\nu,R),dw)}
 \leq C\rho\,w(B(\mathbf x,\rho))^{-1/s}.
\]
Summing over the bounded number of covering balls yields
\[
 \left\|
 T_{j,\mathbf y}\Gamma(\mathbf y,\mathbf x,\tau)
 \right\|_{L^{s'}(E_\rho(\mathbf x),dw)}
 \leq C\rho\,w(B(\mathbf x,\rho))^{-1/s},
 \qquad
 \frac1{s'}=\frac1{q_1}-\frac1{\mathbf N}.
\]
The constant is independent of $\rho$, $\mathbf x$, and $\tau$.
Hölder's inequality and summation over decreasing
dyadic annuli therefore imply
\[
 \int_{U(\mathbf x,\rho)}
 |T_{j,\mathbf y}\Gamma(\mathbf y,\mathbf x,\tau)|
 \,dw(\mathbf y)
 \leq C\rho,
 \qquad 0<\rho<c\,m(\mathbf x)^{-1}.
\]
The constants are independent of $\tau$.
The analogous estimates for $\Gamma_0$ follow from
the free kernel estimates.

To identify this derivative distributionally across
$\mathcal O(\{\mathbf x\})$, use a $G$-invariant cutoff
vanishing on $U(\mathbf x,\rho)$ and equal to one outside
$U(\mathbf x,2\rho)$, with gradient bounded by $C/\rho$.
For a fixed test function, the additional term in
integration by parts is bounded by
\[
 C\rho^{-1}
 \int_{U(\mathbf x,2\rho)}
 |\Gamma(\mathbf y,\mathbf x,\tau)|\,dw(\mathbf y)
 \leq C\rho.
\]
It tends to zero as $\rho\to0$.
Thus the locally integrable derivative above represents
the distributional Dunkl derivative on the whole space.
Integration over the finite spectral interval then
justifies the kernel representations for
$S_{j,n}$ and $S_{j,n}^0$.

For almost every $\mathbf x$, set
\[
 r=m(\mathbf x)^{-1},
 \qquad
 E_\ell(\mathbf x)
 =
 \{\mathbf y:
 2^{\ell-1}r<d(\mathbf x,\mathbf y)\leq2^\ell r\}.
\]
Fix $M>1$ and put $R=2^\ell r$. By Minkowski's inequality
and \eqref{eq:far_resolvent_gradient_short}, for $\ell\geq1$,
\[
\begin{split}
 &\|K_{j,n}(\cdot,\mathbf x)\|_{L^{s'}(E_\ell(\mathbf x),dw)}
 \\
 &\quad\leq
 C_M R\,w(B(\mathbf x,R))^{-1/s}
 (1+Rm(\mathbf x))^{-\varepsilon}
 \int_{\{n^{-1}<|\tau|<n\}}
 |\tau|^{-1/2}(1+|\tau|^{1/2}R)^{-M}\,d\tau.
\end{split}
\]
Similarly, \eqref{eq:local_resolvent_difference_annular_short}
gives, for $\ell\leq0$,
\[
\begin{split}
 &\|(K_{j,n}-K_{j,n}^0)(\cdot,\mathbf x)\|
 _{L^{s'}(E_\ell(\mathbf x),dw)}
 \\
 &\quad\leq
 C_M R\,(Rm(\mathbf x))^\delta
 w(B(\mathbf x,R))^{-1/s}
 \int_{\{n^{-1}<|\tau|<n\}}
 |\tau|^{-1/2}(1+|\tau|^{1/2}R)^{-M}\,d\tau,
\end{split}
\]
where $\delta=2-\mathbf N/q_1>0$.

The spectral integrals satisfy
\[
 \int_{\{n^{-1}<|\tau|<n\}}
 |\tau|^{-1/2}(1+|\tau|^{1/2}R)^{-M}\,d\tau
 \leq
 \int_{\mathbb R}
 |\tau|^{-1/2}(1+|\tau|^{1/2}R)^{-M}\,d\tau
 =\frac{4}{R(M-1)}.
\]
Thus the factor $R$ cancels. Since $Rm(\mathbf x)=2^\ell$,
we obtain
\[
 \|K_{j,n}(\cdot,\mathbf x)\|_{L^{s'}(E_\ell(\mathbf x),dw)}
 \leq
 C2^{-\ell\varepsilon}
 w(B(\mathbf x,2^\ell r))^{-1/s},
 \qquad \ell\geq1,
\]
and
\[
 \|(K_{j,n}-K_{j,n}^0)(\cdot,\mathbf x)\|
 _{L^{s'}(E_\ell(\mathbf x),dw)}
 \leq
 C2^{\ell\delta}
 w(B(\mathbf x,2^\ell r))^{-1/s},
 \qquad \ell\leq0.
\]
The constants are independent of $n$.
The same integrable majorants show that
$K_{j,n}\to K_j$ on each exterior annulus and
$K_{j,n}-K_{j,n}^0\to K_j-K_j^0$ on each interior annulus
in $L^{s'}(dw)$.
The free derivative estimate
\eqref{eq:free_resolvent_first_derivative} similarly gives
$K_{j,n}^0\to K_j^0$ in $L^{s'}(dw)$ on every fixed
annulus separated from the orbit of $\mathbf x$.
These statements hold for almost every $\mathbf x$,
by Fubini's theorem.

Let $f$ be bounded and compactly supported.
For each $n$, the locally integrable kernel
representations give
\[
\begin{split}
 S_{j,n}^*f(\mathbf x)
 ={}&
 \int_{d(\mathbf x,\mathbf y)>r}
 \overline{K_{j,n}(\mathbf y,\mathbf x)}
 f(\mathbf y)\,dw(\mathbf y)
 \\
 &+
 \int_{d(\mathbf x,\mathbf y)<r}
 \overline{(K_{j,n}-K_{j,n}^0)(\mathbf y,\mathbf x)}
 f(\mathbf y)\,dw(\mathbf y)
 \\
 &+(S_{j,n}^0)^*f(\mathbf x)
 -
 \int_{d(\mathbf x,\mathbf y)>r}
 \overline{K_{j,n}^0(\mathbf y,\mathbf x)}
 f(\mathbf y)\,dw(\mathbf y).
\end{split}
\]
For fixed $\mathbf x$, only finitely many exterior
annuli meet the support of $f$. Hence convergence on
individual annuli handles both exterior integrals.

For the local difference, we repeat the argument used in the
proof of Lemma~\ref{lem:sharp_Riesz_maximal_estimates}, now
using the annular estimates for $K_{j,n}-K_{j,n}^0$ obtained
above, whose constants are independent of $n$.
H\"older's inequality and the definition of the maximal
operator give
\[
\begin{split}
 &\int_{E_\ell(\mathbf x)}
 |K_{j,n}(\mathbf y,\mathbf x)-K_{j,n}^0(\mathbf y,\mathbf x)|
 |f(\mathbf y)|\,dw(\mathbf y)\leq
 C2^{\ell\delta}
 \left(
 \sum_{g\in G}M(|f|^s)(g(\mathbf x))
 \right)^{1/s},
 \qquad \ell\leq0.
\end{split}
\]
Since $\delta>0$, these bounds are summable in $\ell$,
uniformly in $n$. On each fixed annulus, the
$L^{s'}$ convergence of $K_{j,n}-K_{j,n}^0$, together
with H\"older's inequality, gives convergence of the
integrals against $f$. Dominated convergence for the
sum over the annuli therefore yields
\[
\begin{split}
 &\lim_{n\to\infty}
 \int_{d(\mathbf x,\mathbf y)<r}
 \overline{(K_{j,n}-K_{j,n}^0)(\mathbf y,\mathbf x)}
 f(\mathbf y)\,dw(\mathbf y)=
 \int_{d(\mathbf x,\mathbf y)<r}
 \overline{(K_j-K_j^0)(\mathbf y,\mathbf x)}
 f(\mathbf y)\,dw(\mathbf y)
\end{split}
\]
for almost every $\mathbf x$.

Finally, use \eqref{eq:Riesz_spectral_truncations_L2}
and choose a common subsequence along which the two
operator terms converge almost everywhere.
Passing to the limit proves
\eqref{eq:sharp_Riesz_operator_decomposition}.
Absolute convergence follows from the same annular
estimates and the compact support of $f$.
\end{proof}

\subsection{\texorpdfstring{$L^p$}{Lp}-boundedness of the Riesz transforms}

\begin{proposition}\label{prop:sharp_Riesz_reduction}
Assume that $N>2$, $\mathbf N/2<q<\mathbf N$, and that
$V\in{\rm RH}^q(dw)$ is $G$-invariant.
Then, for every
\begin{equation}\label{eq:sharp_Riesz_range}
 1<p<p_0,
 \qquad
 \frac1{p_0}=\frac1q-\frac1{\mathbf N},
\end{equation}
there exists a constant $C_p>0$ such that
\begin{equation}\label{eq:sharp_Riesz_conclusion}
 \|\nabla_k(-\Delta_k+V)^{-1/2}f\|_{L^p(dw)}
 \leq C_p\|f\|_{L^p(dw)}.
\end{equation}
Here the operator is initially defined on $L^2(dw)$
and extended to $L^p(dw)$ by density.
\end{proposition}

\begin{proof}
Recall that $L=-\Delta_k+V$ and
\[
 \mathcal R_j=T_jL^{-1/2},
 \qquad
 \mathcal R_j^0=T_j(-\Delta_k)^{-1/2}.
\]
By Lemma~\ref{lem: self improve RH}, choose
$q<q_1<\mathbf N$ such that $V\in{\rm RH}^{q_1}(dw)$,
and let
\[
 \frac1s=1-\frac1{q_1}+\frac1{\mathbf N}.
\]

Let $f$ be bounded and compactly supported.
Apply the decomposition from
Lemma~\ref{lem:Riesz_operator_decomposition}.
The first two terms are bounded by
Lemma~\ref{lem:sharp_Riesz_maximal_estimates}.
Since $K_j^0$ is real-valued, the last integral is
bounded in absolute value by
$\mathcal R_{j,*}^{0,d}f(\mathbf x)$.
Consequently,
\[
\begin{split}
 |\mathcal R_j^*f(\mathbf x)|
 \leq{}&
 C\left(
 \sum_{g\in G}M(|f|^s)(g(\mathbf x))
 \right)^{1/s}+|(\mathcal R_j^0)^*f(\mathbf x)|
 +\mathcal R_{j,*}^{0,d}f(\mathbf x).
\end{split}
\]

For $p>s$, the $G$-invariance of $dw$ and the
$L^{p/s}(dw)$ boundedness of $M$ give
\[
 \left\|
 \left(
 \sum_{g\in G}M(|f|^s)(g(\,\cdot\,))
 \right)^{1/s}
 \right\|_{L^p(dw)}
 \leq C_p\|f\|_{L^p(dw)}.
\]
The free Dunkl--Riesz transform and its adjoint are
bounded on $L^p(dw)$ for $1<p<\infty$.
Together with
Lemma~\ref{lem:free_Riesz_maximal_truncation},
this yields
\[
 \|\mathcal R_j^*f\|_{L^p(dw)}
 \leq C_p\|f\|_{L^p(dw)},
 \qquad s<p<\infty.
\]
By density, this estimate extends to $L^p(dw)$.

Duality therefore gives the boundedness of
$\mathcal R_j$ on $L^p(dw)$ for $1<p<s'$, where
\[
 \frac1{s'}=\frac1{q_1}-\frac1{\mathbf N}.
\]
Since $q_1>q$, we have $s'>p_0$.
Combining the estimates for $j=1,\ldots,N$ proves
\eqref{eq:sharp_Riesz_conclusion}.
\end{proof}

\subsection{Gradient estimates with a potential factor}

\begin{lemma}\label{lem:weighted_gradient_maximal_estimate}
Assume that $N>2$, $q>\mathbf N/2$, and that
$V\in{\rm RH}^q(dw)$ is $G$-invariant.
Let $q_1>q$ be a finite exponent supplied by the
reverse H\"older self-improvement.

If $\frac{\mathbf{N}}{2}<q_1<\mathbf{N},$
define $s$ by
\begin{equation}\label{eq:s_weighted_gradient_below_dimension}
    \frac1s=1-\frac{3}{2q_1}+\frac1{\mathbf{N}}.
\end{equation}
Then, for every compactly supported bounded function $f$,
\begin{equation}\label{eq:weighted_gradient_missing_below_dimension}
 \left|\int_{\mathbb{R}^N}T_{j,\mathbf{y}}
 \Gamma(\mathbf{x},\mathbf{y},0)V(\mathbf{y})^{1/2}
 f(\mathbf{y})\,dw(\mathbf{y})\right|
 \leq C
 \left(\sum_{g\in G}M(|f|^s)(g(\mathbf{x}))\right)^{1/s}.
\end{equation}

If $q_1\geq\mathbf{N}$, define $s$ by
\begin{equation}\label{eq:s_weighted_gradient_above_dimension}
    \frac1s=1-\frac1{2q_1}.
\end{equation}
Then estimate
\eqref{eq:weighted_gradient_missing_below_dimension} holds with this
value of $s$.
\end{lemma}

\begin{proof}
Fix $\mathbf{x}\in\mathbb{R}^{N}$ and put $r=m(\mathbf{x})^{-1}$.
We first consider the case $\frac{\mathbf{N}}{2}<q_1<\mathbf{N}$.
Define $t$ by
\begin{equation}\label{eq:t_weighted_gradient}
    \frac1t=\frac1{q_1}-\frac1{\mathbf{N}}.
\end{equation}
It follows from
\eqref{eq:s_weighted_gradient_below_dimension} and
\eqref{eq:t_weighted_gradient} that
\begin{equation}\label{eq:three_exponents_weighted_gradient}
    \frac1t+\frac1{2q_1}+\frac1s=1.
\end{equation}

We first record the annular estimate for the gradient of the fundamental
solution. For every integer $\ell\in\mathbb{Z}$, set $R=2^\ell r$. The local resolvent-gradient estimate
\eqref{eq:local_resolvent_gradient}, applied to $u(\mathbf{y})=\Gamma(\mathbf{x},\mathbf{y},0)$,
together with
\eqref{eq:fundamental_tau_decay}, gives, for every $M>0$,
\begin{equation}\label{eq:annular_gradient_fundamental_below_dimension}
\begin{split}
 &\left(
 \int_{\{R/2<d(\mathbf{x},\mathbf{y})\leq R\}}
 |T_{j,\mathbf{y}}\Gamma(\mathbf{x},\mathbf{y},0)|^t
 \,dw(\mathbf{y})
 \right)^{1/t}\leq
 C_M(1+Rm(\mathbf{x}))^{-M}
 \frac{R}{w(B(\mathbf{x},R))^{1-1/t}}.
\end{split}
\end{equation}
Indeed, the orbit annulus on the left-hand side can be covered by a fixed
finite number of orbit neighbourhoods of radius comparable with $R$.
On each such neighbourhood, the point
$\mathbf{x}$ and all its reflected points lie outside the corresponding
larger neighbourhood. Thus
$\Gamma(\mathbf{x},\cdot,0)$ satisfies
\[
    -\Delta_k u+Vu=0
\]
there. 
Applying the reverse H\"older inequality with exponent
$q_1$ on $B(g(\mathbf{x}),2R)$, we obtain
\[
 \|V\|_{L^{q_1}(B(g(\mathbf{x}),2R),dw)}
 \leq
 Cw(B(g(\mathbf{x}),2R))^{1/q_1}
 \frac{1}{w(B(g(\mathbf{x}),2R))}
 \int_{B(g(\mathbf{x}),2R)}V\,dw.
\]
Since $dw$ is $G$-invariant and doubling,
\[
 w(B(g(\mathbf{x}),2R))
 =w(B(\mathbf{x},2R))
 \leq Cw(B(\mathbf{x},R)).
\]
Consequently,
\[
 \|V\|_{L^{q_1}(B(g(\mathbf{x}),2R),dw)}
 \leq
 Cw(B(\mathbf{x},R))^{1/q_1}
 \frac{1}{w(B(g(\mathbf{x}),2R))}
 \int_{B(g(\mathbf{x}),2R)}V\,dw.
\]
Since $V$ is $G$-invariant,
\[
    m(g(\mathbf{x}))=m(\mathbf{x}),
    \qquad
    w(B(g(\mathbf{x}),R))=w(B(\mathbf{x},R)).
\]
The critical-radius estimates give, for some fixed
$\kappa_0>0$,
\[
 R^2\frac{1}{w(B(\mathbf x,2R))}
 \int_{B(\mathbf x,2R)}V\,dw
 \leq C(1+Rm(\mathbf x))^{\kappa_0}.
\]
Thus the potential term in the local gradient estimate
contributes at most the factor
$(1+Rm(\mathbf x))^{\kappa_0}$.
For a prescribed $M>0$, apply
Theorem~\ref{theorem:fundamental_invariant}
with decay exponent $A=M+\kappa_0$.
The resulting factors satisfy
\[
 (1+Rm(\mathbf x))^{\kappa_0}
 (1+Rm(\mathbf x))^{-A}
 =(1+Rm(\mathbf x))^{-M}.
\]
This gives the required decay in
\eqref{eq:annular_gradient_fundamental_below_dimension}.

We decompose $\mathbb{R}^{N}$ into dyadic orbit annuli:
\[
 \mathbb{R}^{N}
 =
 \bigcup_{\ell\in\mathbb{Z}}
 \{\mathbf{y}:2^{\ell-1}r<d(\mathbf{x},\mathbf{y})
                         \leq2^\ell r\},
\]
up to a set of measure zero. By
\eqref{eq:three_exponents_weighted_gradient} and Hölder's inequality,
\begin{align}
 &\int_{\{R/2<d(\mathbf{x},\mathbf{y})\leq R\}}
 |T_{j,\mathbf{y}}\Gamma(\mathbf{x},\mathbf{y},0)|
 V(\mathbf{y})^{1/2}|f(\mathbf{y})|\,dw(\mathbf{y})
 \nonumber\\
 &\quad\leq
 \left(
 \int_{\{R/2<d(\mathbf{x},\mathbf{y})\leq R\}}
 |T_{j,\mathbf{y}}\Gamma(\mathbf{x},\mathbf{y},0)|^t
 \,dw(\mathbf{y})
 \right)^{1/t}
 \nonumber\\
 &\qquad\times
 \left(
 \int_{\{d(\mathbf{x},\mathbf{y})\leq R\}}
 V(\mathbf{y})^{q_1}\,dw(\mathbf{y})
 \right)^{1/(2q_1)}
 \left(
 \int_{\{d(\mathbf{x},\mathbf{y})\leq R\}}
 |f(\mathbf{y})|^s\,dw(\mathbf{y})
 \right)^{1/s}.
 \label{eq:three_Holder_weighted_gradient}
\end{align}

Since
\[
 \{\mathbf{y}:d(\mathbf{x},\mathbf{y})\leq R\}
 \subseteq
 \bigcup_{g\in G}B(g(\mathbf{x}),R),
\]
the reverse Hölder inequality and the $G$-invariance of $V$ and $dw$
give
\begin{equation}\label{eq:potential_factor_weighted_gradient}
\begin{split}
 &\left(
 \int_{\{d(\mathbf{x},\mathbf{y})\leq R\}}
 V(\mathbf{y})^{q_1}\,dw(\mathbf{y})
 \right)^{1/(2q_1)}
 \leq
 Cw(B(\mathbf{x},R))^{1/(2q_1)}
 \left(
 \frac{1}{w(B(\mathbf{x},R))}
 \int_{B(\mathbf{x},R)}V(\mathbf{y})\,dw(\mathbf{y})
 \right)^{1/2}.
\end{split}
\end{equation}
Similarly,
\begin{equation}\label{eq:function_factor_weighted_gradient}
\begin{split}
 &\left(
 \int_{\{d(\mathbf{x},\mathbf{y})\leq R\}}
 |f(\mathbf{y})|^s\,dw(\mathbf{y})
 \right)^{1/s}\leq
 w(B(\mathbf{x},R))^{1/s}
 \left(
 \sum_{g\in G}M(|f|^s)(g(\mathbf{x}))
 \right)^{1/s}.
\end{split}
\end{equation}
Combining
\eqref{eq:annular_gradient_fundamental_below_dimension}--
\eqref{eq:function_factor_weighted_gradient}, and using
\eqref{eq:three_exponents_weighted_gradient}, we obtain
\begin{align}
 &\int_{\{R/2<d(\mathbf{x},\mathbf{y})\leq R\}}
 |T_{j,\mathbf{y}}\Gamma(\mathbf{x},\mathbf{y},0)|
 V(\mathbf{y})^{1/2}|f(\mathbf{y})|\,dw(\mathbf{y})
 \nonumber\\
 &\quad\leq
 C_M(1+Rm(\mathbf{x}))^{-M}
 R
 \left(
 \frac{1}{w(B(\mathbf{x},R))}
 \int_{B(\mathbf{x},R)}V(\mathbf{y})\,dw(\mathbf{y})
 \right)^{1/2}
 \left(
 \sum_{g\in G}M(|f|^s)(g(\mathbf{x}))
 \right)^{1/s}.
 \label{eq:annular_weighted_gradient_reduced}
\end{align}
Notice that all powers of $w(B(\mathbf{x},R))$ cancel because
\[
 -1+\frac1t+\frac1{2q_1}+\frac1s=0.
\]

Suppose first that $\ell\leq0$, so that $R\leq r$. Then
\[
 \frac{R^2}{w(B(\mathbf{x},R))}
 \int_{B(\mathbf{x},R)}V(\mathbf{y})\,dw(\mathbf{y})
 \leq
 C\left(\frac{R}{r}\right)^{2-\mathbf{N}/q_1}.
\]
Consequently,
\begin{equation}\label{eq:small_annuli_weighted_gradient}
 R
 \left(
 \frac{1}{w(B(\mathbf{x},R))}
 \int_{B(\mathbf{x},R)}V\,dw
 \right)^{1/2}
 \leq
 C\left(\frac{R}{r}\right)^{1-\mathbf{N}/(2q_1)}
 =
 C2^{\ell(1-\mathbf{N}/(2q_1))}.
\end{equation}
Since $q_1>\mathbf{N}/2$, $1-{\mathbf{N}}/{2q_1}>0$, 
and hence the sum over $\ell\leq0$ is convergent.

For $\ell>0$, the doubling property of the measure $V\,dw$ and the
properties of the critical function give, for some $\kappa_0>0$,
\[
 \frac{R^2}{w(B(\mathbf{x},R))}
 \int_{B(\mathbf{x},R)}V\,dw
 \leq
 C\left(\frac{R}{r}\right)^{\kappa_0}.
\]
Therefore,
\begin{equation}\label{eq:large_annuli_weighted_gradient}
\begin{split}
 &(1+Rm(\mathbf{x}))^{-M}
 R
 \left(
 \frac{1}{w(B(\mathbf{x},R))}
 \int_{B(\mathbf{x},R)}V\,dw
 \right)^{1/2}\leq
 C(1+2^\ell)^{-M}2^{\ell\kappa_0/2}.
\end{split}
\end{equation}
Choosing $M>\kappa_0/2$ makes the sum over $\ell>0$ convergent.
Summing \eqref{eq:annular_weighted_gradient_reduced} over
$\ell\in\mathbb{Z}$ proves
\eqref{eq:weighted_gradient_missing_below_dimension} when
$q_1<\mathbf{N}$.

We now consider the case $q_1>\mathbf N$.
First, for $\mathbf y\in U(\mathbf a,R)$, the volume
growth estimate and a dyadic decomposition give
\[
\begin{split}
 &\int_{U(\mathbf a,2R)}
 \left(
 \frac{d(\mathbf y,\mathbf z)}
 {w(B(\mathbf y,d(\mathbf y,\mathbf z)))}
 \right)^{q_1'}\,dw(\mathbf z) \leq
 C R^{q_1}{}'
 w(B(\mathbf a,R))^{1-q_1'}
 \sum_{\ell=0}^{\infty}
 2^{-\ell(q_1'-\mathbf N(q_1'-1))}.
\end{split}
\]
Here the factor $R^{q_1}{}'$ is understood as $R^{q_1'}$.
The series converges because
\[
 q_1'-\mathbf N(q_1'-1)
 =\frac{q_1-\mathbf N}{q_1-1}>0.
\]
Consequently, H\"older's inequality yields
\[
\begin{split}
 &\operatorname*{ess\,sup}_{\mathbf y\in U(\mathbf a,R)}
 \int_{U(\mathbf a,2R)}
 \frac{d(\mathbf y,\mathbf z)}
 {w(B(\mathbf y,d(\mathbf y,\mathbf z)))}
 |F(\mathbf z)|\,dw(\mathbf z)\leq
 CR\,w(B(\mathbf a,R))^{-1/q_1}
 \|F\|_{L^{q_1}(U(\mathbf a,2R),dw)}.
\end{split}
\]
Using this bound for the potential term in the proof of
\eqref{eq:local_resolvent_gradient}, together with
the free local gradient estimate, gives
\[
\begin{split}
 &\operatorname*{ess\,sup}_{U(\mathbf a,R)}|T_j u|\leq
 C\left(
 R^{-1}
 +R\,w(B(\mathbf a,R))^{-1/q_1}
 \|V\|_{L^{q_1}(U(\mathbf a,2R),dw)}
 \right)
 \operatorname*{ess\,sup}_{U(\mathbf a,4R)}|u|.
\end{split}
\]
The reverse H\"older inequality and the critical-radius
estimates therefore imply, for some $\kappa_0>0$,
\[
 \operatorname*{ess\,sup}_{U(\mathbf a,R)}|T_j u|
 \leq
 CR^{-1}(1+Rm(\mathbf a))^{\kappa_0}
 \operatorname*{ess\,sup}_{U(\mathbf a,4R)}|u|.
\]
Apply this estimate to
$u=\Gamma(\mathbf x,\cdot,0)$ on a finite covering
of each orbit annulus by sufficiently small orbit
neighbourhoods. Using the symmetry of $\Gamma$,
Lemma~\ref{lem:m_growth}, and
Theorem~\ref{theorem:fundamental_invariant}, with
a sufficiently large decay exponent, gives, for every $M>0$,
\begin{equation}\label{eq:pointwise_gradient_fundamental_above_dimension}
 |T_{j,\mathbf y}\Gamma(\mathbf x,\mathbf y,0)|
 \leq
 C_M(1+d(\mathbf x,\mathbf y)m(\mathbf x))^{-M}
 \frac{d(\mathbf x,\mathbf y)}
 {w(B(\mathbf x,d(\mathbf x,\mathbf y)))}
\end{equation}
for almost every pair $(\mathbf x,\mathbf y)$ off
the orbit diagonal.

For $ R=2^\ell r$, estimate \eqref{eq:pointwise_gradient_fundamental_above_dimension}
implies
\[
 \operatorname*{ess\,sup}_{\{R/2<d(\mathbf{x},\mathbf{y})\leq R\}}
 |T_{j,\mathbf{y}}\Gamma(\mathbf{x},\mathbf{y},0)|
 \leq
 C_M(1+Rm(\mathbf{x}))^{-M}
 \frac{R}{w(B(\mathbf{x},R))}.
\]
By \eqref{eq:s_weighted_gradient_above_dimension}, $
    1/{(2q_1)}+1/s=1$. 
Hölder's inequality therefore gives
\begin{align*}
 &\int_{\{R/2<d(\mathbf{x},\mathbf{y})\leq R\}}
 |T_{j,\mathbf{y}}\Gamma(\mathbf{x},\mathbf{y},0)|
 V(\mathbf{y})^{1/2}|f(\mathbf{y})|\,dw(\mathbf{y})
 \\
 &\quad\leq
 C_M(1+Rm(\mathbf{x}))^{-M}
 R
 \left(
 \frac{1}{w(B(\mathbf{x},R))}
 \int_{B(\mathbf{x},R)}V\,dw
 \right)^{1/2}\times
 \left(
 \sum_{g\in G}M(|f|^s)(g(\mathbf{x}))
 \right)^{1/s}.
\end{align*}
The sum over $\ell\in\mathbb{Z}$ converges by exactly the arguments in
\eqref{eq:small_annuli_weighted_gradient} and
\eqref{eq:large_annuli_weighted_gradient}. This proves the result when
$q_1>\mathbf{N}$.

Finally, suppose that $q_1=\mathbf{N}$. The reverse Hölder
self-improvement gives an exponent $q_2>\mathbf{N}$ such that $V\in{\rm RH}^{q_2}(dw)$.
Applying the preceding argument with $q_2$ gives the estimate with
\[
    \frac1{s_2}=1-\frac1{2q_2}.
\]
Since $q_2>q_1$, we have $s_2<s$, where
\[
    \frac1s=1-\frac1{2q_1}.
\]
Jensen's inequality for the maximal operator gives
\[
    M(|f|^{s_2})(\mathbf{z})^{1/s_2}
    \leq
    |G|^{1/s_2-1/s}M(|f|^s)(\mathbf{z})^{1/s}.
\]
Thus the estimate obtained with $s_2$ implies the required estimate
with $s$. This completes the proof.
\end{proof}

\begin{proposition}\label{prop:weighted_gradient_reduction}
Assume that $N>2$, $q>\mathbf{N}/2$, and that
$V\in{\rm RH}^q(dw)$ is $G$-invariant. Put
\[
 p_1=
 \begin{cases}
 \displaystyle
 \left(\frac{3}{2q}-\frac1{\mathbf{N}}\right)^{-1},
 & \mathbf{N}/2<q<\mathbf{N},\\[2mm]
 2q, & q\geq\mathbf{N}.
 \end{cases}
\]
Then, for every $1<p<p_1$, the operator
$V^{1/2}\nabla_k(-\Delta_k+V)^{-1}$,
initially defined through the Green potential on
$C_c^\infty(\mathbb R^N)$, extends to a bounded operator
from $L^p(dw)$ to $L^p(dw;\mathbb C^N)$. In particular,
\begin{equation}\label{eq:weighted_gradient_conclusion}
 \|V^{1/2}\nabla_k(-\Delta_k+V)^{-1}f\|_{L^p(dw)}
 \leq C_p\|f\|_{L^p(dw)}.
\end{equation}
\end{proposition}

\begin{proof}
Write $L=-\Delta_k+V$, and choose a finite exponent
$q_1>q$ such that $V\in{\rm RH}^{q_1}(dw)$.
Let $s$ be the exponent in
Lemma~\ref{lem:weighted_gradient_maximal_estimate}.

For bounded compactly supported $f$, define
\[
 \mathcal A_jf(\mathbf x)
 =
 \int_{\mathbb R^N}
 T_{j,\mathbf y}\Gamma(\mathbf x,\mathbf y,0)
 V(\mathbf y)^{1/2}f(\mathbf y)\,dw(\mathbf y).
\]
The absolute-integral estimate proved in
Lemma~\ref{lem:weighted_gradient_maximal_estimate}
ensures that this integral is well defined.
For every $s<a<\infty$, the boundedness of $M$ on
$L^{a/s}(dw)$ and the $G$-invariance of $dw$ give
\[
 \|\mathcal A_jf\|_{L^a(dw)}
 \leq C_a\|f\|_{L^a(dw)}.
\]
Thus $\mathcal A_j$ extends to a bounded operator
on $L^a(dw)$.

We identify its adjoint using the Green kernel.
For bounded compactly supported $f,h$, the absolute
estimate above and H\"older's inequality justify
Fubini's theorem. The reality and symmetry of
$\Gamma(\mathbf x,\mathbf y,0)$, together with its
distributional derivative representation, give
\[
 \langle \mathcal A_jf,h\rangle
 =
 \langle f,V^{1/2}T_jL^{-1}h\rangle.
\]
Here $L^{-1}h$ is understood as the Green potential.
Consequently, by duality,
\[
 \|V^{1/2}T_jL^{-1}h\|_{L^{a'}(dw)}
 \leq C_a\|h\|_{L^{a'}(dw)},
 \qquad s<a<\infty.
\]
The operator therefore extends by density to
$L^p(dw)$ for every $1<p<s'$.

If $q_1<\mathbf N$, then
\[
 \frac1{s'}=\frac{3}{2q_1}-\frac1{\mathbf N}
 <\frac{3}{2q}-\frac1{\mathbf N}
 =\frac1{p_1}.
\]
If $q<\mathbf N\leq q_1$, then
$p_1<2\mathbf N\leq2q_1=s'$.
Finally, if $q\geq\mathbf N$, then
$s'=2q_1>2q=p_1$.
Thus $s'>p_1$ in every case.
Combining the estimates for $j=1,\ldots,N$
proves \eqref{eq:weighted_gradient_conclusion}.
\end{proof}

\subsection{Mixed operators and estimates for equations}

\begin{theorem}\label{thm:conditional_sharp_mixed_estimates}
Assume that $N>2$, $q>\mathbf N/2$, and that
$V\in{\rm RH}^q(dw)$ is $G$-invariant. Put
\[
 p_0=
 \begin{cases}
 \displaystyle\frac{\mathbf Nq}{\mathbf N-q},
       & \mathbf N/2<q<\mathbf N,\\[2mm]
 \infty, & q\geq\mathbf N,
 \end{cases}
\]
with the convention $p_0'=1$ when $p_0=\infty$.
For a vector field $\mathbf f=(f_1,\ldots,f_N)$, write
\[
 \operatorname{div}_k\mathbf f=\sum_{j=1}^N T_jf_j.
\]
Then
\begin{align}
 \|(-\Delta_k+V)^{-1/2}
       \operatorname{div}_k\mathbf f\|_{L^p(dw)}
 &\leq C_p\|\mathbf f\|_{L^p(dw)},
 &&p_0'<p<\infty,
 \label{eq:conditional_adjoint_Riesz}\\
 \|\nabla_k(-\Delta_k+V)^{-1}
       \operatorname{div}_k\mathbf f\|_{L^p(dw)}
 &\leq C_p\|\mathbf f\|_{L^p(dw)},
 &&p_0'<p<p_0,
 \label{eq:conditional_mixed_Riesz}\\
 \|V^{1/2}(-\Delta_k+V)^{-1}
       \operatorname{div}_k\mathbf f\|_{L^p(dw)}
 &\leq C_p\|\mathbf f\|_{L^p(dw)},
 &&p_0'<p<2q.
 \label{eq:conditional_mixed_potential}
\end{align}
The operators are understood as bounded extensions
of the corresponding compositions defined below.
\end{theorem}

\begin{proof}
If $\mathbf N/2<q<\mathbf N$,
Proposition~\ref{prop:sharp_Riesz_reduction} gives
the boundedness of $\nabla_kL^{-1/2}$ on $L^p(dw)$
for $1<p<p_0$.

If $q\geq\mathbf N$, then
$V\in{\rm RH}^{q_*}(dw)$ for every
$\mathbf N/2<q_*<\mathbf N$.
Given any $1<p<\infty$, choose $q_*$ sufficiently
close to $\mathbf N$ that
\[
 p<\frac{\mathbf Nq_*}{\mathbf N-q_*}.
\]
The same proposition therefore gives the boundedness
of $\nabla_kL^{-1/2}$ on $L^p(dw)$ for every
$1<p<\infty$.

Since the adjoint of $\nabla_kL^{-1/2}$ is
$-L^{-1/2}\operatorname{div}_k$, duality proves
\eqref{eq:conditional_adjoint_Riesz}.
The composition
\[
 (\nabla_kL^{-1/2})
 (L^{-1/2}\operatorname{div}_k)
\]
is consequently bounded on $L^p(dw)$ whenever
$p_0'<p<p_0$, which proves
\eqref{eq:conditional_mixed_Riesz}.

Finally, Theorem~\ref{thm:fractional_potential_boundedness}
gives the boundedness of $V^{1/2}L^{-1/2}$ for
$1<p<2q$. Hence
\[
 (V^{1/2}L^{-1/2})
 (L^{-1/2}\operatorname{div}_k)
\]
is bounded for $p_0'<p<2q$, proving
\eqref{eq:conditional_mixed_potential}.
The mixed operators in the statement are understood
through these compositions.
\end{proof}

\end{document}